\documentclass[10pt]{article}
\usepackage{amssymb,amsmath,mathdots,amsthm}
\usepackage{epsfig,graphicx}
\usepackage{fullpage}
\pdfoutput=1

\newcommand{\eq}[1]{(\ref{#1})}
\newcommand{\la}{\lambda}

\newtheorem{definition}{Definition}
\newtheorem{theorem}{Theorem}
\newtheorem{proposition}{Proposition}
\newtheorem{corollary}{Corollary}
\newcommand{\CC}{\mathbb{C}}
\newcommand{\II}{\mathbb{I}}

\newcommand{\RR}{\mathbb{R}}

\newcommand{\LL}{\mathcal L}

\newcommand{\MM}{\mathcal M}
\newcommand{\BB}{\mathcal B}
\newcommand{\VV}{\mathcal V}
\newcommand{\UU}{\mathcal U}
\newcommand{\DD}{\mathcal D}

\newcommand{\spanv}{\mathop{\mathrm{span}}} 
\newcommand{\gKer}{\mathop{\mathrm{gKer}}}
\newcommand{\Ker}{\mathop{\mathrm{Ker}}} 

\newcommand{\Ran}{\mathop{\mathrm{Ran}}}
\renewcommand{\Re}{\mathop{\rm Re}\nolimits}
\renewcommand{\Im}{\mathop{\rm Im}\nolimits}
\newcommand{\sign}{\mathop{\rm sign}\nolimits}

\newcommand{\sspan}{\spanv}
\newcommand{\ml}{l\kern-0.035cm\char39\kern-0.03cm}
\newcommand{\Haragus}{H{\u a}r{\u a}gu\c{s}}

\newcounter{example}
\newenvironment{example}[1][]{\refstepcounter{example}\par\medskip\noindent%
 $\vartriangleleft$  \textit{Example~\theexample. #1} \rmfamily}{$\vartriangleright$ \medskip}

\begin{document}
\title{Graphical Krein Signature Theory and Evans-Krein Functions}

\author{Richard Koll{\'a}r \\
Department of Applied Mathematics and Statistics\\
Faculty of Mathematics, Physics and Informatics \\
Comenius University \\
Mlynsk{\'a} dolina,  Bratislava, Slovakia \\
E-mail: {\tt kollar@fmph.uniba.sk}\\
\\
Peter D. Miller\\
Department of Mathematics \\ 
University of Michigan \\
530 Church Street, Ann Arbor, MI 48109--1043, U.S.A. \\
E-mail: {\tt millerpd@umich.edu}
}

\maketitle

\abstract{%
Two concepts, evidently very different in nature, have proved to be useful in analytical 
and numerical studies 
of spectral stability in nonlinear wave theory:  (i) the Krein signature of an eigenvalue, 
a quantity usually defined in terms of the relative orientation of certain subspaces that is 
capable of detecting the structural instability of imaginary eigenvalues and hence their potential for 
moving into the right half-plane leading to dynamical instability under perturbation of the system, and (ii) the Evans function, 
an analytic function detecting the location of eigenvalues.  
One might expect these two concepts to be related, but unfortunately examples demonstrate that 
there is no way in general to deduce the Krein signature of an eigenvalue from the Evans function, 
for example by studying  derivatives of the latter.  

The purpose of this paper is to recall and popularize a simple graphical interpretation of the Krein 
signature well-known in the spectral theory of polynomial operator pencils.  Once established, this interpretation 
avoids altogether the need to view the Krein signature in terms of root subspaces and their relation 
to indefinite quadratic forms.  To demonstrate the utility of this graphical interpretation of the Krein 
signature, we use it to define a simple generalization of the Evans function ---  the Evans-Krein function --- 
that allows the calculation of Krein signatures in a way that is easy to incorporate into existing Evans 
function evaluation codes at virtually no additional computational cost. 
The graphical interpretation of the Krein signature also enables us to give elegant 
proofs of index theorems for linearized Hamiltonians in the finite dimensional setting: 
a general result implying as a corollary
the generalized Vakhitov-Kolokolov criterion (or Grillakis-Shatah-Strauss criterion)
and a count of  real eigenvalues 
for linearized Hamiltonian systems in canonical form.  These applications demonstrate how the simple graphical nature of 
the Krein signature may be  easily exploited.}
%


\section{Introduction}
This paper concerns  relations among several concepts that are commonly regarded as useful
in the analysis of (generalized) eigenvalue problems such as those that occur in the stability theory of nonlinear waves:  
\emph{Krein signatures} of eigenvalues, \emph{Evans functions}, and
\emph{index theorems} (also known as \emph{inertia laws} or \emph{eigenvalue counts}) 
governing the spectra of pairs of related operators.  The most elementary setting in which many of these notions appear is 
the stability analysis of equilibria for finite-dimensional Hamiltonian systems, and we take this opportunity right at the beginning 
of the paper to introduce the key ideas in this simple setting, including a beautiful and simple graphical method of analysis.

\subsection{A graphical method for linearized Hamiltonians\label{s:Example}}
The most important features of the linearized dynamics of a finite dimensional 
Hamiltonian system close to a critical point are determined by the values of $\nu$ in the 
\emph{spectrum} $\sigma(J\!L)$ of the (nonselfadjoint) problem
\begin{equation}
J\!L u = \nu u\, , \qquad \qquad L^{\ast}: = \overline{L^\mathsf{T}}= L, \ J^{\ast} =  -J.
\label{JL}
\end{equation} 
Here $J$ is an invertible skew-Hermitian matrix and $L$ a Hermitian matrix 
of the same dimension (both over the complex numbers $\mathbb{C}$).%
\footnote{On notation:  we use $\overline{\nu}$ to denote the complex conjugate of a complex number $\nu$,  
while $A^{\ast}$  denotes the conjugate transpose of $A$, or more generally when an inner product is understood, 
the adjoint operator of $A$.  We use $(u,v)$ to denote an inner product on vectors $u$ and $v$, linear in $u$ 
and conjugate-linear in $v$.}
The conditions on 
$J$ require that the dimension of the space be even and hence $L$ and $J$ have dimension $2n \times 2n$.  
Indeed, a Hamiltonian system linearized about an equilibrium
takes the form
\begin{equation}
\frac{dy}{dt} = J\!L y,
\label{dydtexamples}
\end{equation}
where $y$ denotes the displacement from equilibrium.  The spectral problem \eqref{JL} arises by looking for solutions 
growing as $e^{\nu t}$ by
making the substitution $y(t)=e^{\nu t}u$, and we say that \eqref{dydtexamples} is \emph{spectrally stable} if $\sigma(J\!L)$ 
consists of only purely imaginary  numbers.  The points $\nu\in\sigma(J\!L)$ for which $\Re\{\nu\}\neq 0$ are called 
the \emph{unstable spectrum}
\footnote{Although the points $\nu$ with $\Re\{\nu\} < 0$ correspond to decaying 
solutions of \eq{dydtexamples}, due to the basic Hamiltonian symmetry of \eq{JL} they always go hand-in-hand with 
points $-\overline{\nu}$ in the right half-plane that are associated with exponentially growing solutions of \eq{dydtexamples}, 
explaining why they are also included in the unstable spectrum.} 
of \eqref{dydtexamples}. The \emph{linearized energy} associated with the equation \eq{dydtexamples} 
is simply the quadratic form $E[u]:=(Lu,u)$, 
and the fundamental conservation law corresponding to \eq{dydtexamples} is that $dE[y]/dt=0$ on all solutions $y=y(t)$.

With the use of information on the spectrum $\sigma(L)\subset\mathbb{R}$ of $L$ our goal is to characterize the spectrum 
$\sigma(J\!L)\subset\mathbb{C}$ of $J\!L$ and 
in particular to determine (i) the part of $\sigma(J\!L)$  in the open right half of the complex plane (i.e., the unstable spectrum)
and (ii) the potential for purely imaginary points in $\sigma(J\!L)$ to collide on the imaginary axis under 
suitable perturbations of the matrices resulting in bifurcations of  Hamiltonian-Hopf type \cite{MacKay,Meiss} 
in which points of $\sigma(J\!L)$ leave the imaginary axis and further destabilize the system.

The key idea is to relate the purely imaginary spectrum of \eq{JL} to invertibility of the selfadjoint (Hermitian)
linear pencil $\LL(\la)$ (see Section~\ref{s:pencils} for a proper definition of operator pencils) defined by
\begin{equation}
\LL = \LL(\la) := L-\la K, \qquad \qquad K := (iJ)^{-1}=K^*, \quad \la := i \nu \in\mathbb{R}.
\label{JLpencil}
\end{equation}
Indeed, upon multiplying \eq{JL} by $-iJ^{-1}$ one sees that \eq{JL} is equivalent to the generalized spectral problem 
\begin{equation}
L u = \la Ku\, .
\label{LKeq}
\end{equation}
The graphical method \cite{BinVol1996,GLR,Grillakis1988,KLjub} is based on the observation that the imaginary spectrum of \eq{JL}, 
or equivalently the real spectrum $\lambda$ of \eq{LKeq}, can be found by studying the spectrum of the eigenvalue 
pencil $\LL(\la)$ and the way it depends on $\lambda\in\mathbb{R}$, i.e., by solving the selfadjoint eigenvalue problem 
\begin{equation}
\LL(\la)u(\la) =(L - \la K) u(\la) = \mu(\la) u(\la)\, 
\label{gpencil}
\end{equation}
parametrized by $\la\in\mathbb{R}$ (see Fig.\ \ref{figintro}).
\begin{figure}[htp]
\centering
\includegraphics{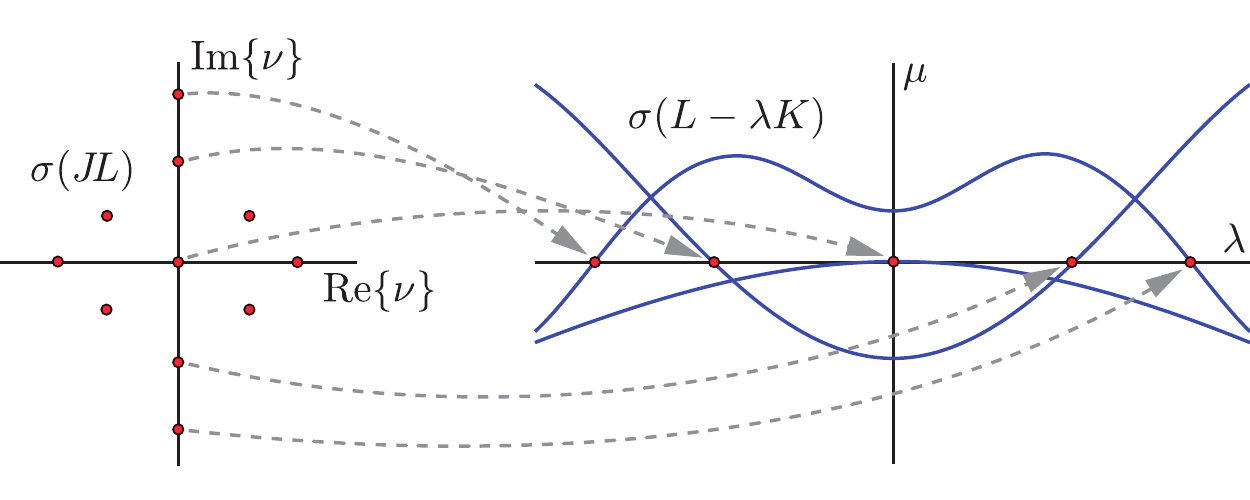}
\caption[]{The key to the graphical method lies in the correspondence of purely imaginary points $\nu$ in the spectrum of $J\!L$ 
with zeros in the spectrum of the pencil $\LL(\la) = L-\la K$, where $K:=(iJ)^{-1}$.}
\label{figintro}
\end{figure}
Clearly, $\nu=-i\lambda$ is a purely imaginary number in $\sigma(J\!L)$ if and only if $\lambda\in\mathbb{R}$ is a value for which $\LL(\la)$ 
has a nontrivial kernel, that is, $0\in\sigma(\LL(\la))$. It is also easy to see from the right-hand plot in Fig.\ \ref{figintro} 
that the particular values $\la$ 
for which $0 \in \sigma(\LL(\la))$ correspond in a one-to-one fashion to intercepts of curves
$\mu=\mu(\la)\in\sigma(\LL(\la))$ with the axis $\mu = 0$. This identification also holds true for spectrum of higher multiplicity 
(for details see Theorem~\ref{multiplicities}). 

To make a clear distinction between the two types of spectrum that are related by the above simple argument we now introduce the 
following terminology. The points $\la$ in $i\sigma(J\!L)$ will be called \emph{characteristic values} of $\LL$ with corresponding invariant 
\emph{root spaces} spanned by \emph{generalized characteristic vectors} or \emph{root vectors}.  Therefore a purely imaginary point 
in $\sigma(J\!L)$ corresponds to a real characteristic value of $\LL$, but the characteristic values of $\LL$ are generally complex
(see Definition~\ref{def:char_matrix}). 
On the other hand, given an arbitrary real number $\lambda$, the points $\mu$ in $\sigma(\LL(\la))$ will be called 
\emph{eigenvalues} of $\LL(\la)$ with corresponding \emph{eigenspaces} spanned by (genuine, due to selfadjointness as 
$\lambda\in\mathbb{R}$)  \emph{eigenvectors}.    When we consider how $\mu=\mu(\lambda)$ depends on $\lambda\in\mathbb{R}$ 
we will frequently call $\mu(\lambda)$ an \emph{eigenvalue branch}.

\subsection{Use of the graphical method to find unstable spectrum}
Purely imaginary points in $\sigma(J\!L)$ correspond in a one-to-one fashion to real characteristic values of $\LL$ simply by rotation 
of the complex plane by $90^\circ$.  The most obvious advantage of the graphical method is that the presence and location of 
real characteristic values of $\LL$ can be easily read off from plots of  the eigenvalue branches $\mu(\la)$ of $\LL(\la)$ using 
the following elementary observations. Selfadjointness of $L$ and skewadjointness of $J$ imply that 
$\sigma(J\!L)$ consists of pairs $(\nu, -\overline{\nu})$, and hence the characteristic values of $\LL$ come in complex-conjugate pairs.  
Indeed, if $u\in\mathbb{C}^{2n}$ is a (right) root vector of $J\!L$ corresponding to $\nu\in\sigma(J\!L)$, 
then $(J^{-1}u)^{\ast}$ is a left root vector of $J\!L$ corresponding to $-\overline{\nu}\in\sigma(J\!L)$.
Therefore the $J\!L$-invariant subspace spanned by root spaces corresponding to non-real characteristic values of $\LL$  
is even-dimensional.  Since the base space $X=\mathbb{C}^{2n}$ is also even-dimensional, the number 
of real characteristic values of $\LL$ (counting multiplicities) is even, and consequently the total number of intercepts 
of eigenvalue branches $\mu(\la)$ with $\mu = 0$ (again, counting multiplicities) must be even.

If furthermore $J$ and $L$ are real matrices (see Definition~\ref{Def:real} for a more general definition of reality in less obvious contexts) 
then $\sigma(J\!L)$ also consists of pairs $(\nu,\overline{\nu})$ and we say that \eq{JL} has \emph{full Hamiltonian symmetry}.  
In such a case, $\sigma(\LL(\la))$ is necessarily symmetric with respect to the vertical ($\lambda=0$) axis, so the plots of 
eigenvalue branches $\mu$ as functions of $\lambda$ are left-right symmetric.  This is the case illustrated in Fig.\ \ref{figintro}.

Finally, the dimension of the problem limits  the maximal number of intercepts  
of all  branches $\mu(\la)$ with the axis $\mu = 0$ to $2n$. This is also true for intercepts with every horizontal axis $\mu = \mu_0$,
as one can consider the problem  $J(L-\mu_0 I) u = \nu u$. 
With this basic approach in hand, we now turn to several elementary examples.

\begin{example}[Definite Hamiltonians.]\label{ex:posdef}
It is well-known that if $L$ is a definite (positive or negative) matrix, then $\sigma(J\!L)$ is purely imaginary and nonzero. 
Indeed, if $u\in \mathbb{C}^{2n}$ is a root vector of $J\!L$ corresponding to $\nu\in\sigma(J\!L)$, then $Lu=\nu J^{-1}u$, 
so taking the Euclidean inner product with $u$ gives
$$
0 \neq (Lu, u) = \nu (J^{-1}u, u),
$$
and hence neither $\nu$ nor $(J^{-1}u,u)$ can be zero.  Moreover, $(Lu,u)\in\RR$ and 
$(J^{-1}u, u) \in i\RR$, and hence $\nu=(Lu,u)/(J^{-1}u,u)$ is a purely imaginary nonzero number.  

This simple fact can also be deduced from a plot of the eigenvalue branches $\mu(\lambda)$ of
the selfadjoint pencil $\LL(\la)$.  Let us assume without loss of generality that $L$ is positive definite.  We need just three facts:
\begin{itemize}
\item The $2n$ eigenvalue branches $\mu(\lambda)$ may be taken to be (we say this only because ambiguity can arise in defining 
the branches if $\LL(\la)$ has a nonsimple, but necessarily semisimple, eigenvalue for some $\lambda\in\mathbb{R}$) continuous functions 
of $\lambda$.  In fact, they may be chosen to be \emph{analytic} functions of $\lambda$, although we do not require 
smoothness of any sort here.
\item The $2n$ eigenvalue branches $\mu(\lambda)$ are all positive at $\lambda=0$ since
$\LL(0)=L$.
\item By simple perturbation theory, $\sigma(\LL(\la))=-\lambda\sigma(K)+O(1)$ as $|\lambda|\to\infty$.  
Since $K$ is Hermitian and invertible, it has
$m\le 2n$ strictly positive eigenvalues and $2n-m$ strictly negative eigenvalues.  
Hence there are exactly $m$ eigenvalue branches $\mu(\lambda)$ 
tending to $-\infty$ as $\lambda\to +\infty$, and exactly $2n-m$ eigenvalue branches $\mu(\lambda)$ 
tending to $-\infty$ as $\lambda\to -\infty$.
\end{itemize}
That all $2n$ characteristic values of $\LL$ are nonzero real numbers, and hence $\sigma(J\!L)\subset i\RR$ making the system 
\eqref{dydtexamples} spectrally stable, therefore appears as a simple consequence of 
the intermediate value theorem;  $m$ branches $\mu(\lambda)$ necessarily cross $\mu=0$ for $\lambda>0$ and $2n-m$ branches 
cross for $\lambda<0$.  Since $m+(2n-m)=2n$ exhausts the dimension of $X=\mathbb{C}^{2n}$, all characteristic values 
have been obtained in this way. This approach provides the additional information that exactly
$m$ of the points in $\sigma(J\!L)$ are negative imaginary numbers.  Note that in the case of full Hamiltonian symmetry, $m=n$.

To illustrate this phenomenon, the branches $\mu(\lambda)$ corresponding to the following
specific choices of $L$ positive definite and $J\!L$ having full Hamiltonian symmetry:
\[
L=\frac{1}{2}\begin{pmatrix}1 & 0 & 0 & 0\\0 & 2 & 0 & 0\\0 & 0 & 3 & 0\\0 & 0 & 0 & 4\end{pmatrix},\quad
J=\begin{pmatrix}0 & 0 & 1 & 0\\0 & 0 & 0 & 1\\-1 & 0 & 0 & 0\\0 & -1 & 0 & 0\end{pmatrix},
\]
are plotted in the left-hand panel of Fig.\ \ref{fig1}.
\begin{figure}[htp]
\centering
\includegraphics{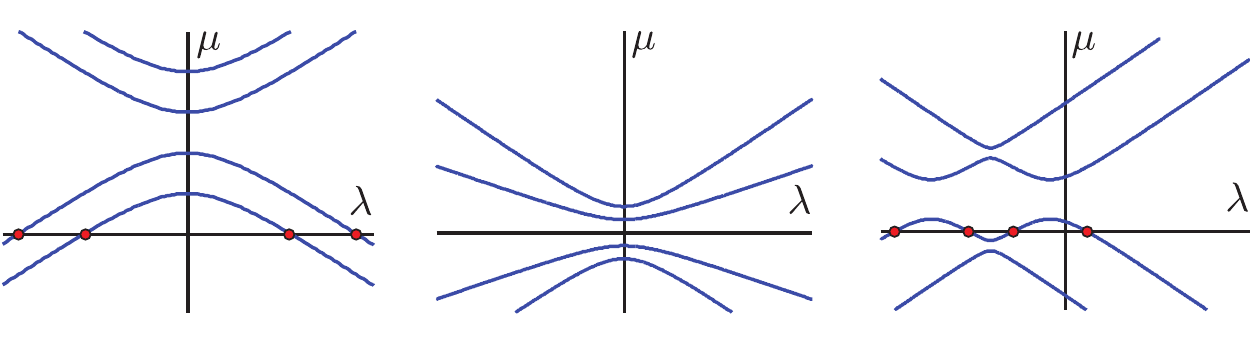}
\caption[]{%
The four eigenvalue branches $\mu(\la)$ of $\LL(\la)$ in Examples 1--3. 
The real (characteristic) values $\lambda$ for which eigenvalues $\mu(\lambda)$ of $\LL(\la)$ intersect 
the horizontal line $\mu = 0$ and that  
correspond to purely imaginary points $\nu=-i\lambda\in\sigma(J\!L)$ are indicated with red dots.}
\label{fig1}
\end{figure}
\end{example}

\begin{example}[Indefinite Hamiltonians and instability.]
\label{ex:indefinite}
If $L$ is indefinite, then $\sigma(J\!L)$ is not automatically confined to the imaginary axis. As an illustration of this phenomenon, 
consider the matrices
\[
L=\begin{pmatrix}-1&0&0&0\\0&2&0&0\\0&0&1&0\\0&0&0&-2\end{pmatrix},\quad
J=\begin{pmatrix}0&0&2&0\\0&0&0&1\\-2&0&0&0\\0&-1&0&0\end{pmatrix},
\]
which again provide full Hamiltonian symmetry, but now $L$ has two positive and two negative eigenvalues.  The corresponding eigenvalue 
branches $\mu(\lambda)$ are plotted against $\lambda\in\RR$ in the middle panel of Fig.\ \ref{fig1}.  
Obviously $\sigma(\LL(\lambda))$ is bounded away from zero as $\lambda\in\RR$ varies, and hence there are 
no real characteristic values of $\LL$, implying that $\sigma(J\!L)\cap i\RR=\emptyset$.  
The full Hamiltonian symmetry in this example implies that $\sigma(J\!L)$ is organized in quadruplets.
Thus either $\sigma(J\!L)$ consists of two positive and two negative real points, or $\sigma(J\!L)$ consists of 
a single quadruplet of non-real and non-imaginary points. Hence all of the spectrum is unstable and the graphical 
method easily predicts spectral instability in this case.
\end{example}

\begin{example}[Stability with indefinite Hamiltonians.] \label{ex:indef-nonreal}
If $J$ and $L$ are not both real the full Hamiltonian symmetry and hence the left-right symmetry of the union 
of eigenvalue branches $\mu(\lambda)$ is broken. In the right-hand panel of Fig.\ \ref{fig1}  one such example is shown 
corresponding to
\begin{equation}
L = \frac{1}{2} 
 \begin{pmatrix} 2 & -8 & 5 & 0\\ -8 & -3 & 0 & 0\\ 5 & 0 & 9 & -8 \\ 0 & 0 & -8 & 13  
\end{pmatrix}, \quad 
J = -i  \begin{pmatrix} 0 & 1 & 0 & 0\\ 1 & 0 & 0 & 0\\ 0 & 0 & 0 & 1 \\ 0 & 0 & 1 & 0  
\end{pmatrix}.
\label{ex3}
\end{equation}
Here we see  
one eigenvalue branch $\mu(\la)$ of $\LL(\la)$ intersecting $\mu = 0$ transversely at four nonzero locations.
Since $\dim(X)=4$, this implies $\sigma(J\!L) \subset i\RR\setminus\{0\}$, and the graphical method predicts spectral stability.
Note that $L$ is indefinite in this case showing that indefiniteness of $L$ does not imply the existence of unstable spectrum.  
That is, while definiteness of $L$ implies stability, the converse is false.   
\end{example}

\subsection{Use of the graphical method to obtain Krein signatures}
These three examples have illustrated the use of the graphical method to count purely imaginary points in $\sigma(J\!L)$, 
which are encoded as zero intercepts of the eigenvalue curves $\mu=\mu(\la)$.  We now wish to dispel any impression that the location 
of the intercepts might be the only information contained in plots like those in the right-hand panel of Fig.\ \ref{figintro} 
and in Fig.\ \ref{fig1}. 
This requires that we introduce briefly the notion of the Krein signature of a purely imaginary point $\nu\in\sigma(J\!L)\cap i\mathbb{R}$.

Krein signature theory \cite{Arnold,ArnoldAvez,Krein1950, Krein1951} allows one to understand aspects of the dynamics 
of Hamiltonian flows near equilibria (see \cite{Meiss} for a recent review).  Equilibria that are local extrema 
of the associated linearized energy are energetically stable, and this is the situation described in Example~1 above.   
However, even in cases when an equilibrium is not a local extremum, the energetic argument can still predict stability
if the linearized energy is definite on an invariant subspace; the fact is that many Hamiltonian systems of physical interest 
have constants of motion independent of the energy (e.g. momentum) and this means that effectively the linearized 
dynamics should be considered on a proper invariant subspace of the linearized phase space corresponding 
to holding the constants of motion fixed. Therefore, in such a situation only certain subspaces are relevant 
for the linearized energy quadratic form, and definiteness of the form can be recovered under an appropriate symplectic 
reduction. Of course, for \eqref{dydtexamples}, the invariant subspaces are simply the root spaces of $J\!L$, and for 
the particular case of genuine characteristic vectors $u$ corresponding to $\nu\in\sigma(J\!L)$, it is easy to see how 
the definiteness of the linearized energy relates to stability.  Indeed, if $J\!Lu=\nu u$, then for invertible $J$ 
we have $Lu=\nu J^{-1}u=i\nu Ku$, and by taking the inner product with $u$ one obtains the identity $(Lu,u) = i \nu (Ku, u)$.   Selfadjointness of $L$ and $K$ then implies
\begin{equation}
0 = \Re\{\nu\}(Ku,u) \quad \text{and}\quad (Lu,u)=-\Im\{\nu\}(Ku,u).
\label{eq:Real-Im}
\end{equation}
If $\nu$ lies in the unstable spectrum, then $\Re\{\nu\}\neq 0$ and the first equation requires that $(Ku,u)=0$ which from the second
equation implies $E[u]=(Lu,u)=0$, so the linearized energy is indefinite on the subspace.  On the other hand, if $\nu$ is purely imaginary, 
then the first equation is trivially satisfied
but the second gives no information about $(Lu,u)$.  This calculation suggests that $(Lu,u)$, or equivalently $(Ku,u)$, carries nontrivial 
information when $\nu\in\sigma(J\!L)\cap i\mathbb{R}$.  We will define%
\footnote{\label{f:Ks}Unfortunately, it is equally common in the field for the Krein signature 
to be defined as $\mathrm{sign}(Lu,u)$, and while the latter is more obviously connected to the linearized energy, 
our definition is essentially equivalent according to \eqref{eq:Real-Im} and provides a more direct generalization beyond
the context of linear pencils considered here.} the \emph{Krein signature} of a simple purely imaginary 
point $\nu\in \sigma(J\!L)$ (or equivalently of the purely real characteristic value $\lambda=i\nu$ of $\LL$) as follows:
\begin{equation}
\kappa(\la):=-\sign(Ku,u),\quad \text{where $\la\in \mathbb{R}$, $J\!Lu=-i\la u$, $K=(iJ)^{-1}$, and $u\neq 0$}.
\label{simpleKreinexamples}
\end{equation}
For simple $\nu\in i\RR$, $(Ku,u)$ is necessarily nonzero%
\footnote{The argument is as follows:  since $\nu=-i\la$ is a simple point of $\sigma(J\!L)$, the space $X$ can be decomposed as 
$X=\mathrm{span}\{u\}\oplus X_c$ where $X_c$ is the complementary subspace invariant under $J\!L$, and 
$J\!L+i\la$ is invertible on $X_c$.  Therefore, $v\in X_c$ can be represented in the form $v=(J\!L+i\la)w$ 
for some $w\in X_c$, and it follows that for all $v\in X_c$ we have
\[
(J^{-1}u,v)=(J^{-1}u,(J\!L+i\la)w)=-((LJ+i\overline{\la})J^{-1}u,w)=-(J^{-1}(J\!L+i\overline{\la})u,w).
\]
Since $\la\in \mathbb{R}$ and $J\!Lu=-i\la u$ we see that $(J^{-1}u,\cdot)$ vanishes on $X_c$.
But since $J$ is invertible this form cannot vanish on all of $X$ and hence it must be definite on $X\ominus X_c$ 
implying that $(J^{-1}u,u)\neq 0$.  See \cite{KP} for more details.} 
by the invertibility of $J$.

Consider  a  simple point $\nu=-i\la$ in $\sigma(J\!L)\cap i\RR$.  If  the matrices $J$ and $L$ are subjected to 
sufficiently small admissible perturbations, then (i) $\nu$ remains purely imaginary and simple and (ii) $\kappa(\la)$ remains constant; 
as an integer-valued continuous function (of $J$ and $L$), the only way a Krein signature $\kappa(\la)$ can change under perturbation 
is if $\nu=-i\la$  collides with another nonzero purely imaginary eigenvalue.
From this point of view, one of the most important properties of the Krein signature is that it captures the susceptibility of a 
point $\nu\in\sigma(J\!L)\cap i\RR$ to  \emph{Hamiltonian-Hopf bifurcation}  \cite{ArnoldAvez,vdM,YS} (see also MacKay \cite{MacKay} 
for a geometric interpretation of the Krein signature within this context) in which two simple purely imaginary points of $\sigma(J\!L)$ 
collide under  and leave the imaginary axis.  Indeed, for bifurcation to occur, it is necessary that the colliding points have
opposite Krein signatures.  In fact, this condition is also sufficient, in the sense that if satisfied there
exists an admissible deformation of $J\!L$ that causes the Hamiltonian-Hopf bifurcation to occur.
On the other hand, imaginary points of $\sigma(J\!L)$ with the same Krein signatures cannot leave 
the imaginary axis  even if they collide under perturbation.

The definition \eqref{simpleKreinexamples} makes the Krein signature $\kappa(\la)$ appear as either 
an algebraic quantity (computed via inner products),
or possibly a geometric quantity (measuring relative orientation of root spaces with respect to the positive and 
negative definite subspaces of the linearized energy form).  
We would now like to emphasize a third interpretation
of the formula, related to the graphical method for linearized Hamiltonians introduced above.
Indeed, from the point of view of Krein signature theory, the main advantage of the reformulation of \eq{JL} as \eq{gpencil} 
is that if $\nu_0$ is a purely imaginary simple point in $\sigma(J\!L)$ corresponding
to the intersection of an eigenvalue branch $\mu(\la)$ with $\mu=0$ at the real characteristic value $\la=\la_0=i\nu_0$ of $\LL$,
then the Krein signature $\kappa(\la_0)$ turns out to have a simple interpretation as the  sign of the
slope of the branch at the intersection point:
\begin{equation}
\kappa(\la_0) = \sign \left[\frac{d \mu}{d \la}(\la_0) \right]\, . 
\label{simpleKrein2examples}
\end{equation}
To prove  \eq{simpleKrein2examples}, 
one differentiates  \eq{gpencil} with respect to $\lambda$ at $\la = \la_0$ and $\mu = 0$, obtaining the equation 
$(L-\la_0K)u'(\la_0)=(\mu'(\la_0)+K)u$. Taking the inner product  with the characteristic vector $u=u(\la_0)$ 
satisfying $(L-\la_0K)u=0$ and using selfadjointness of $L-\la_0K$ for $\la_0\in\mathbb{R}$ gives
\begin{equation}
\mu'(\la_0) (u, u) = -(Ku, u).
\label{derivKreinexamples}
\end{equation}
Since $(u,u)>0$, the expression \eq{simpleKrein2examples} immediately follows from the definition \eq{simpleKreinexamples}.
This shows that not only are the locations of the intercepts of the curves $\mu=\mu(\la)$ important, but it is also useful 
to observe the way the curves cross the $\mu=0$ axis.

Therefore, without computing any inner products at all, we can read off the Krein signatures
of the imaginary spectrum of $J\!L$ in Examples~\ref{ex:posdef} and \ref{ex:indef-nonreal} above simply from looking at the 
diagrams in the corresponding panels of Fig.\ \ref{fig1}.  For Example~\ref{ex:posdef}, the signatures of the negative characteristic values 
are $\kappa=1$ while those of the positive characteristic values are $\kappa=-1$, and the only possibility for the system 
to become structurally unstable to Hamiltonian-Hopf bifurcation would be for a pair of characteristic values to collide 
at $\la=0$ (forced by full Hamiltonian symmetry), and this clearly requires $L$ to become indefinite under perturbation.  
On the other hand,  Example~\ref{ex:indef-nonreal} represents a somewhat more structurally unstable case; 
in order of increasing $\la$ the signatures are $\kappa=1,-1,1,-1$.  
Therefore all of the pairs of adjacent real characteristic values are susceptible to Hamiltonian-Hopf bifurcation.
We remind the reader that this simple graphical identification of the Krein signatures and potential Hamiltonian-Hopf 
bifurcations is due to our choice to define the signature $\kappa$ as $-\sign (Ku,u)$ as opposed to $\sign (Lu,u)$ 
(see also footnote \ref{f:Ks} above). The quantity $(Lu,u)$ entails an additional change of sign every time a purely imaginary 
point in $\sigma(J\!L)$ crosses the origin \cite{KP}, and as a result if one defines Krein signatures using $(Lu,u)$ one has to 
treat the potential Hamiltonian-Hopf bifurcation at the origin as a kind of special case.  Indeed, as $L$ is positive definite in
Example~\ref{ex:posdef}, the quantity $(Lu,u)$ is positive although a Hamiltonian-Hopf bifurcation is indeed possible at the origin.

The formula \eqref{simpleKrein2examples} yields perhaps the easiest proof that
simple real characteristic values of the same Krein signature cannot undergo Hamiltonian-Hopf bifurcation even should they collide; 
locally one has two branches with the same direction of slope, and the persistent reality of the roots as the branches evolve 
can be seen as a consequence of the Intermediate Value Theorem.  

\subsection{Generalizations.  Organization of the paper}
The notion of Krein signature of real characteristic values has become increasingly important 
in the recent literature in the subjects of nonlinear waves, oscillation theory, and integrable systems
\cite{BSS, BronskiJohnson, BJK, ChagPel, KapHar,Kap,Kirillov, KollarH,KP,VP}. In many of these
applications, the situation is more general than the one we have considered so far.  One 
direction in which the theory can be usefully generalized is to replace the matrix $\LL=L-\lambda K$ with 
a more general Hermitian matrix function of a real variable $\lambda$ resulting in a \emph{matrix pencil} 
that is generally nonlinear in $\lambda$.  Another desirable generalization is to be able to work in infinite-dimensional 
spaces where the operators involved are, say, differential operators as might occur in wave dynamics problems. 
The basic graphical method described above can also be applied in these more general settings.

The main applications we have in mind are for linearized Hamiltonian systems of the sort that arise in spectral stability 
analysis of nonlinear waves, i.e., linear eigenvalue problems.  However
many of the ideas used here can be traced back to a nearly disjoint but well-developed body of
literature concerning nonlinear eigenvalue pencils and matrix polynomials \cite{GLRmatrix,GLR,Markus}. 
In particular, the pioneering works of Binding and co-workers \cite{BinBrown1988,BinVol1996} (see also references therein) 
use a graphical method in the context of Sturm-Liouville problems to detect Krein signatures and other properties of spectra of  
operator pencils,  yielding results similar to those presented here. The thought to connect aspects of this theory to problems in 
stability of nonlinear waves appears to have come up quite recently, although similar ideas were already 
used  in \cite{KLjub}. In \S\ref{s:pencils} we review the basic theory of matrix and operator pencils, which lays the groundwork 
for both the generalization to nonlinear dependence on the characteristic value and the generalization to infinite-dimensional spaces.  
Then, in \S\ref{s:Krein} we give a precise definition of Krein signature along the lines of \eq{simpleKreinexamples} and 
show how also for operator pencils there is a way to deduce the Krein signature from the way 
that an eigenvalue curve passes through $\mu=0$, a procedure that
is a direct generalization of the alternate formula \eqref{simpleKrein2examples}.

In \S\ref{s:Evans} we consider the problem of relating Krein signatures to a
common tool used to detect eigenvalues, the so-called \emph{Evans function} 
\cite{AGJ,Evans4, PegoWeinstein}. 
Unfortunately, attempts to deduce the Krein signature of eigenvalues from properties of the 
Evans function itself are easily seen to be inconclusive at best.  
However, the graphical (or perhaps topological) interpretation of the Krein signature as in \eqref{simpleKrein2examples} 
suggests a simple way to modify the traditional definition of the Evans function in such a way that the Krein signatures are all captured.  
This modification is even more striking when one realizes that 
the Evans function itself is based on a (different) topological 
concept \cite{AGJ}, a Mel'nikov-like transversal intersection between stable and unstable manifolds. 
We call this modification of the Evans function
the \emph{Evans-Krein function}, and we describe it also in \S\ref{s:Evans}.  
The main idea is that, while in the linearization of Hamiltonian systems the Evans function restricted to the imaginary 
$\nu$-axis characterizes the \emph{product} of individual algebraic root factors,  the Evans-Krein 
function is able to separate these factors with the help of the additional parameter $\mu$. 
The use of the Evans-Krein function therefore allows these different root factor 
branches to be traced numerically, without significant changes to existing Evans function evaluating codes. 

In \S\ref{s:Counts} we extend the kind of simple arguments used to determine spectral stability in
Examples \ref{ex:posdef}--\ref{ex:indef-nonreal} above to give short and essentially topological proofs of some of the well-known
index theorems for nonselfadjoint spectral problems  that were originally proven by very different, 
algebraic methods \cite{Grillakis1988,Jones1988,Kap,KKS,Pel}.  To keep the exposition as simple as possible, we
present our new proofs in the finite-dimensional setting.  The added value of the graphical approach is that it makes 
the new proofs easy to visualize, and hence to remember and generalize. 
We conclude in \S\ref{s:Discussion} with a brief discussion of related open problems. 
For the readers convenience, two of the longer and more technical proofs of results from the theory of operator pencils are given in full detail in the Appendix.

Our paper features many illustrative examples.  Readers trying to understand the subject for the first time may find 
it useful to pay special attention to these.
\subsection{Acknowledgments}
Richard Koll\'ar was  supported by National Science Foundation under grant DMS-0705563
and by the European Commission Marie Curie International Reintegration Grant 239429. 
Peter D. Miller was supported by National Science Foundation under grants DMS-0807653 and DMS-1206131.
The authors would also like to thank Paul Binding for comments that clarified our notation and
Oleg Kirillov for pointing us to important references.

\section{Matrix and Operator Pencils\label{s:pencils}}
\subsection{Basic terminology and examples}
In the literature the terms \emph{operator pencil} or \emph{operator family} frequently refer to the same type of object: 
a collection of linear operators depending on a complex parameter lying in an open set $S\subset\CC$, 
that is, a map $\LL=\LL(\lambda)$ from $\lambda\in S$ into a suitable 
class of linear operators from one Banach space, $X$, into another, $Y$.  
Perhaps the most common type of pencil is a so-called \emph{polynomial pencil} for which $\LL(\lambda)$ is simply a polynomial 
in $\lambda$ with coefficients that are fixed linear operators.  (The sub-case of a linear pencil has already 
been introduced in Section~\ref{s:Example}.)  
If $X$ and $Y$ are finite-dimensional Banach spaces, we have the special case of a \emph{matrix pencil}.  The coefficients of 
a polynomial matrix pencil $\LL(\la)$ can (by choice of bases of $X$ and $Y$) be represented as constant matrices of the same dimensions.  
We will only consider the case in which $X$ and $Y$ have the same dimension, in which 
case the coefficient matrices of a polynomial matrix pencil are all square. Specializing in a different direction, if  $X=Y$ is 
a (self-dual) Hilbert space, a pencil (operator or matrix) is said to be \emph{selfadjoint} if $S=\overline{S}$ and
$\LL(\overline{\la})=\LL(\la)^*$. With the choice of an appropriate orthonormal basis, the coefficients of a  
selfadjoint polynomial matrix pencil all become Hermitian matrices.
 
An operator pencil consisting of \emph{bounded} linear operators $\LL(\la)$ on 
a fixed (possibly infinite-dimensional) Banach space $X$ is said to be \emph{holomorphic} at $\lambda_0\in S$
if there is a neighborhood $D$ of $\lambda_0$ in which $\LL(\la)$ can be expressed  
as a convergent (in operator norm) power series in $\lambda$ (see also \cite[pp.~55--56]{Markus}).  
If $D$ contains a real open interval $I \subset \RR\cap S$, we say that $\LL$ 
is \emph{continuously differentiable} on $I$ if the restriction of $\LL$ to $I$ is continuously norm-differentiable. 
For holomorphic (continuously differentiable) matrix pencils, the individual matrix entries of the matrix $\LL(\la)$ 
are all holomorphic functions of $\la$ near $\lambda_0$ (continuously differentiable functions on $I$).

The theory of operator pencils, and matrix pencils in particular, is well-developed in the literature. 
The review articles by Tisseur and Meerbergen \cite{tm} and  Mehrmann and Voss \cite{mv}
are  general references that give a numerous applications of operator pencils and survey suitable numerical methods for their study.
Polynomial pencils are particularly well-understood and have an extensive spectral theory \cite{GLRmatrix,GLR,Markus}.
While most of the theory is concerned with matrix pencils
(sometimes also called \emph{$\lambda$-matrices} or \emph{gyroscopic systems})
some results have been obtained for general operator pencils (see \cite{KLa, KLb, Markus} and references therein).
The spectral theory of $\lambda$-matrices under various conditions was developed in detail by Lancaster {\it et al.}
\cite{hlr,Lancaster, lt, lz}; see also \cite{lmz} for the related subject of perturbation theory of analytic matrix functions. 

In order to better motivate the theory of linear and nonlinear operator pencils, we first
give some concrete examples of how they arise naturally in several applications.

\begin{example}[Spectral problems in inverse-scattering theory for integrable partial differential equations.]
\label{ex:integrable}
As is well known, some of the most interesting and important nonlinear partial differential equations of mathematical physics 
including the \emph{modified focusing nonlinear Schr\"odinger equation}
\begin{equation}
i\phi_t + \frac{1}{2}\phi_{xx} + |\phi|^2\phi + i\alpha (|\phi|^2\phi)_x=0,\quad\alpha\in\mathbb{R}
\label{eq:MNLS}
\end{equation}
governing the envelope $\phi(x,t)$ of (ultrashort, for $\alpha\neq 0$) pulses propagating in weakly nonlinear and strongly 
anomalously dispersive optical fibers, and the \emph{sine-Gordon equation}
\begin{equation}
u_{tt}-u_{xx}+\sin(u)=0
\label{eq:SG}
\end{equation}
arising in the analysis of denaturation of DNA molecules and the modeling of superconducting Josephson junctions, 
are \emph{completely integrable systems}.  One of the key implications of complete integrability is the existence 
of an \emph{inverse-scattering transform} for solving the Cauchy initial-value problem for $x\in\mathbb{R}$ in which 
\eqref{eq:MNLS} is given the complex-valued Schwartz-class initial condition $\phi(x,0)=\phi_0(x)$ and \eqref{eq:SG} 
is given the real-valued Schwartz-class\footnote{More generally, one only assumes that $\sin(f(x))$ is Schwartz class to admit the 
physically interesting possibility of nonzero \emph{topological charge} in the initial data in which the angle $f$ increases by a nonzero
integer multiple of $2\pi$ as $x$ ranges over $\mathbb{R}$.} initial conditions $u(x,0)=f(x)$ and $u_t(x,0)=g(x)$, 
and in each case the solution is desired for $t>0$.  The inverse-scattering transform explicitly associates the initial 
data for each of these equations to a certain auxiliary linear equation involving a spectral parameter, the spectrum of 
which essentially encodes all of the key properties of the solution for $t>0$.  For the modified focusing nonlinear Schr\"odinger 
equation \eqref{eq:MNLS} with $\alpha\neq 0$ the auxiliary linear equation for the inverse-scattering transform is the so-called 
WKI spectral problem due to Wadati {\it et al.} \cite{WadatiKI79} with spectral parameter $\xi$ and vector-valued unknown $v = v(x)$:
\begin{equation}
\mathcal{L}_\mathrm{WKI}(\xi)v=0,\quad\text{where}\quad \mathcal{L}_\mathrm{WKI}(\xi):=\xi^2W_2 +\xi W_1+W_0,
\end{equation}
involving a quadratic pencil $\mathcal{L}_\mathrm{WKI}(\xi)$, with coefficients
being the linear operators
\begin{equation}
W_2:=2i\sigma_3,\quad W_1:=
-2i\alpha\begin{pmatrix}0 & \phi_0\\\overline{\phi}_0 & 0\end{pmatrix},\quad
W_0:=\alpha\frac{d}{dx} -\frac{1}{2}i\sigma_3.
\label{eq:WKI}
\end{equation}
Here $\sigma_3$ is a Pauli spin matrix\footnote{The Pauli spin matrices are:
\[
\sigma_1:=\begin{pmatrix}0 & 1\\1 & 0\end{pmatrix},\quad\sigma_2:=\begin{pmatrix}0 & -i\\
i & 0\end{pmatrix},\quad\sigma_3:=\begin{pmatrix}1 & 0 \\ 0 & -1\end{pmatrix}.
\]
}.
In the special case $\alpha=0$, one has to use a different auxiliary linear equation known as the Zakharov-Shabat problem \cite{ZakharovS72}
which is commonly written in the form
\begin{equation}
\frac{dv}{dx}=-i\xi\sigma_3v +\begin{pmatrix}0 & \phi_0\\-\overline{\phi}_0 & 0\end{pmatrix}v
\end{equation}
which by multiplication on the left by $i\sigma_3$ is obviously equivalent to a usual eigenvalue problem 
for an operator that is non-selfadjoint with respect to the $L^2(\mathbb{R})$ inner product (augmented in the obvious 
Euclidean way for vector functions $v$).  However, note that if one sets $\lambda=i\xi$, $v(x)=e^{i\pi\sigma_3/4}w(x)$, 
and multiplies through on the left by $i\sigma_1e^{-i\pi\sigma_3/4}$, 
the Zakharov-Shabat problem is also equivalent to the equation
\begin{equation}
\mathcal{L}_\mathrm{ZS}(\la)w=0,\quad\text{where}\quad \mathcal{L}_\mathrm{ZS}(\la):=\la Z_1+Z_0
\end{equation}
involving a linear pencil $\mathcal{L}_\mathrm{ZS}(\la)$ with coefficients
\begin{equation}
Z_1:=\sigma_2,\quad Z_0:=i\sigma_1\frac{d}{dx} -\begin{pmatrix}\overline{\phi}_0 &0\\ 0 & \phi_0\end{pmatrix}.
\end{equation}
On the other hand, the initial-value problem for the sine-Gordon equation \eqref{eq:SG} is solved by means of 
the Faddeev-Takhtajan problem:
\begin{equation}
\mathcal{L}_\mathrm{FT}(\xi)v=0,\quad\mathcal{L}_\mathrm{FT}:=\xi F_1+F_0+\xi^{-1}F_{-1},
\label{eq:FT}
\end{equation} 
a problem for a rational operator pencil with coefficients
\begin{equation}
F_0:=4i\frac{d}{dx}-g\sigma_2\quad F_{\pm 1}:=\sin(\tfrac{1}{2}f)\sigma_1\mp
\cos(\tfrac{1}{2}f)\sigma_3.
\end{equation}
In all of these cases, the values of $\xi$ for which there exists a nontrivial $L^2(\RR)$ solution $v$
parametrize the amplitudes and velocities of the \emph{soliton} components of the solution of the corresponding 
(nonlinear) initial-value problem.  The solitons are localized coherent structures that appear in the long-time limit, and 
as such it is important to have accurate methods to determine the location of any such discrete spectrum. Bounds for the discrete
spectrum of the WKI spectral problem can be found in \cite{DiFrancoM08}, and very sharp results that under some natural qualitative
conditions on the initial data confine the discrete spectrum to the imaginary axis for the Zakharov-Shabat problem and the unit 
circle for the Faddeev-Takhtajan problem were found by Klaus and Shaw \cite{KlausShaw} and by Bronski and Johnson
\cite{BronskiJohnson} respectively.  In particular, the techniques used in \cite{BronskiJohnson,KlausShaw} can be interpreted 
in terms of Krein signatures.  Indeed, one of the key conditions required in \cite{KlausShaw} is that the initial condition 
$\phi_0(x)$ is a real-valued function, which makes $\LL_\mathrm{ZS}(\la)$ a (formally) \emph{selfadjoint} linear pencil. The significance of
selfadjointness will become clear later.
\end{example}

\begin{example}[Hydrodynamic stability.]
Consider a steady plane-parallel 
shear flow of a fluid in a two-dimensional channel with horizontal coordinate $x\in\mathbb{R}$ and vertical coordinate $z\in (z_1, z_2)$,
and let $U=U(z)$ be the  horizontal velocity profile.
If the stream function of such a flow is perturbed by a  Fourier mode of the form $\Phi(z)e^{ik(x-ct)}$ with horizontal 
wavenumber $k\in\mathbb{R}$,  then in the case of an inviscid fluid the dynamical stability of such a flow 
is governed by the \emph{Rayleigh equation} \cite{DR}:
\begin{equation}
\mathcal{L}_\mathrm{R}(c)\Phi=0,\quad \mathcal{L}_\mathrm{R}(c):=c L_1 +L_0,
\label{eq:Rayleigh}
\end{equation}
where $\Phi$ is subjected to the boundary conditions $\Phi(z_1)=\Phi(z_2)=0$, and where
$\mathcal{L}_\mathrm{R}(c)$ is 
a linear operator pencil with coefficients
\begin{equation}
L_1:=\frac{d^2}{dz^2}-k^2,\quad L_0:=U''-UL_1.
\end{equation}
For a viscous fluid, the unperturbed flow is characterized by the Reynolds number $R$, and its stability is determined 
from the \emph{Orr-Sommerfeld equation} \cite{DR}:
\begin{equation}
\mathcal{L}_\mathrm{OS}(c)\Phi=0,\quad\mathcal{L}_\mathrm{OS}(c):=cM_1 +M_0,
\label{eq:Orr-Sommerfeld}
\end{equation}
where $\Phi$ is subjected to the ``no-slip'' boundary conditions $\Phi(z_1)=\Phi(z_2)=0$ and $\Phi'(z_1)=\Phi'(z_2)=0$, and where  
$\mathcal{L}_\mathrm{OS}(c)$ is a linear operator pencil with coefficients
\begin{equation}
M_1:=ikRL_1,\quad M_0:=ikRL_0+
L_1^2.
\end{equation}
In both cases, the values of $c\in\mathbb{C}$ for which there exists a nontrivial solution $\Phi$ are associated 
with exponential growth rates of $-ikc$, and hence the flow is stable to perturbations of wavenumber $k$ if the corresponding
values of $c$ are all real.
\end{example}

\begin{example}[Traveling wave stability in Klein-Gordon equations.]
Let $V:\RR\to\RR$ be a potential function.  The Klein-Gordon equation (of which the sine-Gordon equation \eqref{eq:SG} 
is a special case for $V(u):=1-\cos(u)$) is
\begin{equation}
u_{tt}-u_{xx}+V'(u)=0.
\label{eq:KG}
\end{equation}
Traveling wave solutions $u(x,t)=U(z)$, $z=x-ct$, of phase speed $c$ satisfy the Newton-type equation $(c^2-1)U''(z)+V'(U(z))=0$. 
The wave is stationary in the co-moving frame in which \eqref{eq:KG} is rewritten for $u=u(z,t)$ in the form
\begin{equation}
u_{tt}-2cu_{zt}+(c^2-1)u_{zz}+V'(u)=0.
\end{equation}
Writing $u=U(z)+v(z,t)$ and linearizing for $v$ small one obtains
\begin{equation}
v_{tt}-2cv_{zt}+(c^2-1)v_{zz}+V''(U(z))v=0,
\end{equation}
and seeking solutions of the form $v(z,t)=\phi(z)e^{i\lambda t}$ for $\lambda\in\mathbb{C}$ one arrives at 
a spectral problem involving a quadratic pencil:
\begin{equation}
\mathcal{L}_\mathrm{KG}(\la)\phi=0,\quad\mathcal{L}_\mathrm{KG}:=\la^2L_2 +\la L_1+L_0,
\label{eq:LKG-1}
\end{equation}
where the coefficients are
\begin{equation}
L_2:=1,\quad L_1:=2ic\frac{d}{dz},\quad L_0:=
(1-c^2)\frac{d^2}{dz^2} - V''(U(z)).
\label{eq:LKG-2}
\end{equation}
Note that $\mathcal{L}_\mathrm{KG}(\la)$ is an example of a (formally) selfadjoint quadratic operator pencil.  
Spectral stability is deduced \cite{JonesMMP12} if all  values of $\lambda$ for which there exists a nontrivial solution $\phi$ 
are purely real. See \cite{HSS,StSt2012} for recent general results concerning stability of traveling waves of second order 
in time problems and a list of related references.
\label{example:KG}
\end{example}

\begin{example}[The Rayleigh-Taylor problem.]
Here we give an example of an operator pencil involving partial differential operators and nonlocality (through a divergence-free constraint).
The Rayleigh-Taylor problem concerns the stability of a stationary vertically stratified incompressible viscous fluid of equilibrium 
density profile $\rho_e(z)$.  Making a low Mach number approximation, assuming a small perturbation of the zero velocity field, 
and linearizing the Navier-Stokes equations, one is led to consider normal mode perturbations of the form 
$e^{\lambda t}\mathbf{u}(x,y,z)$ where the linearized velocity field satisfies appropriate no-slip boundary conditions and 
\begin{equation}
-\lambda^2 \rho_e {\mathbf u} + \nabla \cdot (\rho_e {\mathbf u}) {\mathbf g} + \lambda(\nabla p - \eta \triangle {\mathbf u}) = 
0\, , \qquad \nabla \cdot {\mathbf u} = 0,
\label{RTE}
\end{equation}
where  $\mathbf{g}$ is the gravitational acceleration field, $\eta>0$ is the viscosity, and $p$ is the pressure term needed 
to satisfy the incompressibility constraint \cite{GKP}.  A weak reformulation  on an appropriate divergence-free space 
allows  \eq{RTE} to be cast into the form of an equivalent quadratic pencil with parameter $\lambda$ on a Hilbert space.   
For non-Newtonian Maxwell linear viscoelastic fluids, the pencil that arises in the Rayleigh-Taylor problem is a cubic polynomial \cite{KollarH}.
\end{example}

The preceding examples all involve polynomial operator pencils (or, like the rational pencil appearing in the Faddeev-Takhtajan 
spectral problem in Example~\ref{ex:integrable}, that can easily be converted into such).  However, it is important to observe 
that spectral problems for non-polynomial operator pencils also occur very frequently in applications, especially those involving 
a mix of discrete symmetries and continuous symmetries for which the dispersion relation for linear waves is transcendental.  
A fundamental example is the following.
\begin{example}[Delay differential equations]
\label{ex:DDE}
Non-polynomial operator pencils appear naturally in systems of
differential equations with delays
\cite{JKM}. Consider the system
\begin{equation}
\dot{x}(t) = A x(t) + B x(t -\tau),
\label{delayeq}
\end{equation}
for $x\in\mathbb{C}^n$, where $\tau>0$ is a fixed delay
and $A, B$ are complex $n\times n$ matrices.
To study the stability of solutions to \eq{delayeq} with exponential time-dependence of the form $x(t)=e^{\la t}x_0$ 
one needs to solve the spectral problem
\begin{equation}
\LL_{\mathrm{DDE}}(\la)x_0=0,
\label{eq:delay-spectral}
\end{equation}
where $\LL_\mathrm{DDE}$ is the essentially transcendental matrix pencil
\begin{equation}
\LL_{\mathrm{DDE}}(\la) := \la \mathbb{I} - A - e^{-\tau \la} B.
\end{equation}
The existence of values of $\la\in\mathbb{C}$ with $\Re\{\la\}>0$ for which there exists a nonzero 
solution $x_0$ of \eqref{eq:delay-spectral} indicates instability of the system \eqref{delayeq}.
\end{example}

Other applications of operator pencils include the analysis of electric power systems
\cite{mr}, least-squares optimization, and vibration analysis  \cite{tm}.
Various examples of non-polynomial operator pencils are described in the documentation to the MATLAB Toolbox NLEVP \cite{BHMST},
including applications to the design of optical fibers and radio-frequency gun cavity analysis.
Non-polynomial operator pencils also arise in band structure calculations for photonic crystals \cite{EKE}, another example of 
a system in which discrete symmetry (entering through the lattice structure of the crystal) interacts with continuous symmetry 
(time translation). An application of quadratic operator pencils to second-order in time Hamiltonian equations can be found 
in \cite{BJK2}, in which a particular case of the generalized ``good" Boussinesq equation is studied. Finally, note that in 
an alternative to the approach to linearized Hamiltonians described in \S\ref{s:Example}, the so-called \emph{Krein matrix} 
method of Kapitula \cite{Kap} relates a general linearized Hamiltonian spectral problem to a non-polynomial operator pencil.

\subsection{Spectral Theory of Operator Pencils}
In this subsection we summarize the theoretical background needed for our analysis.  
In addition to a proper definition of the spectrum (in particular, the characteristic values) of an operator pencil $\LL$ and its relation to the 
spectra $\sigma(\LL(\la))$ of the individual operators $\LL(\la)$ making up the pencil $\LL$ as $\la$ varies, our method 
relies on a kind of continuity of $\sigma(\LL(\la))$ with respect to $\la$, which can be obtained with appropriate general assumptions. 
Note that even in the general case of operator pencils on infinite dimensional spaces, finite systems of eigenvalues 
have many properties similar  to those of eigenvalues of matrices \cite{Kato}. This allows us to study a wide class of problems, 
although infinite systems of eigenvalues may exhibit various types of singular behavior.  

\subsubsection{Matrix pencils}
Matrix pencils and their perturbation theory are studied 
in \cite{GLRmatrix, GLRinvsub, GLR,Kato, Markus} with a particular emphasis on polynomial matrix pencils. 
The finite dimensional setting allows a simple formulation of the spectral problem. 
\begin{definition}[Characteristic values of matrix pencils]
Let $\LL = \LL(\lambda)$ be a matrix pencil on $X=\CC^N$ defined for $\la\in S\subset\CC$. 
The characteristic values of $\LL$ are the complex numbers $\la\in S$ that satisfy the characteristic equation
$\det (\LL(\lambda)) = 0$.  The set of all characteristic values is called the spectrum of $\LL$ and is denoted $\sigma(\LL)$.
\label{def:char_matrix}
\end{definition}

While the matrices involved all have size $N\times N$ for each $\lambda\in S$,
the characteristic equation need not be a polynomial in $\la$ 
and matrix pencils in general can have an infinite number of characteristic values.
However, characteristic equations of polynomial matrix pencils of degree $p$ in $\la$ are polynomial of degree at most $pN$ in $\la$, 
with equality if and only if the leading coefficient is an invertible matrix.

On the other hand, the eigenvalues $\mu$ of the related eigenvalue problem $\LL(\la) u = \mu u$ satisfy the 
modified characteristic equation $\det(\LL(\lambda) - \mu\mathbb{I}) = 0$ \cite[II.2.1, p.~63]{Kato} 
and belong to the $\la$-dependent spectrum $\sigma(\LL(\la))$ (understood in the usual sense) of  the $N\times N$ matrix $\LL(\la)$.  
Of course, as $\mu$ solves a polynomial equation of degree $N$ in $\mu$, the total algebraic multiplicity 
of the eigenvalues $\mu$ is equal to $\dim(X)=N$, independent of $\lambda\in S$. It is well-known that if the matrix pencil $\LL$ is 
\emph{holomorphic} at $\la_0\in S$ and if all eigenvalues $\mu$ of $\LL(\la_0)$ are simple, then the $N$ 
eigenvalue branches $\mu=\mu(\la)$ are locally holomorphic functions of $\la$ near $\la_0$.  In fact this is an elementary 
consequence of the Implicit Function Theorem. The possible singularities at 
exceptional points $\lambda$ (corresponding to non-simple eigenvalues $\mu$) are well understood.
For linear pencils of the form $\LL(\lambda) = L_0 + \lambda L_1$ 
it is worth emphasizing that the individual eigenvalue functions $\mu=\mu(\la)$ are generally \emph{not} linear 
functions of $\la$; moreover, the corresponding eigenprojections $P_{\mu(\la)}$ can
have poles as functions of $\la$.  See Motzkin and Taussky \cite{MT1, MT2} for a study of special conditions 
under which the eigenvalues of linear pencils are indeed linear in $\lambda$ and the corresponding eigenprojections are entire functions.
Furthermore, mere continuity of $\LL(\lambda)$ does not imply continuity of 
eigenvectors (see \cite[II.1.5, p.110]{Kato} for specific examples). The situation simplifies 
for \emph{selfadjoint} holomorphic matrix pencils as the following theorem indicates. 
\begin{theorem}[{\cite[II.6.1, p.~120, Theorem 6.1]{Kato}}]
Let $\LL(\la)$ be a selfadjoint holomorphic matrix pencil defined for $\la\in S=\overline{S}$ 
with $\mathbb{R}\subset S$.  Then, for $\la\in\mathbb{R}$, the eigenvalue functions $\mu=\mu(\lambda)$ 
can be chosen to be holomorphic. 
Moreover, for $\lambda\in\mathbb{R}$, the (orthonormal) eigenvectors can be chosen as holomorphic functions of $\lambda$.
\label{matrixhol} 
\end{theorem}

Thus even if a $k$-fold eigenvalue $\mu$ occurs for some $\la_0\in\mathbb{R}$, locally this can be viewed as 
the intersection of $k$ real holomorphic branches $\mu=\mu(\la)$ (and similarly for the corresponding eigenvectors). 
Note that analyticity is not a necessary condition for mere differentiability of eigenvalues $\mu$ and eigenvectors 
for $\lambda\in\mathbb{R}$.  Indeed,  according to Rellich \cite{Rellich8},  
selfadjoint continuously differentiable matrix pencils have continuously 
differentiable eigenvalue and eigenfunction branches \cite[II.6.3, p.~122, Theorem 6.8]{Kato}.

\subsubsection{Operator pencils}
In passing to infinite-dimensional spaces, we want to restrict our attention to holomorphic pencils, and to handle unbounded 
(e.g. differential) operators, we need to generalize the definition given earlier.  A convenient generalization is the following.

\begin{definition}[Holomorphic families of type (A) {\cite[VII.2.1, p.~375]{Kato}}]
An operator pencil $\LL$ consisting of closed operators $\LL(\la): D(\la) \subset X \rightarrow X$, defined for $\la \in S$, 
is called a holomorphic family of type (A) if the domain $D=D(\la)$ is independent of $\la \in S$ and 
if for every $u\in D$, $\LL(\la) u$ is holomorphic as a function of $\la \in S$.  
\label{def:HolomorphicTypeA}
\end{definition}

Note that a holomorphic pencil of densely-defined bounded operators (having by definition an operator-norm convergent 
power series expansion about each $\la_0\in S$) is an example of a holomorphic family of type (A).  
Conversely, a holomorphic family of type (A) consisting of uniformly (with respect to $\la$) bounded operators is a holomorphic pencil in the original sense \cite[VII.1.1, p. 365]{Kato}.
In this context, we present the following notion of spectrum of operator pencils (see also \cite{GLR, Keldysh, KLa}).
\begin{definition}[Spectrum of operator pencils and related notions, \cite{Markus}]
Let $\LL$ be a holomorphic family of type (A) on a Banach space $X$ with domain $D$ defined for $\la\in S\subset\CC$.
A complex number $\la_0\in S$ is called a regular point of $\LL$ if $\LL(\la_0)$ has a bounded inverse on $X$. 
The set of all regular points of $\LL$ is called the resolvent set $\rho(\LL)$ of $\LL$. 
The complement of $\rho(\LL)$ in $S\subset\CC$ is called the spectrum $\sigma(\LL)$ of $\LL$. 
A complex number $\la_0\in S$ is called a characteristic value of $\LL$ if there exists a nonzero 
$u \in D$, called a characteristic vector corresponding to $\la_0$, satisfying 
\begin{equation}
\LL(\la_0) u = 0.
\label{Lu}
\end{equation}
The dimension of the kernel of $\LL(\la_0)$ is called the geometric multiplicity of $\la_0$.  
\label{def:op_spectrum}
\end{definition}

The characteristic values of $\LL$ are contained in (but need not exhaust) the spectrum $\sigma(\LL)$. 
The correct generalization of Jordan chains of generalized eigenvectors to the context of operator pencils is given by the following definition.
\begin{definition}[Root vector chains and maximality]
Let $\LL$ be a holomorphic family of type (A) on a Banach space $X$ with domain $D$ defined for $\la\in S\subset\CC$, 
and let $\la_0\in S$ be a characteristic value of $\LL$ of finite geometric multiplicity.
A sequence of vectors $\{u^{[0]},u^{[1]}, \dots, u^{[m-1]}\}$, each lying in $D$, where $u^{[0]}\neq 0$ is a characteristic vector for  
$\la_0$, is called a chain of root vectors (or generalized characteristic vectors) of length $m$ for $\la_0$ if
\footnote{By Definition~\ref{def:HolomorphicTypeA}, $\LL^{(\ell)}(\la_0)u$ is a well-defined vector in $X$ for each $\la_0\in S$ and each $u\in D$.}
\begin{equation}
\sum_{\ell=0}^q \frac{1}{\ell !} \LL^{(\ell)}(\la_0) u^{[q-\ell]} = 0,  \
q=1, \dots, m-1,\quad \LL^{(\ell)}(\la_0)u:=\left.\frac{d^\ell}{d\la^\ell}(\LL(\la)u)\right|_{\la=\la_0}.
\label{m0}
\end{equation}
A chain for $\la_0$ of length $m$ that cannot be extended to a chain of length $m+1$ is called a maximal chain for $\la_0$.  
That is, $\{u^{[0]},u^{[1]},\dots,u^{[m-1]}\}$ is maximal if $u^{[0]}\neq 0$ lies in $\Ker(\LL(\la_0))$ and 
$u^{[1]},\dots,u^{[m-1]}$ satisfy \eqref{m0}, 
but there does not exist $u^{[m]}\in D$ such
that \eqref{m0} holds for $q=m$.
\label{def:op_chains}
\end{definition}

Maximal chains of root vectors for a characteristic value $\la_0$ give rise to the notion of algebraic multiplicity for operator pencils.  
Roughly speaking, if a basis of $\Ker(\LL(\la_0))$ is chosen such that the maximal chains generated therefrom are as long as 
possible, then the algebraic multiplicity of $\la_0$ is the sum of lengths of these chains.  A more precise definition involves a flag 
(a nested sequence) of subspaces of $\Ker(\LL(\la_0))$.
\begin{definition}[The canonical flag of subspaces and canonical sets of chains]
Let $\LL$ be a holomorphic family of type (A) on a Banach space $X$ with domain $D$ defined for $\la\in S\subset\CC$, 
and let $\la_0\in S$ be a characteristic value of $\LL$ of finite geometric multiplicity $k$.  Let $X_s\subset\Ker(\LL(\la_0))$ 
denote the subspace of $\Ker(\LL(\la_0))$ spanned by characteristic vectors $u^{[0]}$ that begin maximal chains of length 
at least $s$. The sequence of subspaces $\{X_s\}_{s=1}^\infty$ is called the canonical flag of $\LL$ corresponding to the 
characteristic value $\la_0$. A set of maximal chains $\{\{u_j^{[0]},\dots,u_j^{[m_j-1]}\}\}_{j=1}^k$ is said to be canonical
if each subspace of the canonical flag can be realized as the span of some subset of the $k$ vectors $\{u_j^{[0]}\}_{j=1}^k$.  
Finally, if $\{u^{[0]},\dots,u^{[m-1]}\}$ is a chain belonging to a canonical set, then the span of these $m$ vectors is called 
the root space of $\LL$ corresponding to the characteristic value $\la_0$ and the characteristic vector $u^{[0]}$.
\label{def:op_canonicalflag}
\end{definition}

Obviously $X_1=\Ker(\LL(\la_0))$ and $X_1\supseteq X_2\supseteq X_3\supseteq\cdots$.  The definition of canonical sets 
of chains makes precise the idea of maximal chains that are ``as long as possible''.  Therefore, we are now in a position to 
properly define the algebraic multiplicity of a characteristic value.

\begin{definition}[Algebraic multiplicity of characteristic values]
\label{def:op_multiplicity}
Let $\LL$ be a holomorphic family of type (A) on a Banach space $X$ with domain $D$ defined for $\la\in S\subset\CC$, 
and let $\la_0\in S$ be a characteristic value of $\LL$ of finite geometric multiplicity $k$ with the canonical set of maximal chains 
$\{\{u_j^{[0]},\dots,u_j^{[m_j-1]}\}\}_{j=1}^k$. Then the numbers $m_1,\dots,m_k$ are called the partial multiplicities, and their sum $\alpha:=m_1+\cdots + m_k$ the algebraic multiplicity, of the characteristic value $\la_0$.  
A characteristic value $\la_0$ is called semi-simple if its geometric and algebraic multiplicities are finite and equal:  $\alpha=k$, and simple if $\alpha=k=1$.
\end{definition}%

Canonical sets of maximal chains need not be unique, but every canonical set of chains results in the same partial multiplicities, 
a situation reminiscent of Jordan chains consisting of generalized eigenvectors of a matrix.  
However, in stark contrast to a Jordan chain of a non-semi-simple eigenvalue of a fixed matrix, 
the vectors forming a chain for an operator or matrix pencil need not be linearly independent, and a generalized 
characteristic vector can be identically equal to zero. 
 Note that for holomorphic matrix pencils, 
the algebraic multiplicity of a characteristic value $\la_0$ reduces to the order of vanishing of the characteristic determinant
$\det(\LL(\la))$ at $\la_0$ (see \cite[\S1.7, p. 37, Proposition 1.16]{GLRmatrix}). The following examples 
illustrate some of the unique features of root vector chains in the simple context of matrix pencils.

\begin{example}
Consider the quadratic selfadjoint matrix pencil acting on $X=\CC^2$:
\begin{equation}
\mathcal{L}(\lambda):=\begin{pmatrix}
\lambda^2-2\lambda + 1 & 2-2\lambda
\\
2-2\lambda & \lambda^2+3
\end{pmatrix}.
\end{equation}
Since $\det(\mathcal{L}(\lambda))=(\lambda-1)^3(\lambda+1)$,  the characteristic values are $\la_0=\pm 1$.

For $\la_0=-1$, we have
$\Ker(\LL(-1))=\sspan\{u^{[0]}\}$ where $u^{[0]}:=(1,-1)^\mathsf{T}$,
so the geometric multiplicity of $\la_0=-1$ is $1$.  The maximal chain beginning with $u^{[0]}$ is $\{u^{[0]}\}$, that is, 
the condition \eqref{m0} with $q=1$ governing $u^{[1]}$ is inconsistent.  The singleton $\{\{u^{[0]}\}\}$ is therefore 
a canonical set of chains for $\la_0=-1$ with algebraic multiplicity $\alpha=1$ consistent with the linear degree of the 
factor $(\la+1)$ in $\det(\LL(\la))$.

For $\la_0=1$, we have
$\Ker(\LL(1))=\sspan\{u^{[0]}\}$ where $u^{[0]}=(1,0)^\mathsf{T}$,
so the geometric multiplicity of $\la_0=1$ is again $1$.  In this case, \eqref{m0} with $q=1$
admits the general solution
$u^{[1]}=(c_1,\tfrac{1}{2})^\mathsf{T}$ where $c_1$ is arbitrary,
and then \eqref{m0} with $q=2$ admits the general solution
$u^{[2]}=(c_2,\tfrac{1}{2}c_1-\tfrac{1}{4})^\mathsf{T}$ where $c_2$ is again arbitrary.
However \eqref{m0} with $q=3$ is inconsistent regardless of the values of the free parameters $c_1$ and $c_2$.  
Hence for any choice of the constants $c_j$, the singleton  $\{\{u^{[0]},u^{[1]},u^{[2]}\}\}$ is a canonical set of 
chains for $\la_0=1$ and we conclude that the algebraic multiplicity of $\la_0=1$ is $\alpha=3$ consistent with 
the cubic degree of the factor $(\la-1)^3$ in $\det(\LL(\la))$. Note that in the case of $\la_0=1$, 
the vectors of the (unique) chain 
in the canonical set are clearly linearly dependent, as there are three of them and the overall space $X$ has 
dimension only two. Moreover, if we choose $c_1=\tfrac{1}{2}$ and $c_2=0$, then the generalized characteristic 
vector $u^{[2]}$ vanishes identically. 
\end{example}

\begin{example}
Consider the quadratic selfadjoint  matrix pencil acting in $X=\CC^2$:
\begin{equation}
\mathcal{L}(\lambda):=\begin{pmatrix} 
\lambda^2 -\lambda & 1-\lambda 
\\
1-\lambda & \lambda^2-\lambda
\end{pmatrix},
\end{equation}
for which we again have $\det(\mathcal{L}(\lambda))=(\lambda-1)^3(\lambda+1)$ with
characteristic values $\la_0=\pm 1$.

For $\la_0=-1$ we have $\Ker(\LL(-1))=\spanv\{u^{[0]}\}$ where $u^{[0]}=(1,-1)^\mathsf{T}$.
It is easy to check that \eqref{m0} with $q=1$ is inconsistent, so $\{\{u^{[0]}\}\}$ is a canonical set of 
chains for $\la_0=-1$ and the geometric and algebraic multiplicities are both $1$.

For $\la_0=1$ we have $\LL(1)=0$ and hence
$\Ker(\LL(1))=X=\CC^2$.
The geometric multiplicity of $\la_0=1$ is therefore $2$, and hence one may select as many as two 
linearly independent vectors $u_j^{[0]}$, $j=1,2$, and each one will generate its own maximal chain.  
To understand what distinguishes a canonical set of chains, it is useful to consider $u^{[0]}$ to be a 
completely general nonzero vector in $\Ker(\LL(1))$ by writing $u^{[0]}=(r,s)^\mathsf{T}$ for $r$ and $s$ not both zero.
Then \eqref{m0} for $q=1$ reads
\begin{equation}
\begin{pmatrix}0 & 0\\0 & 0\end{pmatrix}u^{[1]}+\begin{pmatrix}1 & -1\\-1 & 1\end{pmatrix} 
u^{[0]}=0\quad\Leftrightarrow\quad \begin{pmatrix}r-s\\s-r\end{pmatrix}=0.
\end{equation}
Since $\LL(1)=0$ this equation cannot place any condition on $u^{[1]}$ at all, and moreover it is a consistent equation 
only if $s=r$.  Therefore the maximal chain generated from any nonzero vector $u^{[0]}$ not lying in the span of 
$(1,1)^\mathsf{T}$ has length $1$.  On the other hand if we take $r=s=1$, then \eqref{m0} for $q=1$ is consistent 
but places no condition on $u^{[1]}$, and hence $u^{[0]}=(1,1)^\mathsf{T}$ implies that $u^{[1]}=(c_1,c_2)^\mathsf{T}$
for arbitrary constants $c_1$ and $c_2$ (which could be taken to be equal or even zero).  It is then easy to check that 
there does not exist a choice of these constants making \eqref{m0} for $q=2$ consistent, and hence the maximal chain 
starting with $u^{[0]}=(1,1)^\mathsf{T}$ has length $2$.

In the case of $\la_0=1$ we then see that any basis of $X_1:=\Ker(\LL(1))$ that does not include a nonzero vector proportional 
to $(1,1)^\mathsf{T}$ will generate two maximal chains, each of length $1$ for a value of $\alpha=1+1=2$.  However, if one 
of the basis vectors is proportional to $(1,1)^\mathsf{T}$, then its maximal chain will have length $2$ and hence $\alpha=1+2=3$.  
Only in the latter case do the two chains make up a canonical set for $\la_0=1$, and we deduce that the algebraic multiplicity of
$\la_0=1$ is $\alpha=3$.  The corresponding subspaces of the canonical flag are $X_2:=\sspan\{(1,1)^\mathsf{T}\}$ 
and $X_s=\{0\}$ for $s\ge 3$.
\end{example}

While pencils consisting of unbounded operators occur frequently in applications, it is sometimes mathematically preferable to deal instead with bounded operators.  This can be accomplished with the use of the following result, 
the proof of which can be found in Appendix~\ref{app1}.
\begin{proposition}
Let $\LL=\LL(\la)$ be a holomorphic family of type (A) defined for $\la\in S\subset\CC$ on a domain $D\subset X$ of a Banach space $X$.  Let $\la'\in S$, and let $\delta\in\CC$ be such that $\LL(\la')+\delta\II$ has a bounded inverse defined on $X$ (without loss of generality we may assume $\delta\neq 0$ because the resolvent set of the operator $\LL(\la')$ is open).  Then there exists $\epsilon>0$ such that 
\begin{equation}
\MM(\la):=\II-\BB(\la),\quad \BB(\la):=\delta(\LL(\la)+\delta\II)^{-1}
\label{eq:MMdefine1}
\end{equation}
is a holomorphic pencil of bounded operators on $X$ for $|\la-\la'|<\epsilon$.  Moreover,
\begin{itemize}
\item[(a)] $\la_0$ with $|\la_0-\la'|<\epsilon$ is a characteristic value of $\LL$ if and only if it is a characteristic value of $\MM$.
\item[(b)] A sequence of vectors $\{u^{[0]},u^{[1]},\dots,u^{[m-1]}\}$ is a maximal chain for $\LL$ corresponding to a characteristic value $\la_0$ with $|\la_0-\la'|<\epsilon$ if and only if it is a maximal chain also for $\MM$ with the same characteristic value.
\item[(c)] The algebraic and geometric multiplicities of a common characteristic value $\la_0$ with $|\la_0-\la'|<\epsilon$ are the same for $\LL$ and $\MM$.
\end{itemize}
\label{theorem:algebraic-equivalence}
\end{proposition}

\subsubsection{A special technique for polynomial operator pencils}
\label{sec:companion}
Let $\LL$ be a polynomial operator pencil of degree $p$:  $\LL(\la):=L_0+\la L_1 +\cdots + \la^p L_p$, 
where $L_j$ are all operators defined on a common dense 
domain $D$ within a Banach space $X$, and suppose that $L_p$ is invertible.  Such a pencil is obviously 
a holomorphic family of type (A).  In this case, the problem 
of characterizing the spectrum $\sigma(\LL)$ can be reduced to that of finding the spectrum (in the usual sense) 
of a single operator $C_\LL$, 
called the \emph{companion matrix} \cite{GLR,Markus}, acting on the $p$-fold Cartesian product space $X^p$.  Indeed, given $u\in D$, 
define $\mathbf{u}:=(u,\la u,\dots,\la^{p-1}u)^\mathsf{T}\in X^p$. Then the equation $\LL(\la)u=0$ on $X$ is equivalent 
to the standard eigenvalue equation $C_\LL\mathbf{u}=\la\mathbf{u}$ on $X^p$, where $C_\LL$ is the $p\times p$ matrix of operators:
\begin{equation}
C_\LL:=\begin{pmatrix}
0 & \II & 0 & \cdots & 0\\
0 & 0 & \II & \cdots & 0\\
\vdots & \vdots & \vdots & \ddots & \vdots\\
0 & 0 & 0 & \cdots & \II \\
-\tilde{L}_0 & -\tilde{L}_1 & -\tilde{L}_2 & \cdots & -\tilde{L}_{p-1}\end{pmatrix},
\ \tilde{L}_k:=L_p^{-1}L_k, \quad 0 \le k\le p-1.
\end{equation}
Here $\II$ and $0$ are respectively the identity and zero operators  acting on $X$.  The root spaces of the pencil $\LL$ 
are the images  of the invariant subspaces 
of the companion matrix $C_\LL$ under the obvious projection of $X^p$ onto the first factor $X$.
The projection $X^p\to X$ provides another explanation of how a generalized
characteristic vector can vanish; the corresponding generalized eigenvector of the companion matrix $C_\LL$ of course is nonzero, 
but its image under the projection can indeed be zero  (see \cite[Chapter 12]{GLR}).

Since groups of isolated eigenvalues of $C_\LL$ with finite multiplicities have properties similar to those of eigenvalues 
of finite-dimensional matrices, it is possible to completely characterize the chains of the pencil $\LL$ in terms of the Jordan 
chains of $C_\LL$.  The constructive algebraic proof given for matrix pencils in \cite{GLR} can be used in the operator pencil setting without modification (see also \cite[\S12]{Markus}).

\begin{theorem}[{\cite[p.~250, Proposition 12.4.1]{GLR}}]
Let $\LL$ be a polynomial operator pencil of degree $p$ with an invertible leading coefficient $L_p$, and let $C_{\LL}$ 
be the companion matrix of $\LL$. Then $\{u^{[0]}, \dots, u^{[m-1]}\}$ is a chain of length $m$ for $\LL$ as defined 
in \eq{m0}  at a characteristic value $\la_0$ if and only if the columns of the $p \times m$ matrix 
$V =  ( U, UJ_0, \dots, UJ_0^{p-1} )^\mathsf{T}$ form a standard Jordan chain for  $C_{\LL}$ 
corresponding to the same $\la_0$, where $J_0$ is the $m \times m$ matrix Jordan 
block\footnote{$J_{0kl}:=\lambda_0\delta_{k-l}+\delta_{k-l+1}$ where $\delta_i$ denotes the Kronecker delta symbol.} 
with eigenvalue $\la_0$ and $U = (u^{[0]}, \dots, u^{[m-1]})$. 
\label{propGLR}
\end{theorem}

\subsection{Indefinite inner product spaces}
Unfortunately, in the case that $\LL$ is a \emph{selfadjoint} polynomial operator pencil (with coefficients $L_j$ being 
selfadjoint operators densely defined on a Hilbert space $X$), selfadjointness is lost in the extension process and 
the companion matrix $C_{\LL}$ acting on $X^p$ is non-selfadjoint with respect to the ``Euclidean'' inner product  
on $X^p$ induced by that on $X$:  $(\mathbf{u},\mathbf{v}):=(u_1,v_1)+\cdots +(u_{p},v_p)$.
However, a calculation shows that $C_\LL$ is indeed selfadjoint with respect to an \emph{indefinite} quadratic form defined by
\begin{equation}
\langle\mathbf{u},\mathbf{v}\rangle:=(\mathbf{u},B_\LL\mathbf{v}),\quad
B_\LL:=\begin{pmatrix}L_1 & L_2 &\cdots &L_{p-1} &  L_p\\
L_2 & L_3 & \cdots & L_p & 0\\
\vdots & \vdots & \iddots & \vdots & \vdots\\
L_{p-1} & L_p &\cdots & 0 & 0\\
L_p & 0 &\cdots & 0 & 0\end{pmatrix}.
\label{eq:indefinite-form}
\end{equation}
Indeed, the Hankel-type operator matrix $B_\LL$ (selfadjoint with respect to the Euclidean inner product on $X^p$) 
intertwines the companion matrix $C_\LL$ with its adjoint $C_\LL^*$ with respect to the Euclidean inner product 
as follows: $B_{\LL} C_{\LL} = C_{\LL}^{\ast} B_{\LL}$. This implies that the root spaces in $X^p$ corresponding 
to different  eigenvalues of $C_{\LL}$ (characteristic values of $\LL$) are orthogonal with respect to 
the quadratic form $\langle\cdot,\cdot\rangle$ defined by \eqref{eq:indefinite-form}.

The extension technique therefore closely relates the theory of selfadjoint polynomial operator pencils 
to that of so-called \emph{indefinite inner product spaces} or \emph{Pontryagin spaces}. 
As  is apparent from the definition \eqref{simpleKreinexamples}, the Krein signature of a characteristic value is related to 
a certain indefinite quadratic form.   Although one of the aims of our article is the avoidance of the algebra of indefinite 
inner product spaces, the latter are clearly lying just beneath the surface, so we would like to 
briefly cite some of the related literature. The seminal work of Pontryagin \cite{Pontryagin} 
opened up a huge field devoted to the spectral properties and invariant subspaces 
of operators in such spaces having a wide variety of applications.  Important further developments of the theory were made 
in \cite{Iohvidov}, and general references for many of the key results include \cite{GK,GLR}. 
Over 40 years after its publication a central result of the theory --- the Pontryagin Invariant Subspace Theorem ---
was rediscovered in connection with nonlinear waves
and index  theorems \cite{CP,Grillakis1990, GKP} (see \cite{CP} for a historical discussion). 

\section{Graphical Interpretation of Multiplicity and Krein Signature\label{s:Krein}}
This section contains a survey of known results (with some new generalizations) connecting the characteristic value problem \eq{Lu} 
for a \emph{selfadjoint} operator pencil $\LL$ to the  family, parametrized by $\la\in\RR$, of selfadjoint eigenvalue problems 
\begin{equation}
\LL(\la) u(\la) = \mu(\la) u(\la).
\label{Lmu}
\end{equation} 
In particular, we will be concerned with \emph{real} characteristic values $\la_0$
and the corresponding root vector chains.  We will demonstrate that all of the essential information
is equivalently contained in the way the eigenvalues $\mu(\la)$ and eigenvectors $u(\la)$ of the problem \eqref{Lmu} depend on $\la$ near $\la=\la_0$.
 
In order to consider the $\lambda$-dependent spectrum (in the usual sense) of $\LL(\la)$  as an operator depending parametrically on 
$\la\in S\cap\RR$, and in particular the eigenvalues $\mu(\la)$ of $\LL(\la)$, we will restrict attention
to certain types of selfadjoint operator pencils for which the dependence of the spectrum on $\la$
is analytic for $\la\in\RR$.
The fact that holomorphic families $\LL(\la)$ of type (A) that have compact resolvent for some $\la_0\in S$ 
have a compact resolvent for all $\la \in S$ \cite[VII.2.1, p.~377, Theorem 2.4]{Kato}
can then be used to establish the following theorem%
\footnote{The part of this result that is concerned with compact operators is discussed in  \cite[VII.3.5, p.~393, Remark 3.11]{Kato}. 
The condition that the kernel is trivial is crucial; a nontrivial kernel can ruin the analyticity of eigenvalue branches passing through $\mu=0$.}. 

\begin{theorem}[Kato, {\cite[VII.3.5, p.~392, Theorem 3.9]{Kato}}]
\label{thm:operator-branches}
Let $\LL(\la)$ be a selfadjoint holomorphic family of type (A) defined for $\la \in S$ on a dense domain $D$ in a Hilbert space $X$, 
and let $I_0 \subset S$ be a real interval. If either 
\begin{itemize}
\item
$\LL(\la_0)$ has a compact resolvent for some $\la_0 \in S$ or
\item
$\LL(\la)$ is itself compact and $\ker(\LL(\la))=0$ for all $\la\in S$, 
\end{itemize}
then 
all the eigenvalues of $\LL(\la)$ can be represented as a sequence $\{\mu_j(\la)\}$ of real holomorphic functions on $I_0$. 
Furthermore, the complete orthonormal family of corresponding eigenvectors can also be represented by
a sequence of vector valued holomorphic functions on $I_0$. 
\end{theorem}

Selfadjoint holomorphic families of type (A) with compact resolvent occur frequently in the theory of stability of nonlinear waves, as the following example shows.

\begin{example}[Bloch solutions of the linearized Klein-Gordon equation for periodic traveling waves.]
Recall from Example~\ref{example:KG} the quadratic operator pencil $\LL_\mathrm{KG}$ arising in the theory of linearized stability of traveling waves for the Klein-Gordon equation as defined
by \eqref{eq:LKG-1}--\eqref{eq:LKG-2}.  Let us assume that the traveling wave whose stability is of interest 
is periodic in the sense that $V''(U(z))$ is a bounded periodic function of $z$ with fundamental period 
$T$:  $V''(U(z+T))=V''(U(z))$ for $z\in\RR$, and that $c^2\neq 1$.  The \emph{Floquet} or \emph{Bloch} spectrum for this problem is parametrized 
by $\theta\in\RR\pmod{2\pi}$ and corresponds to 
formulating the characteristic equation $\LL_\mathrm{KG}(\la)u=0$ on the subspace $D_0$ of the domain 
$C^2([0,T])\subset L^2([0,T])$ consisting of $u$ satisfying the 
boundary conditions $u(T)=e^{i\theta}u(0)$ and $u'(T)=e^{i\theta}u'(0)$.  For $\la\in\RR$, the operator $\LL_\mathrm{KG}(\la)$ with 
these side conditions is essentially selfadjoint  
on $D_0$.  The domain $D$ of selfadjointness for $\la\in\mathbb{R}$ is determined by the highest-order derivative and hence is independent of $\la$.  Since Green's function for this problem is a Hilbert-Schmidt kernel on $[0,T]^2$, it is easy to see that the resolvent $(\LL_\mathrm{KG}(\la)-\mu\mathbb{I})^{-1}$ is compact for each $\mu$ for which it exists.  
\end{example}

Of course the spectral theory of compact holomorphic families of type (A) is obviously related to that of holomorphic families of type (A) consisting of Fredholm operators, i.e., compact perturbations of the identity.  The latter class is essentially equivalent to the class of holomorphic families of type (A) with compact resolvent, as was shown in Proposition~\ref{theorem:algebraic-equivalence}.  In fact, Theorem~\ref{thm:operator-branches} implies that
that the $\lambda$-dependent spectrum of the operators $\LL(\la)$ making up a selfadjoint family of type (A) with compact resolvent is real, discrete, and depends continuously upon $\la\in\RR$.
Therefore the constant $\delta\neq 0$ in the statement of Proposition~\ref{theorem:algebraic-equivalence} can be taken to be real, in which case the bounded holomorphic operator pencil $\BB(\la)$ defined by \eqref{eq:MMdefine1}  consists of \emph{compact} selfadjoint operators with trivial kernel for $\la'-\epsilon<\la<\la'+\epsilon$.

\subsection{Multiplicity of characteristic values}
In this section we present a theorem (Theorem~\ref{multiplicities}) well-known in the theory 
of matrix pencils (\cite[\S 12.5, p.~259]{GLR}, see also \cite{KLjub} for a similar approach in the case of operators). 
The key idea here is to relate the linear-algebraic notions of algebraic and geometric multiplicity 
of characteristic values and root vector chains to information contained in the graphs of eigenvector branches 
$\mu=\mu(\la)$ and the corresponding eigenvector branches $u=u(\la)$ of the selfadjoint eigenvalue problem \eqref{Lmu}.

Before formulating the main result, we first lay some groundwork.  Let $X$ be a Hilbert space with inner product $(\cdot,\cdot)$, and let $\LL$ be a selfadjoint operator pencil that is a holomorphic family of type (A) defined for $\la\in S=\overline{S}$ on a dense domain $D\subset X$.  We consider two further restrictions of $\LL$:   either
\begin{itemize}
\item $\LL(\la)$ has compact resolvent for some (and hence all) $\la\in S$, or
\item $\LL(\la)$ is an operator of Fredholm type, having the form $\LL(\la)=\II-\BB(\la)$ for a compact selfadjoint operator pencil $\BB$ with trivial kernel for $\la\in S\cap\RR$.  
\end{itemize}
According to Theorem~\ref{thm:operator-branches}, 
there exist both a sequence of real analytic functions $\{\mu_j(\la)\}_{j=1}^\infty$ and corresponding analytic
vectors $\{u_j(\la)\}_{j=1}^\infty$ defined for $\la\in S\cap\mathbb{R}$ such that 
$\LL(\la)u_j(\la)=\mu_j(\la)u_j(\la)$ for all $j$, and also $(u_j(\la),u_k(\la))=\delta_{jk}$ for all $j,k$.
Moreover the eigenvectors form an orthonormal basis for $X$ for each $\lambda\in S\cap\mathbb{R}$.  For $\la\in S\cap \RR$, we define a unitary operator pencil $\VV(\la):X\to\ell_2 = \ell_2(\mathbb{N})$ assigning to
a vector $w\in X$ its generalized Fourier coefficients:
\begin{equation}
\VV(\la)w:=\left\{\left(w,u_j(\la)\right)\right\}_{j=1}^\infty.
\end{equation}
By completeness, the inverse/adjoint of $\VV(\la)$ is a unitary operator $\UU(\la):\ell_2\to X$ given by
\begin{equation}
\UU(\la)\{a_j\}_{j=1}^\infty:=\sum_{j=1}^\infty a_ju_j(\la).
\end{equation}
Derivatives of $\UU$ with respect to $\la$ may be defined on suitable subspaces of $\ell_2$ by the formula
\begin{equation}
\UU^{(n)}(\la)\{a_j\}_{j=1}^\infty:=\sum_{j=1}^\infty a_ju_j^{(n)}(\la).
\label{eq:UUderivs}
\end{equation}
In particular, arbitrary derivatives of $\UU(\la)$ can be applied to sequences $\{a_j\}_{j=1}^\infty$ for which all but a finite number of the $a_j$ are zero.

Suppose now that $\la_0\in S\cap\RR$ is a real characteristic value%
\footnote{Under our assumptions on $\LL$, the characteristic value $\la_0$ is isolated on the real line, i.e., 
there exists an open real neighborhood $\mathcal{O}$ of $\la_0$ such that $\sigma(\LL) \cap \mathcal{O} = \{ \la_0 \}$. 
This fact follows from \cite[VII.4.5, p.~392, Theorem 3.9]{Kato}.}
of $\LL$.
This means that $\Ker(\LL(\la_0))$ is nontrivial, having some dimension (geometric multiplicity) $k>0$, supposed finite.  
Obviously, $\Ker(\LL(\la_0))$ is spanned by exactly those $k$ orthonormal eigenvectors $u_j(\la_0)$ for which $\mu_j(\la_0)=0$.  
Let $\{X_s\}_{s=1}^{\infty}$ be the canonical flag of $\LL$ corresponding to $\la_0$, as in Definition~\ref{def:op_canonicalflag}.
Define a flag $\{Y_s\}_{s=1}^\infty$ of subspaces of $X$ as follows:  $Y_s$ is the span of those eigenvectors $u_j(\la_0)$ 
for which $\mu^{(n)}(\la_0)=0$ for all $n=0,1,2,\dots,s-1$.  Clearly,
$Y_1:=\Ker(\LL(\la_0))=X_1$, and then $Y_1\supseteq Y_2\supseteq Y_3\supseteq \cdots$.  

The following result shows that in the current context, chains of root vectors for the real characteristic value 
$\la_0$ can be constructed from derivatives of the eigenvectors $u_j(\la)$, and this implies that
 \emph{the flags $\{X_s\}_{s=1}^{\infty}$ and $\{Y_s\}_{s=1}^\infty$ actually coincide}. 
 
\begin{proposition}
Let $\LL$ be a selfadjoint holomorphic family of type (A) defined on a dense domain $D\subset X$ either having compact resolvent or being of Fredholm form with $\II-\LL$
being a compact selfadjoint pencil with trivial kernel.  Let $\la_0\in S\cap \RR$ be a real characteristic value of $\LL$ of finite geometric multiplicity $k>0$, and  
let the flag $\{ Y_s \}_{s=1}^{\infty}$ be defined as above. Then there exists a chain of root vectors for $\lambda_0$ 
of length $m$ if and only if $u^{[0]}$ is a nonzero vector in the subspace $Y_m$, in which case it also holds that 
the chain $\{u^{[0]},u^{[1]},\dots,u^{[m-1]}\}$ has the form 
($\UU_0^{(d)}:=\UU^{(d)}(\la_0)$ and $\VV_0:=\VV(\la_0)$)
\begin{equation}
u^{[r]}=\sum_{d=0}^r\frac{1}{d!}\UU_0^{(d)}\VV_0 w^{[r-d]},\quad r=0,1,2,\dots,m-1, \ \ 
\text{where $w^{[s]}\in Y_{m-s}$.}
\label{eq:urformula}
\end{equation}
\label{prop:chains_derivatives}
\end{proposition}

Note that in \eqref{eq:urformula} at most $k$ elements of the $\ell_2$ vector 
$\VV_0 w^{[r-d]}$ are nonzero because $Y_{m-s}\subset Y_1=\Ker(\LL(\la_0))$, 
and hence the application of $\UU_0^{(d)}$ makes sense.
Although the matrix pencil version of Proposition~\ref{prop:chains_derivatives}  is very well understood, 
we are not aware of an analogous result in the literature for operator pencils. 
We therefore give a complete proof of Proposition~\ref{prop:chains_derivatives} in Appendix~\ref{app2}, one that is quite different from the matrix pencil case (the latter relies on the determinantal characterization of characteristic values and eigenvalues \cite{GLR}).

This result immediately yields a useful ``graphical'' notion of multiplicity of real characteristic values.
\begin{theorem}[Graphical interpretation of multiplicity]
Let $\LL$ be a selfadjoint holomorphic family of type (A) defined on a Hilbert space $X$ for $\la\in S=\overline{S}$, that either has compact resolvent or is of Fredholm form $\LL=\II-\BB$ where $\BB$ is a compact operator pencil with trivial kernel.  Let $\la_0\in S\cap\RR$ be a real characteristic value of $\LL$.  Then
the flags $\{X_s\}_{s=1}^\infty$ and $\{Y_s\}_{s=1}^\infty$ coincide, and hence a canonical set
of chains can be constructed in terms of derivatives of the eigenvectors $u_j(\la)$ at $\la=\la_0$
as in the statement of Proposition~\ref{prop:chains_derivatives}.  Moreover $\la_0$ has finite
geometric multiplicity $k$ and partial algebraic multiplicities $m_1,\dots,m_k$ if and only if
there exist exactly $k$ analytic eigenvalue branches $\mu=\mu_1(\la),\dots,\mu_k(\la)$ 
vanishing at $\la=\la_0$ where $\mu_j(\la)$ vanishes to order $m_j$, i.e., $\mu_j^{(n)}(\la_0)=0$ for $n=0,1,\dots,m_j-1$, while $\mu_j^{(m_j)}(\la_0)\neq 0$.
\label{multiplicities}
\end{theorem}

The significance of Theorem~\ref{multiplicities} is that the algebraic and geometric multiplicities of isolated real 
characteristic values of selfadjoint pencils can be simply read off from plots of  the real eigenvalue branches $\mu=\mu_j(\la)$ 
of the selfadjoint eigenvalue problem \eqref{Lmu}. Indeed, real characteristic values are simply the locations of 
the intercepts of the branches with $\mu=0$, the geometric multiplicity of a real characteristic value $\la_0$ is simply 
the number of branches crossing at the point $(\la,\mu)=(\la_0,0)$, each branch $\mu=\mu_j(\la)$ corresponds to a single maximal 
chain belonging to a canonical set of chains of 
root vectors, and the length of each such chain is simply the order 
of vanishing of the analytic function $\mu_j(\la)$ at $\la=\la_0$.

\subsection{Krein signature of real characteristic values}
The Krein signature of  a real characteristic value
is usually defined in terms of the restriction of an appropriate indefinite quadratic form to a root space, with the 
signature being positive, negative, or indefinite according to the type of the restricted form.  
Our goal is to give a useful definition of Krein signatures for real characteristic values
of quite general operator pencils, however at the beginning we will restrict ourselves to polynomial pencils with invertible leading 
coefficient, as this makes available the companion matrix method described in \S~\ref{sec:companion}.   In this case the relevant 
quantities can be defined as follows.

\begin{definition}[Krein indices and Krein signature]
Let $\LL$ be a selfadjoint polynomial operator pencil of degree $p$ with invertible leading coefficient $L_p$ acting in 
a Hilbert space $X$, and let $\la_0$ be an isolated real characteristic value of $\LL$.  Given a
root space $\mathcal{U}\subset X$ spanned by the vectors of a chain $\{u^{[0]},\dots,u^{[m-1]}\}$ from a canonical set
for $\la_0$, let $\mathbf{u}^{[j]}\in X^p$, $j=0,\dots,m-1$, denote the columns of the matrix $V$ defined
in the statement of Theorem~\ref{propGLR}, and let $W$ be  the $m\times m$ Gram matrix  
with elements $W_{jk}:=\langle \mathbf{u}^{[j]},\mathbf{u}^{[k]}\rangle$, in terms of the indefinite Hermitian quadratic form 
$\langle\cdot,\cdot\rangle$ given by \eqref{eq:indefinite-form}.  The number of positive (negative) eigenvalues of the 
Hermitian matrix $W$ is called the positive (negative) Krein index of the root space%
\footnote{The Krein indices $\kappa^\pm(\mathcal{U},\la_0)$ are in fact well-defined 
although the construction appears to depend on the specific choice of vectors making up the chain 
that spans the space $\mathcal{U}$.} 
$\mathcal{U}$ at $\la_0$ 
and is denoted $\kappa^+(\mathcal{U},\la_0)$ ($\kappa^-(\mathcal{U},\la_0)$).  
The sums of $\kappa^\pm(\mathcal{U},\la_0)$ over all root spaces $\mathcal{U}$ 
of $\la_0$ are called the positive and negative Krein indices of $\la_0$ and are denoted $\kappa^\pm(\la_0)$.  Finally, 
$\kappa(\mathcal{U},\la_0):=\kappa^+(\mathcal{U},\la_0)-\kappa^-(\mathcal{U},\la_0)$ is called
the Krein signature of the root space $\mathcal{U}$ for $\la_0$, and $\kappa(\la_0):=\kappa^+(\la_0)-\kappa^-(\la_0)$
is called the Krein signature of $\la_0$.  
\label{Kreindef}
\end{definition}

If $\kappa^+(\la_0)>0$ (respectively $\kappa^-(\la_0)>0$) we say that $\la_0$ has a non-trivial positive (respectively negative) 
signature, and if both $\kappa^+(\la_0)$ and $\kappa^-(\la_0)$ are positive we say that $\la_0$ has indefinite Krein signature.  
It is not obvious but true%
\footnote{It is part of the proof of Theorem~\ref{theorem:signature} below.  See \cite{KollarH}.} 
that the Krein signature of a single root space $\kappa(\mathcal{U},\la_0)$ is always either $1$, $-1$, or $0$ 
regardless of the dimension of $\mathcal{U}$.  The ``total'' Krein signature $\kappa(\la_0)$ of a real characteristic 
value $\la_0$ can, however, take any integer value by this definition. If $\la_0$ is a simple real characteristic value with characteristic 
vector $u\in X$, then the only root space $\mathcal{U}$ is spanned by $u$, and the corresponding (only) column of the matrix $V$ is
$\mathbf{u}=(u,\la_0 u,\dots,\la_0^{p-1}u)^\mathsf{T}\in X^p$, so the corresponding Gram matrix $W$ is a real scalar given by 
\begin{equation}
W=\langle \mathbf{u},\mathbf{u}\rangle=   (L_1u+2\la_0L_2u +\cdots +p\la_0^{p-1}L_pu,u)=
(\LL'(\la_0)u,u)
\label{eqW}
\end{equation}
as a simple calculation using \eqref{eq:indefinite-form} shows.  If furthermore $\LL$ is a linear selfadjoint pencil, $\LL(\la)=\la L_1+L_0$, 
then $\LL'(\la_0)=L_1$ and the Krein signature is given by $\kappa(\la_0)=\mathrm{sign}(L_1u,u)$, which coincides 
with the usual definition (see  \eqref{simpleKreinexamples}) found frequently in the literature.

This definition mirrors a recent approach \cite{Kap,KapProm,Prom} to spectral analysis of infinite-dimensional linearized 
Hamiltonian systems in which
the important information about a point of the spectrum is obtained from a finite-dimensional reduction called a \emph{Krein matrix}, 
an analogue of the Gram matrix $W$.  Koll{\'a}r and Pego \cite{KP} developed a rigorous perturbation theory of Krein signatures 
for finite systems of eigenvalues in the spirit of Kato \cite{Kato}, proving results that had previously belonged to 
the mathematical folklore in the field.

\subsection{Graphical Krein signature theory} 
One of the main messages of our paper is that the correct way to generalize the notion 
of Krein indices and signatures so that they apply to and are useful in the analysis of spectral problems involving operator pencils 
of non-polynomial type like that illustrated in Example~\ref{ex:DDE} is to eschew linear algebra in favor of analysis of 
eigenvalue curves $\mu=\mu(\la)$ 
of the problem \eqref{Lmu} in the vicinity of their roots.  This approach is attractive even in cases where 
companion matrix methods  explained 
in \S\ref{sec:companion} suffice to define the relevant quantities.  We therefore begin by formulating such 
a ``graphical'' definition of Krein signatures.
\begin{definition}[Graphical Krein indices and signatures]
Let $\LL(\la)$ be a selfadjoint pencil that is either a holomorphic family of type (A) with compact resolvent or 
of Fredholm form $\II-\BB(\la)$ with $\BB(\la)$ compact and injective, and assume that $\LL$ 
has an isolated real characteristic value $\la_0$. Let $\mu=\mu(\la)$ be one of the real analytic eigenvalue branches 
of the problem \eq{Lmu} with $\mu^{(n)}(\la_0)=0$ for $n=0,1,\dots,m-1$, while $\mu^{(m)}(\la_0)\neq 0$.   
Let $\eta(\mu):=\mathrm{sign}(\mu^{(m)}(\la_0))=\pm 1$. Then the quantities 
\begin{equation}
\kappa^\pm_\mathrm{g}(\mu,\la_0):=\begin{cases}
\tfrac{1}{2}m,&\quad \text{for $m$ even}\\
\tfrac{1}{2}(m\pm\eta(\mu)),&\quad\text{for $m$ odd}
\end{cases}
\end{equation}
are called the positive and negative graphical Krein indices of the eigenvalue branch $\mu=\mu(\la)$ corresponding 
to the characteristic value $\la_0$.  The sums of $\kappa^\pm_\mathrm{g}(\mu,\la_0)$ over all eigenvalue branches 
crossing at $(\la,\mu)=(\la_0,0)$ are called the positive and negative graphical Krein indices of $\la_0$ and are denoted
$\kappa^\pm_\mathrm{g}(\la_0)$. Finally,  $\kappa_\mathrm{g}(\mu,\la_0):=\kappa_\mathrm{g}^+
(\mu,\la_0)-\kappa_\mathrm{g}^-(\mu,\la_0)$ is called the graphical Krein signature of the eigenvalue branch $\mu=\mu(\la)$ 
vanishing at $\la_0$, and $\kappa_\mathrm{g}(\la_0):=\kappa_\mathrm{g}^+(\la_0)-\kappa_\mathrm{g}^-(\la_0)$ is called 
the graphical Krein signature of $\la_0$.
\label{graphsig}
\end{definition}

Note that it follows directly from the definition of the graphical Krein indices $\kappa^\pm_\mathrm{g}(\mu,\la_0)$ 
that the graphical Krein signature 
$\kappa_\mathrm{g}(\mu,\la)$ necessarily  takes one of the three values $1$, $-1$, or $0$.  The concept of \emph{sign characteristics} 
of operator pencils \cite[Chapters 7 and 12]{GLR} is closely related to our definition; here each root space of $\LL$ corresponding 
to a characteristic value $\la_0$ is associated with the sign of $\tilde{\mu}(\la_0) \ne 0$, where 
$\mu(\la) = (\la -\la_0)^{m-1} \tilde{\mu}(\la)$ 
near $\la=\la_0$. See also \cite[Theorem 5.11.2]{GLR}. The relationship between sign characteristics and Krein indices can be expressed 
by an algebraic formula \cite[p.~77, equation (5.2.4)]{GLR}. 

Due to this correspondence, the Krein signature as given by Definition~\ref{Kreindef}
and the graphical Krein signature as given by Definition~\ref{graphsig} are well-known
to be connected in the theory of matrix pencils.
The agreement between the Krein signature and graphical Krein signature for a simple real characteristic value $\la_0$ 
of a selfadjoint pencil $\LL(\la)$ can be deduced easily. In the simple case, the graphical Krein signature $\kappa_\mathrm{g}(\la_0)$ is simply the sign of the derivative $\mu'(\la_0)$ of the unique eigenvalue branch $\mu=\mu(\la)$ vanishing to first order at $\la_0$ according to Theorem~\ref{multiplicities}.  On the other hand, by Definition~\ref{Kreindef} the Krein signature  $\kappa(\la_0)$ is given by the sign of 
$W$ defined by \eq{eqW}, i.e., $\kappa(\la_0)=\sign(\LL'(\la_0)u,u)$. Differentiation of the characteristic value problem \eq{Lmu} 
with respect $\la$ at $\la = \la_0$ and $\mu(\la_0) = 0$ gives 
\begin{equation}
\left( \LL'(\la_0) - \mu'(\la_0)\right) u(\la_0) + \left( \LL(\la_0) - \mu(\la_0)\right) u'(\la_0) = 0.
\label{eq:deriv}
\end{equation}
Taking the scalar product of \eq{eq:deriv} with $u=u(\la_0)$ and using self-adjointness of $\LL(\la_0)$ 
immediately yields the agreement of the signatures since 
$(\LL'(\la_0) u, u) = \mu'(\la_0) (u,u)$.
In the general case the connection is described in detail in \cite[p.~260, Theorem 12.5.2]{GLR} (see also \cite[\S10.5]{GLRmatrix}). 
The long and technical proof in these works for matrix pencils is essentially the same as for operator pencils. 
A different  proof of a more perturbative nature can be found in \cite[Lemma 9]{Kap}. 
See also \cite[Lemma 5.3]{CP} and \cite[Theorem~1]{VP} for the same result in 
the special case a  generalized eigenvalue problem (linear pencil). Yet another  proof based on the Frobenius rule 
can be found in \cite[Theorem~4.2]{KollarH}. 

\begin{theorem}[Graphical nature of Krein indices and signatures]
Let $\LL$ be a self-adjoint polynomial operator pencil of degree $p$ with invertible leading coefficient $L_p$, 
and assume also that $\LL$ is a holomorphic family of type (A) with compact resolvent.  Let $\la_0$ be an isolated real 
characteristic value of $\LL$.
Let $\mathcal{U}=\mathrm{span}\{u^{[0]},u^{[1]},\dots,u^{[m-1]}\}$ be the root space for $\la_0$
corresponding (according to Theorem~\ref{multiplicities}) to the analytic eigenvalue branch $\mu=\mu(\la)$ of the associated eigenvalue
problem \eqref{Lmu}.
Then
\begin{equation}
\kappa^\pm(\mathcal{U},\la_0)=\kappa^\pm_\mathrm{g}(\mu,\la_0),
\end{equation}
from which it follows that
\begin{equation}
\kappa(\mathcal{U},\la_0)=\kappa_\mathrm{g}(\mu,\la_0)
\end{equation}
and summing over all root spaces $\mathcal{U}$ and corresponding eigenvalue branches $\mu$,
\begin{equation}
\kappa^\pm(\la_0)=\kappa_\mathrm{g}^\pm(\la_0)\quad\text{and}\quad
\kappa(\la_0)=\kappa_\mathrm{g}(\la_0).
\end{equation}
\label{theorem:signature}
\end{theorem}

In particular, this result implies that for simple characteristic values $\la_0$, the Krein signature can be calculated easily 
from the formula \eqref{simpleKrein2examples}. This result allows us to extend the notion of Krein signature in a very 
natural way to problems arising in the theory of stability of nonlinear waves that cannot easily 
be formulated as spectral problems for polynomial pencils as the spectral parameter $\la$ enters into the problem in 
two different, transcendentally related, ways (as in Example~\ref{ex:DDE}).  
Finally, the graphical Krein signature is the one that is most easily detected by a natural generalization
of the Evans function, the subject that we take up next in \S\ref{s:Evans}.  From now on, when we refer to total 
Krein indices and Krein signatures of a real characteristic value $\la_0$, we will always mean the 
(more widely applicable) graphical quantities of Definition~\ref{graphsig}, and omit the unnecessary subscript ``g''.

\section{Evans Functions\label{s:Evans}}
Consider a dynamical system linearized about an equilibrium in the form (compare to \eq{dydtexamples})
\begin{equation}
\frac{du}{dt} = Au.
\label{Ax}
\end{equation}
In the infinite-dimensional context typical in nonlinear wave theory, the linearized system
\eq{Ax} is usually a partial differential equation and $u$ is a vector in some Hilbert space $X$.  
For waves in $1+1$ dimensions, $X$ is a space of functions of a real spatial variable $x$, and $A$ can be thought of as a linear differential
operator acting in $X$. The key spectral problem in the stability analysis of the equilibrium is (compare to \eq{JL})
\begin{equation}
-i\la u_0= A u_0,
\label{Au}
\end{equation}
as each eigenvector $u_0=u_0(\cdot) \in X$ of \eq{Au} generates a solution of \eq{Ax} in the separated form 
$u=u(x,t)=e^{-i\la t} u_0(x)$. 
Values of $\nu=-i\la\in\sigma(A)$ having a positive real part imply the existence of exponentially growing modes in \eq{Ax}  
and thus linear instability of the equilibrium of the original (nonlinear) dynamical system. 

Although numerical methods to study $\sigma(A)$ based on spatial discretization or some other finite dimensional truncation 
are easy to implement, they may fail to detect the full extent of the unstable spectrum or they may introduce spurious $\la$, 
particularly in the vicinity of any continuous spectrum of $A$. In \cite{Evans1}--\cite{Evans4} a more robust and reliable 
numerical method was developed based on a new concept now called the \emph{Evans function}, and this method was 
successfully applied to study the stability of neural impulses.  Since its introduction, the Evans function has become a popular 
tool for the detection of stability of various types of waves in many
applications including fluid mechanics, condensed matter physics, combustion theory, etc. 

\subsection{Typical construction and use of Evans functions}
\label{sec:typical-construction}
We begin with a standard definition.
\begin{definition}[Evans functions]
\label{def:Evans}
Let $\LL=\LL(\la)$ be an operator pencil.
An analytic function $D:\Omega\subset\mathbb{C}\to\mathbb{C}$ whose roots in $\Omega$ coincide exactly
with isolated characteristic values $\la=\la_0$ of the spectral problem \eq{Lu},
and that vanishes at each such point to precisely the order of the algebraic multiplicity of the characteristic value 
is called an Evans function for $\LL$ (on $\Omega$).
\end{definition}

It is common to refer to ``the'' Evans function for a given spectral problem, and this
usually implies a particular kind of construction appropriate for problems of the special form \eq{Au} that we will describe briefly below. 
However the key properties of analyticity and vanishing on the discrete spectrum with the correct multiplicity are shared by many 
other functions and as the freedom to choose among them can be useful in applications, we prefer to  keep the terminology as 
general as possible.

Let us now describe the classical construction of ``the'' Evans function for \eq{Au} in the case that
$A$ is a scalar ordinary differential operator of order $k$ acting in $X$.  To begin with, the spectral problem
\eq{Au} is rewritten as a (nonautonomous, in the usual case that $A$ has non-constant coefficients) first-order system%
\footnote{There are of course many ways to rewrite a single higher-order linear differential equation as a first-order system, 
and if care is not taken key symmetries of the original equation can be lost in the process.  This is true even in the 
case of \eq{JL}, although for such problems with Hamiltonian symmetry Bridges and Derks \cite{Bridges} have shown  
how some of this structure can be retained.}
\begin{equation}
dv / dx = B(x,\la) v,\qquad v \in X^{k}\, ,
\label{Bv}
\end{equation} 
where $B(x,\la)$ is a $k\times k$ matrix-valued function assumed to take finite limits $B_\pm(\la)$ as $x\to\pm\infty$.  
We consider \eqref{Bv} along with the constant-coefficient ``asymptotic systems''
\begin{equation}
dv / dx=B_\pm(\la)v.
\label{eq:asymp-systems}
\end{equation}
The first-order system \eq{Bv} has a 
nonzero solution $v(x)$ decaying as $|x| \rightarrow \infty$, and hence \eq{Au} has a nontrivial solution $u\in X$, 
exactly for those values of $\la$ for which the forward evolution $W_\mathrm{u}^{-\infty}(x) $  
of the unstable manifold of  the zero equilibrium of  \eqref{eq:asymp-systems} for $B_-(\la)$ has a 
nontrivial intersection
with the backward evolution 
$W_\mathrm{s}^{+\infty}(x)$ of the stable manifold of the zero equilibrium of  \eqref{eq:asymp-systems} for $B_+(\la)$. 
To properly define the evolutes $W_\mathrm{u}^{-\infty}(x)$ and $W_\mathrm{s}^{+\infty}(x)$ for a common 
value of $x\in\mathbb{R}$ requires that both asymptotic systems \eqref{eq:asymp-systems} are hyperbolic with 
``exponential dichotomy'' (that is, the eigenvalues of $B_\pm(\la)$ are bounded away from the imaginary axis as $\la$
varies in the region of interest), and for the existence of isolated characteristic values $\la$ one usually requires 
complementarity  (in $\mathbb{C}^k$) of the dimensions of the stable and the unstable manifolds at $x = \pm \infty$.
The traditional Evans function $D(\la)$ detecting transversality of the intersection of the evolutes can then be
expressed as a $k\times k$ Wronskian determinant whose columns include $p_-<k$ vectors spanning $W_\mathrm{u}^{-\infty}(x)$ 
and $p_+:=k-p_-$ vectors spanning $W_\mathrm{s}^{+\infty}(x)$. 
For coefficient matrices $B(x,\la)$ of trace zero, Abel's Theorem implies that $D(\la)$ is independent of the value of $x\in\mathbb{R}$ 
at which the Wronskian is computed.  For systems with nonzero but smooth trace, \emph{any} value of $x\in\mathbb{R}$ can 
be chosen but different values of $x$ lead to different Evans functions $D(\la)$ the ratio of any two of which is 
an analytic non-vanishing function of $\la$.  
Usually in such situations the unimportant dependence of $D(\la)$ on $x\in\mathbb{R}$ is explicitly removed by an appropriate 
exponential normalization \cite{PegoWarchall}. This type of construction has been given a topological interpretation
by Alexander {\it et al.}~\cite{AGJ}, who related $D(\la)$ to the Chern number of a tangential fiber bundle. 
Pego and Weinstein \cite{PegoWeinstein} used this definition of $D(\la)$ to establish stability properties of solitary waves 
in nonlinear systems including the Korteweg-deVries, Benjamin-Bona-Mahoney, and Boussinesq equations
and also pointed out a connection between $D(\la)$ and the \emph{transmission coefficient} associated with \emph{Jost solutions} 
in scattering theory (see also \cite{Yanagida}). 
This type of construction has also been extended to problems in multiple dimensions where $A$ is a partial differential operator; 
in these cases special steps must be taken to ensure the analyticity of the resulting Evans function.  In low dimensions 
a useful generalization of the Wronskian can be constructed using exterior products \cite{PegoWarchall}, and a robust numerical algorithm
has also been developed \cite{ZumbrunNum} for this purpose that uses continuous orthogonalization to obtain an analytic Evans function.

The main reason for insisting upon the key property of analyticity of an Evans function is that the presence of 
characteristic values, as zeros of the analytic function $D(\la)$, can be detected \emph{from a distance} in the 
complex plane by means of the Argument Principle.  Indeed, to determine the number of characteristic values 
(counting algebraic multiplicity) in a two-dimensional simply-connected bounded domain $\Omega$ of the complex plane, 
it suffices to evaluate an Evans function only on the one-dimensional boundary $\partial\Omega$ and to compute the 
winding number of the phase of $D(\la)$ as $\la$ traverses this curve.  If the winding number is zero, one knows 
that there are no characteristic values in $\Omega$.  Otherwise, $\Omega$ can be split into subregions and the process 
repeated to further confine the discrete spectrum.  From the computational point of view this is both more robust and 
far less expensive than a fully two-dimensional search of the $\la$-plane for characteristic values.
Searching for characteristic values by computing the winding number of an Evans function $D(\la)$ along closed curves 
is a particularly useful approach to the spectral problem \eqref{Lu} if one knows, say by a variational argument, a bound 
of the form $|\la|<M$ for the characteristic values of the problem.  
However, even in cases when such a bound is unavailable, it turns out that if the Krein signature 
of certain characteristic values (usually the real ones) is known or can be computed it is sometimes possible to justify numerics 
and to reduce computational costs significantly \cite{KP}. This suggests that it would be particularly useful if it were possible to extract 
the Krein signature of a real characteristic value directly from an Evans function.  

\subsection{Evans functions and Krein signatures}
By definition, an Evans function $D(\la)$ detects characteristic values of \eq{Lu} but 
the usual constructions produce Evans functions that do not provide or contain any information 
about the corresponding root vectors. To better understand the problem, first consider a simple finite-dimensional case of 
the linearized Hamiltonian spectral problem \eq{JL} and the related spectral problem \eq{LKeq} 
for an equivalent linear pencil $\LL(\la)=L-\la K$.  In this case the most obvious definition of an Evans function is to set
\begin{equation}
D(\la) := \det(\LL(\la))=\det(L-\lambda K).
\label{Edef}
\end{equation}
Note that as $L$ and $K$ are Hermitian matrices, $D$ is real-valued as a function of $\lambda\in\mathbb{R}$ but of course
it has meaning for complex $\lambda$ as well.
It might seem natural to expect that the Evans function also encodes the Krein signature of 
the characteristic values in some sense.  It has frequently been thought that under 
some circumstances at least the sign of the slope $D'(\lambda_0)$ may be correlated 
with the Krein signature $\kappa(\la_0)$.  To explore this idea,
we present the following two examples illustrating the graphical 
interpretation of the Krein signature as explained in Definition~\ref{graphsig} and Theorem~\ref{theorem:signature},
and comparing with the slope of $D(\la)$ at its roots.

\begin{example}\label{ex:KreinEvansMatch}
Let $L$ and $J$ be given by  (compare to Example~\ref{ex:indefinite})
$$
L=\begin{pmatrix}2 & 0 & 0 & 0\\0 & -1 & 0 & 0\\0 & 0 & 1 & 0\\0 & 0 &0 & -2
\end{pmatrix},\quad
J=\begin{pmatrix}0&0&2&0\\0&0&0&1\\-2&0&0&0\\0&-1&0&0\end{pmatrix}.
$$
The four real eigenvalue branches $\mu=\mu(\la)$ of the associated spectral problem \eqref{Lmu} for this example 
are plotted along with the graph of $D(\la)$ as given by \eqref{Edef} in the left-hand panel of Fig.~\ref{fig2}.   
\begin{figure}[htp]
\centering
\includegraphics{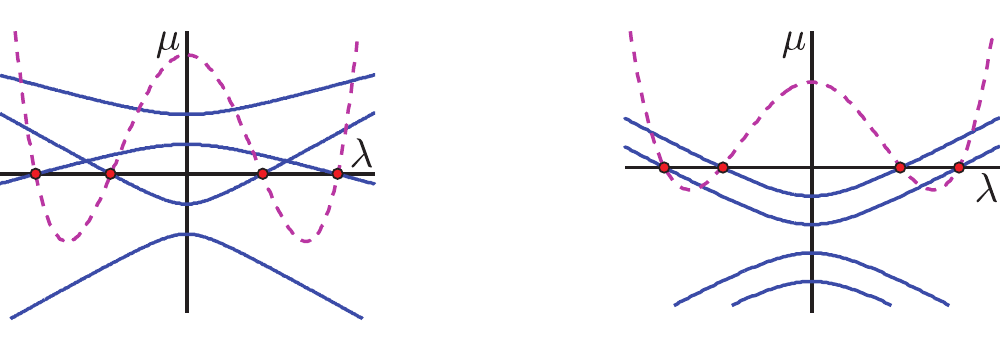}
\caption[]{%
The four eigenvalue branches $\mu=\mu(\la)$ of the linear pencil $\LL(\la)$ for
Examples~\ref{ex:KreinEvansMatch}--\ref{ex:KreinEvansMisMatch} 
are plotted with solid curves. The real intercepts of the branches  with $\mu = 0$ 
corresponding to real characteristic values of $\LL= \LL(\la) = L - \la K$, $K := (iJ)^{-1}$, are indicated with dots.
The dashed curve is the graph of the Evans function (arbitrary units) $D(\la) = \det(\LL(\la))$.}
\label{fig2}
\end{figure}
The Krein signature of each of the real characteristic values $\la=\la_0$ can be easily calculated graphically
using Definition~\ref{graphsig} and Theorem~\ref{theorem:signature}, or equivalently as all four characteristic 
values are simple, via \eqref{simpleKrein2examples}. 
In this example it is easy to see that for each of the four real characteristic values, $\kappa(\la_0)=-\mathrm{sign}(D'(\la_0))$, 
and therefore the sign of the slope $D'(\la)$ of $D(\la)$  at its zeros is strongly correlated with the Krein signature.
\end{example}

\begin{example}\label{ex:KreinEvansMisMatch}
Let $L$ and $J$ be given by (compare to Example~\ref{ex:posdef}) 
$$
L=\begin{pmatrix}-0.5 & 0 & 0 & 0\\0 & -1 & 0 & 0\\0 & 0 & -1.5 & 0\\0 & 0 & 0 & -2\end{pmatrix}, \quad
J=\begin{pmatrix}0 & 0 & 1 & 0\\0 & 0 & 0 & 1\\-1 & 0 & 0 & 0\\0 & -1 & 0 & 0\end{pmatrix}.
$$
This case is illustrated in the right-hand panel of Fig.~\ref{fig2}.
There are again four simple real characteristic values of $\LL$ 
(here consistent with the fact that $L$ is negative definite), and by using the graphical 
definition \eqref{simpleKrein2examples} we see that the positive (negative) characteristic values each have 
Krein signature $\kappa(\la_0)=+1$ ($\kappa(\la_0)=-1$).  By contrast with Example~\ref{ex:KreinEvansMatch}, 
in this case it is obvious that the slope of $D(\la)$ 
at its roots is positively correlated with the Krein signature of half of the characteristic values and 
is negatively correlated with the Krein signature of the other half.  Hence $D'(\la)$ does not detect the Krein 
signature correctly as the Evans function fails to capture the attachment of the characteristic values
to different eigenvalue branches $\mu=\mu(\la)$. 
\end{example}

These simple examples illustrate the state of understanding in much of the literature: it just is not 
clear under which circumstances one may expect derivatives of an Evans function to encode Krein signatures of characteristic values.

\subsection{Evans-Krein functions}
In the general setting of operator pencils we now address
the question posed in \cite{KKS} (in the context of linearized Hamiltonians) as to
whether it is possible to determine Krein signatures from an Evans function.  In fact we 
can answer this question in the affirmative by showing how a simple generalization of the notion of an Evans function 
correctly detects the Krein signature in every case.
Our approach provides a simple way to calculate  Krein signatures of real characteristic values 
of selfadjoint operator pencils by an elementary augmentation of virtually any existing 
Evans function evaluation code at almost no additional programming cost. 
The main idea is to build into an Evans function information about the individual eigenvalue branches $\mu=\mu(\la)$ 
to which the real characteristic values are associated as roots.  In fact, this simply amounts to bringing in $\mu$ 
as an additional parameter.  We now introduce such an appropriate generalization of an Evans function, which we 
call an \emph{Evans-Krein function}.

\begin{definition}[Evans-Krein functions]
\label{def:Evans-Krein}
Let $\LL=\LL(\la)$ be an operator pencil acting on a space $X$, and consider the related pencil 
$\mathcal{K}=\mathcal{K}(\la;\mu):=\LL(\la)-\mu \II$,
where $\mu$ is a sufficiently small real parameter.  An Evans function $E(\la;\mu)$ (in the sense of Definition~\ref{def:Evans}) for 
the $\mu$-parametrized pencil $\mathcal{K}$ is called an Evans-Krein function for $\LL$.
\end{definition}

Under suitable general conditions on $\LL$, an Evans-Krein function is not only analytic for $\la\in\Omega\subset\mathbb{C}$ 
for each small $\mu$, but also will be analytic in both $\la$ and $\mu$ on a set of the 
form $\Omega\times\{|\mu|<\delta\}\subset\mathbb{C}^2$ 
for some $\delta>0$. Upon setting $\mu=0$ we see that if $E(\la;\mu)$ is an Evans-Krein function for $\LL$, then $D(\la):=E(\la;0)$ is 
an Evans function for $\LL$, with all of the coincident properties as given in Definition~\ref{def:Evans}.  
It is also clear that if $\mu=\mu(\la)$ is a real eigenvalue branch of the problem \eqref{Lmu}, then $E(\la;\mu(\la))$ 
vanishes identically for all $\la\in\mathbb{C}$ for which the branch 
is defined and $\mu(\la)$ is sufficiently small.  Moreover, if $\la_0\in\mathbb{R}$ is a real characteristic value of $\LL$ 
of geometric multiplicity $k\ge 1$, and if $\mu_1(\la),\dots,\mu_k(\la)$ are the corresponding eigenvalue branches 
of \eqref{Lmu} crossing at the point $(\la,\mu)=(\la_0,0)$, 
then in a $\mathbb{C}^2$ neighborhood of this point, an arbitrary Evans-Krein function can be locally represented in the normal form
\begin{equation}
E(\la;\mu)=C(\la;\mu)\prod_{j=1}^k(\mu-\mu_j(\la))
\label{eq:EK-normal}
\end{equation}
where $C(\la;\mu)$ is an analytic function with $C(\la_0;0)\neq 0$.

Construction of an Evans-Krein function in practice can be very similar to that of an Evans function in the usual sense.  
For a matrix pencil $\LL$  one can set
\begin{equation}
E(\la;\mu):=\det(\LL(\la)-\mu\mathbb{I}),\quad\text{and}\quad D(\la):=E(\la;0).
\label{DE}
\end{equation}
This definition is also suitable if $\LL$ is an operator pencil that is a trace-class perturbation of the identity, in which 
case the determinants above are to be understood as Fredholm determinants.
In the case that the spectral problem \eqref{Lu} involves differential operators and one
is led to consider the traditional construction of an Evans function (as described in \S\ref{sec:typical-construction}) via a Wronskian of 
subspaces of solutions decaying at the opposite spatial infinities, \emph{exactly the same construction gives an 
Evans-Krein function if the pencil $\LL$ is merely perturbed into $\mathcal{K}$}, because for sufficiently small $\mu$ 
it is easy to guarantee the persistence of hyperbolicity of the asymptotic systems \eqref{eq:asymp-systems} 
as well as the complementarity of the stable and unstable subspace dimensions.

\subsection{Calculation of Krein signatures from an Evans-Krein function}
Let $E(\la;\mu)$ be an Evans-Krein function for a selfadjoint operator pencil $\LL$ for which eigenvalue 
branches $\mu=\mu(\la)$ of the associated problem \eqref{Lmu} are sufficiently smooth, and let $\la=\la_0$ 
be an isolated real characteristic value of $\LL$.  By differentiating the identity $E(\la;\mu(\la))=0$ 
with respect to $\la$ along a particular smooth branch $\mu=\mu(\la)$ for which $\mu(\la_0)=0$ 
and then evaluating at $\la=\la_0$ we easily obtain
\begin{equation}
E_\la(\la_0;0)+E_\mu(\la_0;0)\mu'(\la_0)=0.
\label{ElaEq}
\end{equation}
According to the normal form \eqref{eq:EK-normal}, the partial derivatives $E_\la(\la_0;0)$ and  $E_\mu(\la_0;0)$ 
are both nonzero under the assumption that the characteristic value $\la_0$ is simple, which implies both that the 
number of eigenvalue branches $\mu=\mu(\la)$ of \eqref{Lmu} crossing the $\mu=0$ axis at $\la=\la_0$ is exactly $k=1$,
and also that $\la_0$ is a simple root of the branch function $\mu(\la)$ (see \S\ref{s:Krein} for details). 
Therefore, in this case (corresponding to a simple real characteristic value $\la_0$) we have
\begin{equation}
\mu'(\la_0) = -\frac{E_{\la}(\la_0;0)}{E_{\mu}(\la_0;0)},
\label{muprime}
\end{equation}
and then comparing with the graphical formula \eqref{simpleKrein2examples} for the Krein signature of such a characteristic 
value we learn that
\begin{equation}
\kappa(\la_0)=\mathrm{sign}(\mu'(\la_0))=-\mathrm{sign}\left[\frac{E_\la(\la_0;0)}{E_\mu(\la_0;0)}\right].
\label{kreinkappa}
\end{equation}
Recalling that $E(\la;0)$ is necessarily an Evans function $D(\la)$ for $\LL$, this formula can be re-written in the equivalent form
\begin{equation}
\kappa(\la_0)=-\mathrm{sign}\left[\frac{D'(\la_0)}{E_\mu(\la_0;0)}\right].
\label{simpleKappa}
\end{equation}
This formula%
\footnote{The specific Evans function $D(\la)$ in the numerator  of course depends on 
the Evans-Krein function $E(\la;\mu)$ from which it arises upon setting $\mu=0$.  This has some immediate 
implications; for example if one considers a standard eigenvalue problem $Lu=\la u$ for a selfadjoint operator $L$, 
upon introducing the selfadjoint pencil $\LL(\la):=L-\la \II$ 
and the corresponding $\mu$-perturbation $\mathcal{K}(\la;\mu):=L-\la \II-\mu \II= L-(\la+\mu)\II$ one sees that
$E_\la(\la_0;0)=E_\mu(\la_0;0)$ for all real characteristic values (eigenvalues of $L$) $\la_0$.  Hence the 
Krein signature of each simple characteristic value is necessarily positive, even though by Rolle's 
Theorem $D'(\la)=E_\la(\la;0)$ will have sign changes.} 
shows how the Krein signature of a simple real characteristic value $\la_0$
indeed relates to the derivative of an Evans function; the reason that the relationship seems ambiguous as shown 
in Examples~\ref{ex:KreinEvansMatch}--\ref{ex:KreinEvansMisMatch} is that to obtain agreement one must take 
into account an additional factor, namely the sign of $E_\mu(\la_0;0)$ which can be different for different characteristic values $\la_0$.
The fact that it is only the \emph{sign} of $E_\mu(\la_0;0)$ that enters means that there is substantial freedom to generalize 
the way that an Evans-Krein function can depend on the parameter $\mu$ without changing the Krein signature formula 
\eq{simpleKappa}.  Such generalizations will be discussed in \S\ref{s:gEK} below.

An Evans-Krein function $E(\la;\mu)$ can also be used to provide information about the positive and negative Krein indices of 
the eigenvalue branch(es) passing through $\mu=0$ at an isolated non-simple real characteristic value $\la_0$ of a selfadjoint 
operator pencil.  According to the theory presented in \S\ref{s:Krein} (see Definition~\ref{graphsig}), 
the (graphical) Krein indices $\kappa^\pm_\mathrm{g}(\mu,\la_0)$ associated with the eigenvalue branch $\mu=\mu(\la)$ 
are known in the case of odd algebraic multiplicity (order of vanishing of $\mu(\la)$ at $\la_0$) 
$m$ once the sign of $\mu^{(m)}(\la_0)\neq 0$ is known (if $m$ is even there is nothing to calculate).  But
$\mu^{(m)}(\la_0)$ can also be expressed in terms of nonzero partial derivatives of an Evans-Krein function $E(\la;\mu)$ 
at the point $(\la,\mu)=(\la_0,0)$.  The main idea is to calculate sufficiently many derivatives of the 
identity $E(\la;\mu(\la))=0$ at the point $\la=\la_0$.  
By repeated application of the chain rule, one finds that
\begin{multline}
\frac{d^r}{d\la^r}E(\la;\mu(\la))=\sum_{n=0}^r\binom{r}{n}\frac{\partial^rE}{\partial\mu^n\partial\la^{r-n}}(\la;\mu(\la))\mu'(\la)^n \\
{}+ 
\sum_{p=1}^{r-1}\;\;\;\;\;\sum_{n+s=p}\;\;\mathop{\sum_{d_1\le\cdots \le d_n}}_{d_1+\cdots+d_n=p}\hspace{-0.2in}
\beta_{r,n,s}(d_1,\dots,d_n)\frac{\partial^{r-s}E}{\partial\mu^n\partial\la^{r-p}}(\la;\mu(\la))\prod_{j=1}^n\mu^{(d_j)}(\la),
\label{eq:chain-rule}
\end{multline}
where $n$, $s$, and $d_j$ for $j=1,\dots,n$ are positive integers and $\beta_{r,n,s}(d_1,\dots,d_n)$ 
are certain complicated coefficients.  
Many of the partial derivatives of $E$ can be evaluated at $\la=\la_0$ and $\mu(\la_0)=0$ by using the normal form \eqref{eq:EK-normal}.  
While one can develop a general theory, to keep things simple we just present two representative examples.

\begin{example}[Real characteristic values of geometric multiplicity $k=1$.]
For a real characteristic value $\la_0$ of $\LL$ having geometric multiplicity $k=1$ but arbitrary finite algebraic 
multiplicity $m\ge 1$, there is a unique eigenvalue branch $\mu=\mu_1(\la)$
(analytic under suitable assumptions) of the problem \eqref{Lmu} passing through the point $(\la,\mu)=(\la_0,0)$, 
and we have $\mu_1(\la_0)=\mu_1'(\la_0)=\cdots=\mu_1^{(m-1)}(\la_0)=0$
but $\mu_1^{(m)}(\la_0)\neq 0$.  If $m$ is even, by definition we have $\kappa_\mathrm{g}^\pm(\mu_1,\la_0)=m/2$, 
while for $m$ odd we need to calculate $\mu_1^{(m)}(\la_0)$ to determine
the Krein indices.  But this is easily obtained from an Evans-Krein function $E(\la;\mu)$ by calculating $m$ 
derivatives of the identity $E(\la;\mu_1(\la))=0$ using \eqref{eq:chain-rule} and evaluating at $\la=\la_0$.
Only two terms survive for $\la=\la_0$ due to the fact that $\mu_1(\la)$ vanishes there to order $m$.  Indeed,
\begin{equation}
\left.\frac{d^m}{d\la^m}E(\la;\mu_1(\la))\right|_{\la=\la_0}=\frac{\partial^mE}{\partial\la^m}(\la_0;0) +
\frac{\partial E}{\partial \mu}(\la_0;0)\mu_1^{(m)}(\la_0)=0
\end{equation}
and since $E_\mu(\la_0;0)\neq 0$ as follows from the fact that $k=1$, we obtain
\begin{equation}
\mu_1^{(m)}(\la_0)=-\frac{1}{E_\mu(\la_0;0)}\frac{\partial^mE}{\partial\la^m}(\la_0;0).
\end{equation}
Note that $\partial^mE/\partial\la^m(\la_0;0)\neq 0$ because this is the $m$-th derivative of an Evans function $D(\la)$ 
at a characteristic value of algebraic multiplicity $m$.
\end{example}

\begin{example}[Semisimple real characteristic values.]
If $\la_0$ is a semisimple real characteristic value of a selfadjoint operator pencil $\LL$ of geometric and algebraic multiplicity $k\ge 1$, 
then there are $k$ (analytic under suitable assumptions) eigenvalue branches $\mu=\mu_j(\la)$, $j=1,\dots,k$, 
of the problem \eqref{Lmu} each of which has a simple root at $\la=\la_0$.  To determine the Krein indices it 
suffices to express $\mu_j'(\la_0)$ in terms of an Evans-Krein function $E(\la;\mu)$ for $j=1,\dots,k$.  
Now, from the normal form \eqref{eq:EK-normal} one can easily see that in this case $E(\la_0;\mu)=C(\la_0;\mu)\mu^k$ 
with $C(\la_0;0)\neq 0$ so in particular $E_\mu(\la_0;0)=0$ unless $k=1$ (and then $\la_0$ is simple), 
implying that in general the  method of differentiation once along each branch individually
and evaluating at $\la=\la_0$ as in the calculation leading to \eqref{muprime} will fail.

However, using the normal form \eqref{eq:EK-normal} and the first-order vanishing of all $k$ of the branches 
crossing at $(\la_0,0)$ one can easily show that
\begin{equation}
\frac{\partial^{n+j}E}{\partial\mu^n\partial\la^j}(\la_0;0)=0,\quad n+j<k.
\label{eq:partial-derivatives-not-enough}
\end{equation}
If we choose a branch $\mu=\mu_j(\la)$ and differentiate the identity $E(\la;\mu_j(\la))=0$
exactly $k$ times with respect to $\la$ at $\la=\la_0$, then with $r=k$ in \eqref{eq:chain-rule}
the identities \eqref{eq:partial-derivatives-not-enough} guarantee that all of the terms on the second line vanish identically.  
Therefore the $k$ values $z=\mu_j'(\la_0)$
are determined as the roots of the $k$-th order polynomial equation
\begin{equation}
\sum_{n=0}^k\binom{k}{n}\frac{\partial^kE}{\partial\mu^n\partial\la^{k-n}}(\la_0;0)z^n=0.
\label{eq:polynomial-equation}
\end{equation}
This equation is genuinely $k$-th order because the normal form \eqref{eq:EK-normal} and our assumptions on $\mu_j(\la)$ imply that
\begin{equation}
\frac{\partial^kE}{\partial\mu^n\partial\la^{k-n}}(\la_0;0)=n!C(\la_0;0)\frac{d^{k-n}\gamma_n}{d\la^{k-n}}(\la_0).
\end{equation}
Here, $\gamma_p(\la)$ is defined as the coefficient of $\mu^p$ in the product
\begin{equation}
\prod_{j=1}^k(\mu-\mu_j(\la))=\sum_{p=0}^k\gamma_p(\la)\mu^p.
\end{equation}
Since $\gamma_k(\la)\equiv 1$, the coefficient of $z^k$ in \eqref{eq:polynomial-equation} is $k!C(\la_0;0)\neq 0$.  
Likewise, since $\gamma_0(\la)=(-1)^k\mu_1(\la)\cdots\mu_k(\la)$, the constant term in the polynomial on 
the left-hand side of \eqref{eq:polynomial-equation} is $(-1)^kC(\la_0;0)\mu_1'(\la_0)\cdots\mu_k'(\la_0)\neq 0$, 
which implies that all of the roots of \eqref{eq:polynomial-equation} are nonzero.  That all $k$ of the roots are real numbers 
is less obvious but necessarily true.
\end{example}

In summary, as was observed in  \cite{Kap, KapProm}, an Evans function $D(\la)$ alone cannot explain the origin of 
Krein indices and signatures of characteristic values. The geometrical reason is simple: $D(\la)$ is proportional to the 
\emph{product of all eigenvalue branches} $\mu(\la)$ of the problem \eqref{Lmu}, and no amount of differentiation 
with respect to $\la$ will separate the individual branches.  
However, the branches can indeed be separated with the help of derivatives of an Evans-Krein function 
$E(\la; \mu)$ with respect to $\mu$ at $(\la,\mu)=(\la_0,0)$ that encode the local behavior of each branch 
close to the characteristic value. 
A conceptually related approach was also used in studies of stability of $N$-pulse solutions \cite{AGJS, MKSJ, Sandstede1998}, where 
rather than analyzing an Evans function as a product of eigenvalue branches the authors study a matrix whose determinant is the
Evans function.  This introduces a variety of different functions (e.g. the matrix elements) that can contain 
extra information compared to a single complex-valued Evans function.

\subsection{Further generalization of Evans-Krein functions}\label{s:gEK}
For some purposes (see Example~\ref{ex:BEC} below) it may be important to determine whether two distinct 
real characteristic values of a selfadjoint operator pencil $\LL$ belong to (are obtained as roots of) the ``same'' 
eigenvalue branch $\mu=\mu(\la)$ of the associated eigenvalue problem \eqref{Lmu}.  Determining whether 
two real characteristic values are connected in this way requires analyzing the eigenvalue branches of the operator $\LL(\la)$ 
away from $\mu=0$.  When we described a typical construction of an Evans-Krein function in problems 
of stability analysis for nonlinear waves, we pointed out that when $\mu$ is sufficiently small the hyperbolicity of the
asymptotic systems \eqref{eq:asymp-systems} is guaranteed if it is present for $\mu=0$.
If we wish to have a construction that works for larger values of $\mu$, we can generalize
the definition of an Evans-Krein function in the following way.

Instead of  \eq{Lmu} one can consider the $\mu$-deformed characteristic value problem
\begin{equation}
\LL(\la) u = \mu Su,
\label{KEfMod}
\end{equation}
and define a generalized Evans-Krein function as an Evans function for the pencil 
$\mathcal{K}_S(\la;\mu):=\LL(\la)-\mu S$, 
where $S$ is a suitable positive definite selfadjoint operator on $X$. 
The correspondence between traditional and graphical Krein indices/signatures  given in Theorem~\ref{theorem:signature} 
can be shown to hold when the eigenvalue branches are defined by the modified problem \eq{KEfMod} for a wide 
class of positive definite selfadjoint operators $S:X\to X$. In the case of spectral analysis of nonlinear wave equations linearized about 
a localized traveling wave solution (for which the Evans function construction described in \S\ref{sec:typical-construction} applies), 
it is enough to define $Su(x)=f(x)u(x)$ where $f(x)$ is a positive Schwartz-class function.  The rapid decay of $f$ as $|x|\to\infty$ 
guarantees that for all $\mu$ however large the asymptotic systems \eqref{eq:asymp-systems} are exactly the same as they are 
for $\mu=0$ and hence hyperbolicity is preserved.  
Such a generalized Evans-Krein function should be straightforward to implement in existing computer codes for 
evaluating Evans functions constructed as described in \S\ref{sec:typical-construction}.
Finally, we can observe that the freedom of choice of the operator $S$ in \eq{KEfMod} suggests a certain structural stability 
of the problem of determining Krein indices  and signatures of real characteristic values of the problem \eq{Lu}.  
A specific example of a kind of problem in which one needs to consider values of $\mu$ that are not small is the following.

\begin{example}[Stability of vortices in axisymmetric Bose-Einstein condensates.]
\label{ex:BEC}
The  wave function $\Psi = \Psi(x,t)$ of a Bose-Einstein condensate 
confined in a harmonic trap evolves according to 
the Gross-Pitaevskii equation 
\begin{equation}\label{GP}
i \hbar \psi_t  = 
\left( - \frac {1}{2}  \Delta  +  \frac{|x|^2}{2} + |\psi|^2 \right) \psi,
\end{equation}
where $\Delta$ is the Laplacian.
In the two-dimensional case ($x\in\RR^2$, $|x|^2:=x_1^2+x_2^2$) the Gross-Pitaevskii equation \eq{GP} has
solitary wave (vortex) solutions of the form
\begin{equation}
\psi(t,r,\theta) = e^{-i(m+1+\mu) t}e^{im\theta} w_{\mu,m}(r),
\end{equation}
where $(r,\theta)$ are polar coordinates, $m$ is the \emph{degree} of the vortex, $m+1+\mu$ is
the vortex rotation frequency and $w_{\mu,m}(r)$ is the radial vortex profile. 

The stability of vortices in Bose-Einstein condensates has been well-studied. 
For example, in \cite{KP} an Evans function was used to study the spectral stability of vortices 
of degree $m=1, 2$  which exist for $\mu \ge 0$. Reformulating the spectral stability problem studied in \cite{KP} 
in the language of selfadjoint operator pencils, spectral instability corresponds to 
the presence of non-real characteristic values (mode frequencies) $\la$ of the spectral problem $\LL_\mathrm{BEC}(\la)y=0$ 
for a function $y=(y_+(r),y_-(r))^\mathsf{T}\in L^2(\RR_+,\mathbb{C}^2;r\,dr)$, 
where $\LL_\mathrm{BEC}(\la)=\la L_1+L_0$ is the linear pencil with coefficients
\begin{equation}
\begin{split}
L_1&:=-\sigma_3, \\L_0&:=\left(-\frac{1}{2r}\frac{d}{dr}r\frac{d}{dr}+\frac{1}{2}r^2-m-1-\mu
+\frac{j^2+m^2}{2r^2}+2|w_{\mu,m}(r)|^2\right)\mathbb{I} \\
&\quad\quad\quad{}+\frac{jm}{r^2}\sigma_3 +|w_{\mu,m}(r)|^2\sigma_1.
\end{split}
\end{equation}
Here, $j$ is an integer satisfying $0 < |j|<2m$ that 
indexes an angular Fourier mode.  Note that in the indicated weighted $L^2$ space, this operator pencil $\LL_\mathrm{BEC}$ 
is selfadjoint for bounded $y$ that decay rapidly as $r\to \infty$.

In \cite{KP} the radial profile of the vortex and then the corresponding
characteristic values of $\LL_\mathrm{BEC}$ were calculated numerically for a moderately large range of values of $\mu \ge 0$.   
Note that $\mu\ge 0$ enters into the pencil $\LL_\mathrm{BEC}$ both through the linear term $-\mu\mathbb{I}$ 
and also through the radial profile $|w_{\mu,m}(r)|^2$ of the vortex.

For $m = 1$ all the characteristic values of $\LL_\mathrm{BEC}$ are real and thus the vortex 
is spectrally stable, while for $m = 2$  variation of the parameter $\mu$ leads to 
Hamiltonian-Hopf bifurcations creating ``bubbles'' of instability \cite{MacKay} in which two colliding characteristic values 
leave and subsequently return to the real axis. As it is typical in Hamiltonian systems, the underlying Hamiltonian 
has only a few negative eigenvalues limiting the number of possible real characteristic values of negative Krein signature. 
In fact, for unstable
vortices ($m=2$, $j = 2$) there is only one real characteristic value  of negative Krein signature that  repeatedly 
collides with the remaining real characteristic values of positive signature as the parameter $\mu$ increases. In \cite{KP} 
a rigorous theory of continuity of Krein signatures for finite systems of characteristic values 
was developed, and its results were used to identify the unique real characteristic value having negative 
Krein signature via a continuation (homotopy) method in $\mu$ starting from the explicitly solvable case of $\mu = 0$.  

The apparent weakness of this method is that a stability check for large values of $\mu$ 
requires a significant computational overhead --- the calculation of the characteristic values of the pencil $\LL_\mathrm{BEC}$ 
for a discrete sampling of a large interval of $\mu$ that is sufficiently dense to 
ensure continuity of characteristic values (and their Krein signatures). Naturally, 
once a characteristic value $\la$ has been calculated for any value of $\mu$ it is possible to determine its Krein signature 
by an evaluation of the particular quadratic form as in \eqref{simpleKreinexamples} at the corresponding characteristic vector.  
Unfortunately, the Evans function used in \cite{KP} was constructed from exterior products, and consequently it is not obvious 
how to numerically recover the characteristic vectors from the Evans function; this meant that the characteristic vectors and 
subsequently the Krein signature of each characteristic value had to be calculated separately.
By contrast,  it would be easy to capture 
the Krein signature of any real characteristic value directly from an Evans-Krein function with 
a minimal extra computational and coding cost for any value of $\mu$, 
thus making calculations for large $\mu$ directly accessible. 
\end{example}

\section{Index Theorems for Linearized Hamiltonian Systems\label{s:Counts}}
The graphical interpretation of Krein indices and signatures afforded by Theorem~\ref{theorem:signature} can be applied to develop 
very simple proofs of some well-known \emph{index theorems} (eigenvalue counts, inertia laws) for linearized Hamiltonian systems 
that appear frequently in stability analysis  \cite{Grillakis1988, GSS1, KKS, Pel,KapStef2012,Pel2013} where the central question 
is the presence of \emph{unstable spectrum} of \eq{JL} in the right-half complex plane. Our goal is to show that the counts 
of unstable spectrum provided by these index theorems can be interpreted in terms of topological indices of planar curves in a quadrant 
(a somewhat different method to obtain index theorems using the graphical Krein signature was used in \cite{BinBrown1988}).

Recall Example~\ref{ex:posdef}, in which it was shown that unstable spectrum can only be present if $L$ is indefinite. 
In applications to stability 
of nonlinear waves  $L$ typically represents the linearization of a Hamiltonian whose kinetic energy part is a positive-definite unbounded 
differential operator.  In such cases $L$ will itself be unbounded in the positive direction but may have a finite number of 
negative eigenvalues that might cause instability of the system. In this context, the purpose of an index theorem is to bound the number 
of possible points in the unstable (non-imaginary) part of the spectrum $\sigma(J\!L)$
of \eq{JL} in terms of the total dimension of negative eigenspaces of $L$ and information about the stable (purely imaginary) part 
of the spectrum.  It turns out that this information amounts to the number of real characteristic values $\la$ of the corresponding 
linear pencil $\LL$ given by \eq{JLpencil} having certain types of Krein signatures.  (Recall that the real characteristic values of $\LL$ 
are in one-to-one correspondence with the purely imaginary points $\nu$ of $\sigma(J\!L)$ simply by a rotation of the complex plane:  
$\la=i\nu$.)  Such theorems therefore generalize the simple statement that positivity of $L$ prevents instability.
An important application of index theorems is the justification of a numerical search for unstable spectrum; 
once the number of complex characteristic values of $\LL$ obtained numerically
(say via Argument Principle calculations involving an appropriate Evans function) equals the maximum 
number admitted by the index theorem,  the numerical search can be considered complete \cite{KP}.

For the sake of clarity, we will work exclusively in the finite-dimensional setting. All of our applications below will be based on 
Theorem~\ref{theorem-pencil-count}, and while some aspects of its proof are not difficult to generalize to the infinite-dimensional 
setting (for instance, by counting codimensions rather than dimensions), the possible presence of an infinite number of discrete 
eigenvalue branches complicates other aspects of the proof.

To explain the index theorems graphically, we shift the focus from the local information contained in the way the eigenvalue 
branches $\mu(\la)$ of \eq{gpencil} cross $\mu=0$ to the global topological information stored in the branches for $\mu$ not small.

\subsection{Graphical analysis of selfadjoint polynomial matrix pencils}
We begin by formulating some simple consequences of graphical analysis of eigenvalue curves $\mu=\mu(\la)$ corresponding to 
selfadjoint polynomial matrix pencils of arbitrary (odd) degree.  In \S\ref{sec:index_applied} below we will apply these results in the special 
case of linear pencils, but it is easy to envision future applications where greater generality could be useful.  
We use the following notation below.
\begin{definition}
If $L$ is a Hermitian $N\times N$ matrix, the number of strictly positive (respectively negative) eigenvalues of $L$ counted with multiplicity 
is denoted $N_+(L)$ (respectively $N_-(L)$).  

If $\LL$ is a selfadjoint $N\times N$ matrix pencil, the number of strictly positive (respectively negative) real characteristic values of $\LL$ 
counted with geometric multiplicity is denoted $N_+(\LL)$ (respectively $N_-(\LL)$). 
The number of eigenvalue curves $\mu=\mu(\la)$ of  $\LL$ 
passing through $(\la,\mu)=(0,0)$ with $\mu(\la)<0$ for sufficiently small positive (respectively negative) $\la$, will be denoted 
$Z^\downarrow_+(\LL)$ (respectively $Z^\downarrow_-(\LL)$), and analogously the number of eigenvalue curves $\mu=\mu(\la)$ 
for which $\mu(\la)>0$ for sufficiently small positive (respectively negative) $\la$ will be denoted 
$Z_+^\uparrow(\LL)$ (respectively $Z_-^\uparrow(\LL)$).
\label{def:counts}
\end{definition}

Note that whenever we refer to a number of eigenvalue curves as in Definition~\ref{def:counts}, we intend them to be 
counted weighted by their geometric multiplicity. It is important to stress that there may be curves counted in 
\emph{both} $Z^\downarrow_+(\LL)$ and  $Z^\downarrow_-(\LL)$, namely those vanishing at $\la=0$ to even order $m$ 
with $\mu^{(m)}(0)<0$. If the order of vanishing $m$ is even but $\mu^{(m)}(0)>0$, then the corresponding curve will  
be counted in \emph{neither} $Z^\downarrow_+(\LL)$ nor $Z^\downarrow_-(\LL)$.  Curves vanishing at $\la=0$ to odd 
order will be counted in exactly one of $Z^\downarrow_+(\LL)$ or $Z^\downarrow_-(\LL)$.

Of course $N_+(L)+N_-(L)+\dim(\Ker(L))=N$, and also 
$Z^{\downarrow}_-(\LL)+Z^{\uparrow}_-(\LL) = Z^{\downarrow}_+(\LL)+Z^{\uparrow}_+(\LL)= \dim(\Ker(\LL(0)))$.  
In general neither $N_+(\LL)$ nor $N_-(\LL)$ is necessarily finite, although the situation is better 
for polynomial pencils, for which we have the following result.

\begin{theorem}
\label{theorem-pencil-count}
Let $\LL=\LL(\la):=L_0+\la L_1 +\cdots + \la^{p}L_{p}$ be a selfadjoint matrix pencil
of odd degree $p=2\ell+1$ acting on $X=\mathbb{C}^N$, and suppose that $L_p$ is invertible.  
Then we have the fundamental graphical conservation law
\begin{equation}
N-2N_-(L_0)-Z^\downarrow_+(\LL)-Z^\downarrow_-(\LL) +\mathop{\sum_{\la\in\sigma(\LL)}}_{\la>0}\kappa(\la) -
\mathop{\sum_{\la\in\sigma(\LL)}}_{\la<0}\kappa(\la)=0.
\label{eq:graphical-conservation-law}
\end{equation}
Also, 
the following inequalities hold true:
\begin{equation}
N_\pm(\LL)\ge \left|N_-(L_0)+Z^\downarrow_\pm(\LL)-N_\mp(L_p)\right|.
\label{eq:graphical-inequality}
\end{equation}
\end{theorem}

\begin{proof}
First note that by Definition~\ref{def:char_matrix} the total algebraic multiplicity of all characteristic values 
making up the spectrum $\sigma(\LL)$  is equal to $pN$ and hence finite.
According to Theorem~\ref{matrixhol}, the $N$ eigenvalue curves $\mu=\mu(\la)$ associated with the selfadjoint pencil $\LL$ 
may be taken as holomorphic functions of $\la\in\mathbb{R}$.  
The assumption that $L_p$ is invertible implies that none of the eigenvalue curves can be a constant function, and hence each 
curve has a well-defined finite order of vanishing at each corresponding real characteristic value of $\LL$.  Let $Q_\pm$ denote 
the open quadrants of the $(\la,\mu)$-plane corresponding to $\mu<0$ and $\pm\la>0$.  By analyticity and hence continuity of 
the eigenvalue branches $\mu=\mu(\la)$, the number of branches entering $Q_\pm$  necessarily equals the number of branches 
leaving $Q_\pm$ (with increasing $\la$).

Consider first the quadrant $Q_-$.  Curves $\mu=\mu(\la)$ can enter $Q_-$ with increasing $\la$
in only two ways:  from $\la=-\infty$ (corresponding to curves with $\mu(\la)<0$ asymptotically
as $\la\to -\infty$) and through the negative $\la$-axis.  Since the odd-degree polynomial pencil $\LL$ is dominated by its invertible 
leading term $\la^pL_p$ as $|\la|\to\infty$ (as is easy to see by applying Rouch\'e's Theorem to the characteristic polynomial 
of the matrix $\la^{-p}\LL(\la)$) , 
the number of curves entering $Q_-$ from $\la=-\infty$ is precisely $N_+(L_p)$.
Similarly, curves $\mu=\mu(\la)$ can exit $Q_-$ with increasing $\la$ in exactly three ways:  through the negative $\mu$-axis, 
through the origin $(\la,\mu)=(0,0)$, and through the negative $\la$-axis.  Since $\LL(0)=L_0$,
the number of curves exiting $Q_-$ through $\mu<0$ is precisely $N_-(L_0)$, and $Z^\downarrow_-(\LL)$ is the count of curves 
exiting $Q_-$ through the origin.  Therefore ``conservation of curves'' for the quadrant $Q_-$ reads
\begin{multline}
N_+(L_p)-N_-(L_0)-Z^\downarrow_-(\LL) \\
{}= \text{Net number of curves exiting $Q_-$ through $\la<0, \mu = 0$}.
\label{eq:Q-minus-count-0}
\end{multline}
If $\la_0$ is a real characteristic value of $\LL$, then according to Definition~\ref{graphsig} and Theorem~\ref{theorem:signature}, 
the Krein signature $\kappa(\la_0)$ exactly counts the net number of eigenvalue branches $\mu=\mu(\la)$ vanishing at $\la_0$ 
exiting the half-plane $\mu<0$ with increasing $\la$.  Therefore conservation of curves for $Q_-$ can be rewritten equivalently as
\begin{equation}
N_+(L_p)-N_-(L_0)-Z^\downarrow_-(\LL)=\mathop{\sum_{\la\in\sigma(\LL), \la<0}}\kappa(\la).
\label{eq:Q-minus-count}
\end{equation}
A conservation law for curves in $Q_+$ is obtained similarly:
\begin{multline}
N_-(L_p) - N_-(L_0)-Z^\downarrow_+(\LL)\\
\begin{aligned} &= 
\text{Net number of curves entering $Q_+$ through $\la>0$, $\mu = 0$}\\
& = -\sum_{\la\in\sigma(\LL), \la>0}\kappa(\la).
\end{aligned}
\label{eq:Q-plus-count}
\end{multline}
Adding together \eqref{eq:Q-minus-count} and \eqref{eq:Q-plus-count}, and taking into account the fact that $N_-(L_p)+N_+(L_p)=N$ 
because $L_p$ is Hermitian $N\times N$ and invertible, we arrive at the graphical conservation law \eqref{eq:graphical-conservation-law}.

The inequalities \eqref{eq:graphical-inequality} follow immediately from the equations \eqref{eq:Q-minus-count-0}
and \eqref{eq:Q-plus-count} by noting that the absolute value of the net number of curves crossing the positive or negative $\la$-semiaxis 
is a lower bound for the number of curves crossing the same axis, and the latter is exactly a count of real characteristic 
values according to their geometric multiplicity.
\end{proof}

\begin{figure}[h]
\begin{center}
\includegraphics{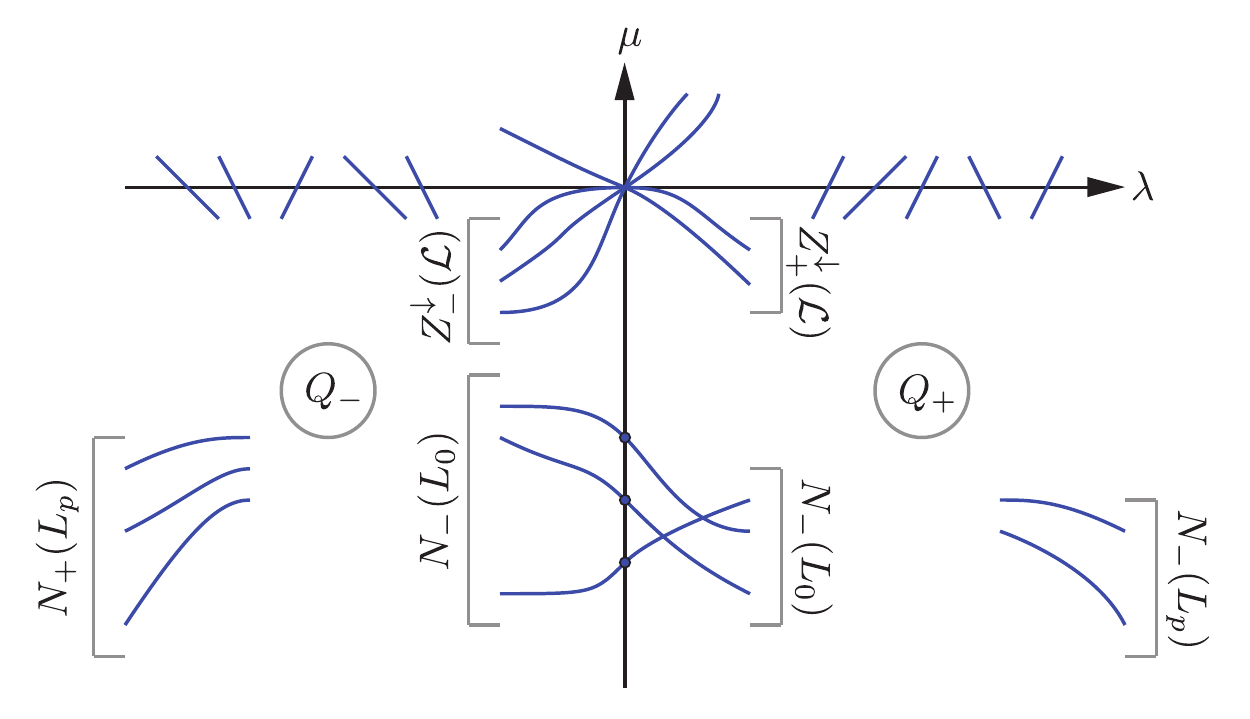}
\end{center}
\caption{The proof of Theorem~\ref{theorem-pencil-count} is essentially graphical in nature.  In each quadrant $Q_\pm$ the number of 
curves entering and exiting through infinity, the negative $\mu$-axis, and through the origin must be balanced by the net numbers exiting 
through the positive and negative $\la$-semiaxes as calculated by summing Krein signatures.  The negative 
eigenvalues of $L_0$ are plotted with blue dots.}
\label{fig:graphical-conservation-law}
\end{figure}

The key ideas of the proof are illustrated in Figure~\ref{fig:graphical-conservation-law}.
As the proof is essentially topological in nature, the analyticity of the branches implied by Theorem~\ref{matrixhol} 
could be replaced by mere continuity.

\subsection{Application to linearized Hamiltonian systems}
\label{sec:index_applied}
We now return to \eqref{JL}, the spectral problem $J\!Lu=\nu u$ for a linearized Hamiltonian system, where $L$ and $J$ are (Hermitian 
and invertible skew-Hermitian, respectively) matrices of size $N\times N$,  and we recall that in the finite-dimensional setting the underlying 
space $X=\CC^N$ necessarily has even dimension $N=2n$ due to invertibility of the skew-Hermitian matrix $J$. 
Let the total algebraic multiplicity of all points $\nu\in\sigma(J\!L)$ with $\Re\{\nu\}\neq 0$ be denoted $2n_\mathrm{u}$ 
(necessarily an even number due to the basic Hamiltonian symmetry $\nu\to -\overline{\nu}$).  This is the dimension of the 
\emph{unstable invariant subspace} of $J\!L$.  The total algebraic multiplicity of all purely imaginary points of $\sigma(J\!L)$ 
is therefore $2n-2n_\mathrm{u}=2n_\mathrm{s}$, so $2n_\mathrm{s}$ is the dimension of the \emph{(spectrally) 
stable invariant subspace} 
of $J\!L$. The first result is the following.
\begin{theorem}[Index theorem for linearized Hamiltonians]
Let $L$ and $J$ be $2n\times 2n$ matrices, and let $L$ be Hermitian and $J$ be invertible and skew-Hermitian.  
Let $\LL=\LL(\la):=L-\la K$, $K:=(iJ)^{-1}$, be the associated linear matrix pencil.  Then
\begin{equation}
n_\mathrm{u}=N_-(L) - \zeta -\sum_{\la\in\sigma(\LL), \la>0}\kappa^+(\la)
-\sum_{\la\in\sigma(\LL), \la<0}\kappa^-(\la).
\label{eq:unstable-count-0}
\end{equation}
The Krein indices $\kappa^\pm(\la)\ge 0$ are associated with real characteristic values $\la$
of the selfadjoint pencil $\LL$, which in turn correspond via rotation of the complex plane to purely imaginary points in $\sigma(J\!L)$.
Here the quantity $\zeta$ is given by
\begin{equation}
\zeta=\frac{1}{2}\dim(\gKer(J\!L)) -\frac{1}{2}\left(Z_+^\downarrow(\LL)+Z_-^\downarrow(\LL)\right),
\label{eq:zeta-0}
\end{equation}
or equivalently by
\begin{equation}
\zeta:=\frac{1}{2}\left(\dim(\gKer(J\!L))-2\dim(\Ker(L))\right) +\frac{1}{2}\left( Z^\uparrow_+(\LL)+Z^\uparrow_-(\LL)\right).
\label{eq:zeta-1}
\end{equation}
\label{theorem-linearized-Hamiltonians-0}
\end{theorem}
\begin{proof}
We apply Theorem~\ref{theorem-pencil-count} to the selfadjoint linear pencil $\LL$ associated with $L$ and $J$.  Here $p=1$, 
and $L_0=L$ while $L_1=-K$.  Note that $K$ is invertible. With $N=2n=2n_\mathrm{u}+2n_\mathrm{s}$, the graphical conservation 
law \eqref{eq:graphical-conservation-law} reads (upon rearrangement and division by $2$)
\begin{equation}
n_\mathrm{u}=N_-(L)+\frac{1}{2}(Z^\downarrow_+(\LL)+Z^\downarrow_-(\LL)) - n_\mathrm{s}
-\frac{1}{2}\mathop{\sum_{\la\in\sigma(\LL)}}_{\la>0}\kappa(\la)+\frac{1}{2}\mathop{\sum_{\la\in\sigma(\LL)}}_{\la<0}\kappa(\la).
\label{eq:unstable-count-1}
\end{equation}
Of course by definition the Krein signatures $\kappa(\la)$ appearing on the right-hand side can be written as differences 
of two positive Krein indices:  
\begin{equation}
\kappa(\la)=\kappa^+(\la)-\kappa^-(\la).  
\label{eq:kappa-difference}
\end{equation}
But $n_\mathrm{s}$ can also be written in terms of the Krein indices of the real characteristic values of $\LL$.  
Indeed, the total algebraic multiplicity of a purely imaginary point $\nu\in\sigma(J\!L)$ is equal to that of the 
corresponding real characteristic value $\la=i\nu$ of $\LL$, and the latter can be expressed
as the sum of the Krein indices $\kappa^+(\la)+\kappa^-(\la)$.  Therefore, summing over
all real characteristic values of $\LL$ we obtain:
\begin{equation}
2n_\mathrm{s} = \sum_{\la\in\sigma(\LL), \la\in\mathbb{R}}\left(\kappa^+(\la)+\kappa^-(\la)\right).
\end{equation}
Substituting this identity and \eqref{eq:kappa-difference} into \eqref{eq:unstable-count-1}, and noting that as a count 
of algebraic multiplicity we have $\kappa^+(0)+\kappa^-(0)=\dim(\gKer(J\!L))$ if $0\in\sigma(J\!L)$, we obtain
\eqref{eq:unstable-count-0} with $\zeta$ given in the form \eqref{eq:zeta-0}.
Since $Z^{\uparrow}_+ (\LL)+ Z^{\downarrow}_+(\LL) =Z^{\uparrow}_- (\LL)+ 
Z^{\downarrow}_-(\LL)= \dim(\Ker (\LL(0)))=\dim(\Ker(L))$ the 
alternative formula \eqref{eq:zeta-1} follows immediately.
\end{proof}

Recall that if  $J$ and $L$ are real matrices, then in addition to the Hamiltonian symmetry $\sigma(J\!L)=-\overline{\sigma(J\!L)}$ 
we have the real symmetry $\sigma(J\!L)=\overline{\sigma(J\!L)}$, in which case system is said to have \emph{full Hamiltonian symmetry}.  
In many applications great utility is gleaned from  this extra symmetry, but one should be aware that the reality 
of the matrices $J$ and $L$ (or operators in the infinite-dimensional setting) is essentially tied to a choice of basis or representation.  
In other words, within the usual Hilbert space axioms one finds no inherent notion of what it means for a vector or operator to 
be ``real''. Therefore, a problem can possess full Hamiltonian symmetry even though it may not be obvious that the operators 
are ``real'' due to poor choice of coordinates.  
A useful general notion of reality is based instead on the existence of an appropriate antiholomorphic involution on $X$ 
and can be formulated as follows (compare with \cite[Assumption 2.15]{KapHar}).
\begin{definition}
Let $I:X\to X$ be an involution on a Hilbert space $X$ that is unitary: $\|I(u)\|=\|u\|$ for all $u\in X$, 
and conjugate linear:  $I(\alpha u+\beta v)=\overline{\alpha}I(u)+\overline{\beta}I(v)$ for all  $u,v\in X$ and $\alpha,\beta\in\CC$.
A linear operator $A:X\to X$ is said to be real (with respect to $I$) if it 
commutes with $I$:  $AI(u)=I(Au)$ for all $u\in X$.  Similarly, a vector $u\in X$ is real (with respect to $I$) if $u=I(u)$.
\label{Def:real}
\end{definition}

The involution $I$ is simply an abstraction of complex conjugation as might be applied to complex coordinates of vectors in $X$.  
In the case of real matrices $J$ and $L$ we have the following important corollary of Theorem~\ref{theorem-linearized-Hamiltonians-0}.

\begin{corollary}[Index theorem for linearized Hamiltonians with full Hamiltonian symmetry]
Let $L$ and $J$ be $2n\times 2n$ matrices, and let $L$ be Hermitian and $J$ be invertible and skew-Hermitian. 
Suppose also that $L$ and $J$ are real with respect to a unitary antiholomorphic involution $I:\mathbb{C}^{2n}\to\mathbb{C}^{2n}$.  
Then with $\zeta$ being given by either \eqref{eq:zeta-0} or \eqref{eq:zeta-1}, 
\begin{equation}
n_\mathrm{u} = N_-(L)-\zeta-2\sum_{\la\in\sigma(\LL), \la>0}\kappa^+(\la)
= N_-(L)-\zeta-2\sum_{\la\in\sigma(\LL), \la<0}\kappa^-(\la),
\label{eq:symmetric-count}
\end{equation}
and moreover, $\zeta$ can be simplified as follows:
\begin{equation}
\begin{split}
\zeta &= \frac{1}{2}\dim(\gKer(J\!L)) - Z^-(\LL) \\
&=\frac{1}{2}\left(\dim(\gKer(J\!L))-2\dim(\Ker(L))\right) + Z^+(\LL),
\end{split}
\label{eq:simple-zeta}
\end{equation}
where $Z^+(\LL)$ (respectively $Z^-(\LL)$) is the number of curves $\mu=\mu(\la)$ vanishing at $\la=0$ 
whose first nonzero derivative $\mu^{(m)}(\la)$ is positive (respectively negative).
If in addition%
\footnote{The condition $\dim(\gKer(J\!L))=2\dim(\Ker(L))$  holds in many examples appearing in stability of nonlinear waves, 
but in fact it is completely independent of the presence of full Hamiltonian symmetry of the system and invertibility of $J$. 
A trivial example with $L$ identically zero and $J$ canonical (see \eq{eq:canonical-form}) yields $\dim(\gKer(J\!L))=\dim(\Ker(L)) = 2n$.
Similarly  $L = \mbox{diag\,}(1, 0, 1,0)$ and $J$ canonical (for $n=2$) gives  $\dim(\gKer(J\!L))=\dim(\Ker(L)) = 2$. 
On the other hand, there are also examples with $\dim(\gKer(J\!L)) > 2\dim(\Ker(L))$.}
one has $\dim(\gKer(J\!L))=2\dim(\Ker(L))$, then $\zeta=Z^+(\LL)$.
\label{corollary-full-symmetry}
\end{corollary}

\begin{proof}
The fact that $J\!L$ commutes with the involution $I$ forces the whole system of curves $\mu=\mu(\la)$ to be symmetric 
with respect to reflection through the vertical $\mu$-axis.  This implies that if $\la_0>0$ is a real characteristic value (with root space $U$) 
of the associated linear pencil $\LL$, then so is $-\la_0$ (with root space $I(U)$), and from the definition of graphical Krein indices 
(see Definition~\ref{graphsig} and Theorem~\ref{theorem:signature}) we see that $\kappa^\pm(\la_0)=\kappa^\mp(-\la_0)$ 
due to the symmetry of the curves.  Therefore, both sums on the right-hand side of \eqref{eq:unstable-count-0} are equal, 
and \eqref{eq:symmetric-count} follows immediately.

The simplified formulae \eqref{eq:simple-zeta} for $\zeta$ follow from the left-right symmetry of the
union of curves $\mu=\mu(\la)$ passing through the origin $(\la,\mu)=(0,0)$.  Indeed, the symmetry
implies that the number of curves vanishing to odd order $m$ is even and divides equally into the number of those 
curves for which $\mu^{(m)}(0)>0$ and that of their reflections through the $\mu$-axis.  This implies 
that $Z_+^\downarrow(\LL)+Z_-^\downarrow(\LL)=2Z^-$, and then since $\dim(\Ker(L))=\dim(\Ker(J\!L))$ 
is the total number of analytic branches through the origin, the proof of \eqref{eq:simple-zeta} is complete.  
The fact that the condition $\dim(\gKer(J\!L))=2\dim(\Ker(L))$ implies that $\zeta = Z^+(\LL)$ is obvious. 
\end{proof}

The next application concerns spectral problems for linearized Hamiltonian systems for which $L$ and $J$ 
have so-called \emph{canonical form}, that is, the $2n\times 2n$ matrices can be written in terms of $n\times n$ blocks as follows:
\begin{equation}
J=\begin{pmatrix}0 & \mathbb{I}\\-\mathbb{I} & 0\end{pmatrix},\quad
L=\begin{pmatrix}L_+ & 0\\0 & L_-\end{pmatrix}
\label{eq:canonical-form}
\end{equation}
where $\mathbb{I}$ is the $n\times n$ identity and $L_\pm$ are $n\times n$ Hermitian matrices.
Note that in this case $J$ is automatically invertible.  This form occurs naturally in many applications, 
but in fact the general $J\!L$ spectral problem \eqref{JL} can be rewritten in canonical form by an appropriate inflation 
of the dimension of the problem \cite{KKS}. Under the assumption that $\Ker (L_+) \perp \Ker (L_-)$,  a lower bound for the 
number of real points $\nu\in\sigma(J\!L)$  can be given in terms of  the difference of the negative indices of operators $M_\pm$ 
that are suitable projections of the operators $L_\pm$.

\begin{theorem}[Lower bound for real points of $\sigma(J\!L)$]
Let $L$ and $J$ be $2n\times 2n$ matrices having canonical form \eqref{eq:canonical-form}.
Suppose also that $\Ker(L_+)\perp\Ker(L_-)$, and let $V$ denote the orthogonal complement of 
$\Ker(L_+)\oplus\Ker(L_-)$ in $Y=\mathbb{C}^n$ with corresponding orthogonal projection $P:Y\to V$. 
If $N_\mathbb{R}(J\!L)\le 2n_\mathrm{u}$ denotes the number of nonzero real points in $\sigma(J\!L)$ 
counted with geometric multiplicity (necessarily an even number by the basic Hamiltonian symmetry
$\sigma(J\!L)=-\overline{\sigma(J\!L)}$), then
\begin{equation}
\frac{1}{2}N_\mathbb{R}(J\!L)\ge \left| N_-(M_+)-N_-(M_-)\right|, \qquad
M_\pm:=PL_\pm P.
\label{eq:JL-inequality}
\end{equation}
\label{theorem-JL-inequality}
\end{theorem}

\begin{proof}
First, we reduce the spectral problem $J\!Lu=\nu u$ on $X=\mathbb{C}^{2n}$ to a spectral problem
for a linear selfadjoint pencil on $Y$ using the projection technique \cite{Grillakis1988,  Kap, KKS,VP}. 
Identifying $u\in X$ with the pair $(u_1, u_2)$ with $u_j\in Y$ for $j=1,2$, the spectral problem $J\!Lu=\nu u$ 
takes the form of a coupled system:
\begin{equation}
L_-u_2=\nu u_1\quad\text{and}\quad -L_+u_1=\nu u_2.
\end{equation}
Now since $\Ker(L_+)\perp\Ker(L_-)$, the space $Y$ can be decomposed into a direct sum of three pairwise orthogonal summands:
$Y=\Ker(L_+)\oplus\Ker(L_-)\oplus V$.
Letting  $P_\pm:Y\to\Ker(L_\pm)$ be the orthogonal projections onto $\Ker(L_\pm)$, and using the facts that
$L_\pm P_\pm=0$ and (as $L_\pm$ are selfadjoint) $P_\pm L_\pm=0$, we
may apply the projections $P_-$ and $P$ to the equation  $L_-u_2=\nu u_1$ and hence
obtain
\begin{equation}
\begin{split}
0 &=\nu P_-u_1\, , \\
PL_-P_+u_2+PL_-Pu_2&=\nu Pu_1\, ,
\end{split}
\end{equation}
and similarly applying the projections $P_+$ and $P$ to the equation $-L_+u_1=\nu u_2$,\begin{equation}
\begin{split}
0 &=\nu P_+ u_2\, ,\\
-PL_+P_-u_1-PL_+Pu_1&=\nu Pu_2\, .
\end{split}
\end{equation}
If $\nu\neq 0$, then we obviously have $P_-u_1=P_+u_2=0$, and setting $v_j=Pu_j\in V$ gives a coupled system 
of equations on the subspace $V$:
\begin{equation}
M_-v_2=\nu v_1\quad\text{and}\quad -M_+v_1=\nu v_2.
\end{equation}
The advantage of this reduction is that $M_\pm$ are both invertible on the subspace $V$, and therefore by 
eliminating $v_2$ we can write the original spectral problem $J\!Lu =\nu u$ for $\nu\neq 0$ as a spectral 
problem $\LL(\la)v_1=0$ where $\LL$ is the linear selfadjoint matrix pencil on $V$ given by 
\begin{equation}
\LL(\la):=L_0+\la L_1,\quad L_0:=M_+,\quad L_1:= M_-^{-1}, \quad \la:=\nu^2.
\end{equation}

Now we invoke Theorem~\ref{theorem-pencil-count}, specifically the inequality \eqref{eq:graphical-inequality} 
for the number $N_+(\LL)$ of positive real characteristic values of $\LL$ counted with geometric multiplicity.  
Since $L_0$ is invertible as well as the leading coefficient $L_1$, and since the total number of curves $\mu=\mu(\la)$ 
associated with the pencil $\LL$ passing through the origin $(\la,\mu)=(0,0)$ is equal to $\dim(\Ker(L_0))$, there are no such curves at all
in this application, thus $Z_\pm^\downarrow (\LL) = 0$.  Therefore \eqref{eq:graphical-inequality} takes the simple form
\begin{equation}
N_+(\LL)\ge\left|N_-(L_0)-N_-(L_1)\right| = \left|N_-(M_+)-N_-(M_-^{-1})\right|=\left|N_-(M_+)-N_-(M_-)\right|.
\end{equation}
The inequality \eq{eq:JL-inequality} follows from the fact that $N_+(\LL)$ counts the number of positive points in $\sigma(\LL)$, 
and hence by $\la=\nu^2$ it is exactly half the number of real nonzero points in $\sigma(J\!L)$, counted with geometric multiplicity.
\end{proof}

\subsection{Historical remarks}
Corollary~\ref{corollary-full-symmetry} is a generalization of an index theorem proved independently by Kapitula {\it et al.} \cite{KKS} 
and Pelinovsky \cite{Pel} in the special case that each maximal chain of root vectors corresponding to the characteristic 
value $\la=0$ has length 2. 
Such a situation corresponds to many typical Hamiltonian systems with symmetries \cite{GSS1, KKS}.  
In this case  $\zeta$ can be written in
the form $N_-(D)$ where $D$ is a Gram-type Hermitian matrix with elements $D_{jk}:=(Lu^{[j]},u^{[k]})$ and the vectors $u^{[j]}$ 
span $\gKer(J\!L)\ominus\Ker(L)$.
Also see \cite{KollarBosak} for an extensive survey of literature on related index theorems appearing in various fields of mathematics.

The relation between the dimension $2n_\mathrm{u}$ of the unstable invariant subspace of $J\!L$ and the number of negative 
eigenvalues of $L$ was first studied for $N_-(L) = 1$, in which case under the assumption that $\dim(\gKer(J\!L))=2\dim(\Ker(L))$, 
the count \eq{eq:symmetric-count}  indicates that $J\!L$ has at most one pair of non-imaginary (unstable) points of spectrum, 
and by the full Hamiltonian symmetry these points are necessarily real. In this case the famous 
\emph{Vakhitov-Kolokolov criterion} \cite{VK} 
applies in a large class of problems in the stability theory of waves and identifies the scalar quantity $\zeta=Z^+(\LL)=N_-(D)$ that
determines whether $n_\mathrm{u} = 0$ or $n_\mathrm{u} = 1$, i.e., 
whether a given wave is spectrally stable. In the Vakhitov-Kolokolov theory, the quantity $N_-(D)$ turns out to be equal to the derivative 
of the momentum (impulse functional) of the wave profile with respect to the wave speed.  
Pego and Weinstein \cite{PegoWeinstein} 
proved that the same quantity is related to the second derivative $D''(0)$ of the (usual) Evans function $D(\la)$, 
which by symmetries of the system satisfies 
$D(0) = D'(0) = 0$.  In celebrated papers, Grillakis {\it et al.} \cite{GSS1, GSS2} extended the spectral stability analysis associated 
to the Vakhitov-Kolokolov instability criterion to establish fully nonlinear stability properties of waves. See \cite{BSS} 
for the analysis preceding these general 
results and \cite{KapProm} for a historical discussion.   Simultaneously, Maddocks \cite{Maddocks} developed the theory of inertial laws
for constrained Hamiltonians and related it to results of MacKay \cite{MacKay} giving rise to an algebraic 
method for proving index theorems.

A generalization of the Vakhitov-Kolokolov criterion to the case of $N_-(L)>1$ (in the particular case that $z(L) = 0$) 
can be found in the work of Binding and Browne \cite{BinBrown1988} 
who used a beautiful combination of homotopy theory and analysis of eigenvalue branches in their argument. 
Later, Kapitula {\it et al.} \cite{KKS} and Pelinovsky \cite{Pel} proved a full generalization
of the Vakhitov-Kolokolov criterion to the case of $N_-(L)>1$.
They derived the formula mentioned above for $\zeta=Z^+(\LL)$ in terms of the Gram matrix $D$, and 
they also interpreted the  law \eqref{eq:symmetric-count} as a parity index by writing it in the form
$n_\mathrm{u}\equiv N_-(L)-N_-(D) \pmod 2$.  This identity plays an important role in the presence of 
full Hamiltonian symmetry because oddness of $n_\mathrm{u}$ indicates the presence of purely real points of 
the unstable spectrum.  Under some technical assumptions, Kapitula \cite{Kap} reproved these results using a similar technique 
involving continuity of eigenvalue curves in conjunction with the so-called \emph{Krein matrix} method.  
The latter effectively projects the problem to a finite-dimensional negative subspace of $L$, but in the process 
introduces unnecessary poles in the eigenvalue branches $\mu(\la)$, and these poles obstruct the simple graphical 
visualization of the Krein signatures and indices arising in the analysis. 
Kapitula and  \Haragus~\cite{KapHar} also proved the same count in the setting of differential operators with periodic 
coefficients using the Bloch wave decomposition (Floquet theory) to reduce the problem with bands of continuous 
spectrum to a collection of problems each of which has purely discrete spectrum. See also \cite{BJK} for an alternative 
proof of \eq{eq:symmetric-count} using aspects of integrable systems theory. An analogous  index theorem for quadratic 
Hermitian matrix pencils was proved 
by Pelinovsky and Chugunova \cite{ChagPel} using the theory of indefinite quadratic forms (Pontryagin spaces) and was later reproved 
by a homotopy technique \cite{KollarH} resembling the graphical method described in this paper. Another related count 
for quadratic operator pencils was recently proved by Bronski {\it et al.} \cite{BJK2}. 
Theorem~\ref{theorem-pencil-count} 
is an example of the kind of results that can be obtained beyond the simple context of problems reducible to linear pencils.
Index theorems similar to those presented here for selfadjoint and alternating (having alternating selfadjoint and skewadjoint coefficients)
polynomial matrix pencils of  arbitrary degree were recently proved in \cite{KHKT} using advanced results from the algebra of matrix 
pencils along with perturbation theory for polynomial pencils.  Similar results can be found also in
\cite{KollarBosak}.
Finally, a very recent and significant development has been the extension of  index theorems to cover the case of 
$J = \partial_x$ that includes important examples of traveling waves for 
Korteweg-deVries-type and Benjamin-Bona-Mahony-type problems.
Kapitula and Stefanov \cite{KapStef2012} and Pelinovsky \cite{Pel2013}
independently proved an analogue of Theorem~\ref{theorem-linearized-Hamiltonians-0} under the assumptions that  (i) there exist finitely
many negative eigenvalues of $L$, (ii) $\dim(\Ker(L)) = 1$, and (iii)  the essential spectrum of $L$  is bounded away from zero. They show 
how their analysis can be used to generalize particular stability results of Pego and Weinstein \cite{PegoWeinstein,PegoWeinsteinII} 
obtained originally for $N_-(L) = 1$. 

When the reality condition guaranteeing full Hamiltonian symmetry is dropped and one has
instead of \eqref{eq:symmetric-count} the more general statement \eqref{eq:unstable-count-0},
it is no longer possible to interpret the inertia law as a parity index.
Results analogous to those recorded in Theorem~\ref{theorem-linearized-Hamiltonians-0}
were previously proven in 
\cite[Theorem~6]{CP} and in \cite[Theorem 2.13]{KapHar} under the assumption that $L$ is an invertible 
operator with a compact inverse.  For the reader familiar with \cite{CP, KapHar}, the apparent difference between 
\eqref{eq:unstable-count-0} and the formulation of the index theorems found therein
is caused by a different choice of the definition of the Krein signature; see the footnote referenced between 
equations \eq{eq:Real-Im} and \eqref{simpleKreinexamples}. 
Finally, note that in \cite{KollarBosak} 
the authors give an algebraic formula for the quantity $\zeta$ appearing in the statement of 
Theorem~\ref{theorem-linearized-Hamiltonians-0} in the case that  $\zeta \neq Z^+(\LL)$,
and in this formula a key role is played by the canonical set of chains described in Theorem~\ref{multiplicities}.

Theorem~\ref{theorem-JL-inequality}, which concerns problems having the canonical symplectic structure \eqref{eq:canonical-form}, 
was proved independently and virtually simultaneously by Jones \cite{Jones1988} and Grillakis \cite{Grillakis1988}.  
Since the inequality \eqref{eq:JL-inequality} only provides a lower bound on the number of purely real points of $\sigma(J\!L)$ 
and not its exact count, the straightforward generalization from the finite-dimensional setting to operator theory 
does not require completeness of the root vectors. Jones' proof \cite{Jones1988} of Theorem~\ref{theorem-JL-inequality} 
is of a very different nature from the graphical one we have presented, but that of Grillakis \cite{Grillakis1988} is quite similar,
with spectral projections playing the role of the eigenvalue curves $\mu=\mu(\la)$.  We think
that our approach, embodied in the proof of Theorem~\ref{theorem-pencil-count},  gives a very simple way to visualize the count.  
This problem was also studied by Chugunova and Pelinovsky \cite{CP}, who proved a number
of related results using the Pontryagin Invariant Subspace Theorem applied to the linear pencil. More recently, a similar approach 
combined with the use of the Krein matrix, an analytical interpretation of the Krein signature, and the Keldysh Theorem 
\cite{Keldysh, Markus},  was employed by Kapitula and Promislow \cite{KapProm}, who reproved Theorem~\ref{theorem-JL-inequality} 
(see the paper for further historical remarks). The connection to the linear pencil was also pointed out in a similar setting in \cite{VP}.

\section{Discussion and Open Problems\label{s:Discussion}}
We hope that our paper has demonstrated how the analytic interpretation of Krein signatures and indices in terms of 
real curves and their order of vanishing at real characteristic values helps to easily visualize, simplify, and organize numerous results 
found in the literature on stability of nonlinear waves. 

The analytical or graphical interpretation of the Krein signature put forth in Definition~\ref{graphsig} and 
Theorem~\ref{theorem:signature} 
is apparently limited to real characteristic values of  selfadjoint pencils.  In particular, such a formulation
applied to the linear pencil \eq{LKeq} related to the Hamiltonian spectral problem \eqref{JL} does not provide 
any direct information about $\sigma(J\!L)$ off the imaginary axis. But such a characterization seems unnecessary, 
as the (traditional) Krein signature of non-imaginary points of $\sigma(J\!L)$ is easy to calculate and is equal to zero. 
Nevertheless it would be interesting to determine whether an approach similar to the one presented here 
can be applied to study non-selfadjoint pencils and/or to detect further information about non-real characteristic 
values of selfadjoint pencils. 

Our main new result, the generalization of the notion of an Evans function to that of an  Evans-Krein function, 
simplifies numerical calculation of Krein signatures from an Evans function and is a method easy to incorporate into existing codes.
From the theoretical perspective an interesting problem would be to find 
an intrinsically geometrical interpretation of Evans-Krein functions similar to the characterization of Evans functions given 
by Alexander {\it et al.} \cite{AGJ}. 

The graphical nature of the signature should also allow us to generalize the index theorems presented in \S\ref{s:Counts} to 
handle operators on infinite-dimensional Hilbert spaces with general kernels, sidestepping
certain unnecessary technical difficulties.   One also expects to be able to identify optimal
(from the graphical point of view) assumptions for the validity of these theorems. 
We advocate that, in general,  the concept of the graphical signature is often more suitable for  analysis
(either rigorous or numerical)  than the traditional one as it does not rely on any particular algebraic structure of the operator pencil. 
This may possibly allow further applications of signature arguments for non-polynomial pencils where the traditional definition 
based on indefinite quadratic forms falls flat. 

Finally, the graphical approach to Krein signature theory may be preferable for understanding various mechanisms 
for avoiding Hamiltonian-Hopf bifurcations \cite{MacKay,Meiss} in Hamiltonian systems with a tuning parameter $t\in\RR$. 
Of particular interest here are ``non-generic" collisions (in the sense of Arnold \cite{Arnold}) involving two real characteristic values  
of a selfadjoint pencil having  opposite Krein signatures but that nonetheless do not bifurcate from the real axis 
(see \cite{DeMeMa1992} for  related analysis).
The question to be addressed is the fate of the various eigenvalue curves $\mu=\mu(\la)$ of \eq{gpencil}
when the spectral problem is perturbed in an admissible fashion.  Here, one expects that certain conditions may ensure that not only 
do real characteristic values of opposite signatures survive collisions, 
but also the transversal intersection of eigenvalue branches $\mu=\mu(\la)$ is preserved as well.
Such a mechanism for avoiding Hamiltonian-Hopf bifurcations is significantly different from the one described by 
Krein and Ljubarskii \cite{KLjub} in the context of differential equations with periodic coefficients. 
Here we only present an illustrative example of what we have in mind;  we expect to publish elsewhere new 
results given necessary and sufficient conditions for the preservation of branch crossings.

\begin{example}[Nongeneric perturbations can avoid Hamiltonian-Hopf bifurcations.]\label{ex:branchprev}
Consider the pencil $\LL(\la):=L-\la K$, $K=(iJ)^{-1}$, where $L$ depends linearly on a control parameter 
$t\in\mathbb{R}$ as $L=A+tB$ where $A:=\mathrm{diag}(5,1,3,2)$, and where
\begin{equation}
B:=\begin{pmatrix}0 & 1 & 0 & 0\\1 & 0 & 0 & 0\\0 & 0 & 0 & 0\\0 &  0 & 0 & 0\end{pmatrix}\quad\text{and}\quad
J:=-\frac{1}{2}\begin{pmatrix}0 & 2i & 0 & 0\\
2i & 0 & 0 & 0\\
0 & 0 & 0 & i\\
0 & 0 & i & 0\end{pmatrix}.
\end{equation}
Fig.~\ref{fig13} illustrates how transversal intersections between the  eigenvalue branches 
$\mu=\mu(\la)$ present for $t=0$ can persist for $t \neq 0$.  For general perturbations of $L=A$ one
would expect all four transversal intersections to simultaneously ``break'' for $t\neq 0$, resulting
in four uniformly ordered analytic eigenvalue branches $\mu_1(\la)<\mu_2(\la)<\mu_3(\la)<\mu_4(\la)$, 
and if these branches have critical points, Hamiltonian-Hopf bifurcations resulting in the loss of one or more pairs of real 
characteristic values becomes possible as the branches shift vertically.  Therefore
the perturbation $B$ has to be quite special for the intersections to persist, a situation that prevents all 
Hamiltonian-Hopf bifurcations from occurring.  
\end{example}

\begin{figure}[htp]
\centering
\includegraphics{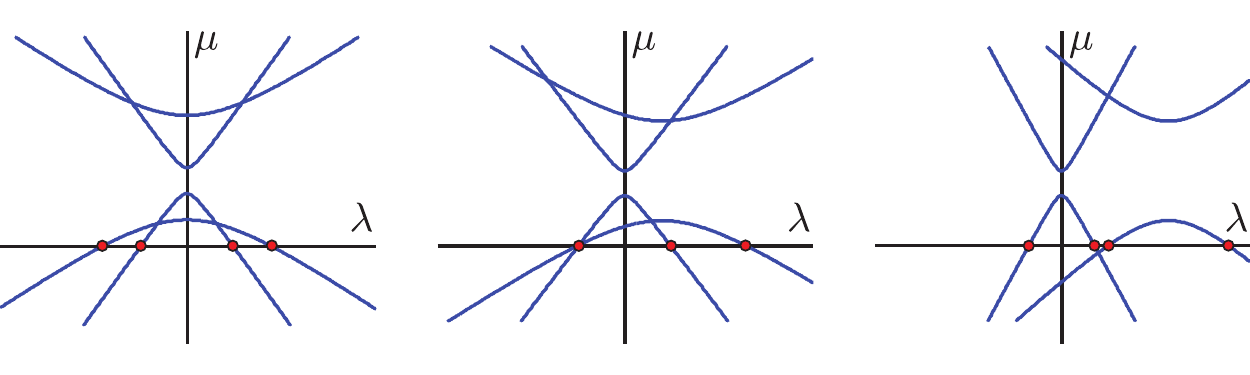}
\caption[]{%
Spectrum of $L(t)-\la K$ plotted against $\la$ for the matrices in Example~\ref{ex:branchprev}.
Left panel:  $t=0$.  Center panel:  $t=1$ (exhibiting a necessarily harmless collision between characteristic 
values of the same Krein signature).  
Right panel:  $t=4$. Between $t=1$ and $t=4$ two real characteristic values of opposite Krein signature 
pass through each other unscathed 
due to the preservation of transversal crossings of eigenvalue branches $\mu=\mu(\la)$.}
\label{fig13}
\end{figure}

\appendix
\section{Proofs of Propositions~\ref{theorem:algebraic-equivalence} and~\ref{prop:chains_derivatives}}
\subsection{Proof of Proposition~\ref{theorem:algebraic-equivalence}}\label{app1}
The statement of the boundedness and holomorphicity of the pencil $\MM$ follows as a special 
case of \cite[VII.1.2, p. 367, Theorem 1.3]{Kato}.  
To prove (a), suppose that $\la_0$ with $|\la_0-\la'|<\epsilon$ is such that there exists a nonzero
vector $u\in D\subset X$ such that $\LL(\la_0)u=0$.  Then also $(\LL(\la_0)+\delta\II)u=\delta u$, so multiplying on the left by the bounded operator $(\LL(\la_0)+\delta\II)^{-1}$ and comparing with
\eqref{eq:MMdefine1} shows that also $\MM(\la_0)u=0$.  On the other hand, if $\la_0$ is such that
there exists a nonzero  vector $u\in X$ for which $\MM(\la_0)u=0$, then also $u=\delta(\LL(\la_0)+\delta\II)^{-1}u$ so that in fact $u$ lies in the domain $D\subset X$ of $\LL(\la_0)$. By acting on the left with $\LL(\la_0)+\delta\II$ one sees that $\LL(\la_0)u=0$.

To prove (b), we use mathematical induction on $m$ to show that the equations
\begin{equation}
\sum_{j=0}^{s} \frac{\LL^{(j)}(\la_0)}{j!} u^{[s-j]} = 0
\label{eq0-Richard-II}
\end{equation}
and 
\begin{equation}
\sum_{j=0}^{s} \frac{\MM^{(j)}(\la_0)}{j!} v^{[s-j]} = 0
\label{eq0b-Richard-II}
\end{equation}
are equivalent, i.e., if there is a chain of vectors $\{u^{[j]}\}_{j=0}^{m-1}$ solving the system \eq{eq0-Richard-II} for $s=0, 1, \dots, m-1$ that cannot be augmented with an additional vector $u^{[m]}$ so that  \eq{eq0-Richard-II} is satisfied for $s=m$,
then  $\{v^{[j]}=u^{[j]}\}_{j=0}^{m-1}$ is a chain of vectors solving the system \eq{eq0b-Richard-II}
that cannot be augmented with an additional vector $v^{[m]}$ so that \eq{eq0b-Richard-II} is satisfied for $s=m$, and vice-versa. 

The statement for $m=0$ follows from our proof of statement (a). Let us assume that $m \ge 1$, and that the equations \eq{eq0-Richard-II} 
and \eq{eq0b-Richard-II} are equivalent for $s = 0, 1, \dots, m-1$. 
Since $\MM(\la)=\II-\BB(\la)$, 
the equation \eq{eq0b-Richard-II} for any $s = 0, 1, \dots, m$ can be  rewritten as
\begin{equation}
-\sum_{j=1}^{s}
 \frac{\BB^{(j)}(\la_0)}{j!} v^{[s-j]} + \MM(\la_0) v^{[s]}  = 0,
\label{eq0d-Richard-II}
\end{equation}
or, equivalently,
\begin{equation}
v^{[s]} =  \sum_{j=0}^s \frac{\BB^{(j)}(\la_0)}{j!} v^{[s-j]}.
\label{eq0e-Richard-II}
\end{equation}
Apply the operator $\LL(\la_0)+\delta\II$ on \eq{eq0d-Richard-II} from the left to obtain 
\begin{equation}
-\sum_{j=1}^{s}
(\LL(\la_0)+\delta\II) \frac{\BB^{(j)}(\la_0)}{j!} v^{[s-j]} + \LL(\la_0) v^{[s]}  = 0,
\label{eq0f-Richard-II}
\end{equation}
where we used the fact that $(\LL(\la)+\delta\II) \MM (\la) = \LL(\la)$. 

Let us differentiate $j$ times, $j \ge 1$, the identity $(\LL(\la)+\delta\II)\BB(\la) = \delta\II$ on $X$ to obtain (also dividing through by $j!$ and evaluating at $\la=\la_0$)
\begin{equation}
(\LL(\la_0)+\delta\II) \frac{\BB^{(j)}(\la_0)}{j!}  =- \sum_{i=1}^{j} \frac{\LL^{(i)}(\la_0)}{i!} \cdot \frac{\BB^{(j-i)}(\la_0)}{(j-i)!},\quad j\ge 1.
\label{eq0g-Richard-II}
\end{equation}
Plugging in \eq{eq0g-Richard-II} for the derivatives of $\BB$ in \eq{eq0f-Richard-II} gives an equivalent equation
\begin{equation}
\sum_{j=1}^s  \left[\sum_{i=1}^{j} \frac{\LL^{(i)}(\la_0)}{i!} \cdot \frac{\BB^{(j-i)}(\la_0)}{(j-i)!}
\right] v^{[s-j]} + \LL(\la_0) v^{[s]} = 0.
\label{eq0h-Richard-II}
\end{equation}
By reordering terms in \eq{eq0h-Richard-II} we obtain
\begin{equation}
\sum_{i=1}^{s} \frac{\LL^{(i)}(\la_0)}{i!} 
\left[ \sum_{j=0}^{s-i} \frac{\BB^{(j)}(\la_0)}{j!} v^{[s-i-j]}\right]+ \LL(\la_0) v^{[s]} = 0.
\label{eq0i-Richard-II}
\end{equation}
According to the induction assumption the systems \eq{eq0-Richard-II} and \eq{eq0b-Richard-II} are equivalent for $s = 0, 1, \dots, m-1$. 
Set $s = m$ in \eq{eq0i-Richard-II} and use the induction assumption, i.e., \eq{eq0e-Richard-II} for $s = 0, 1, \dots, m-1$ to replace the expressions in
the square brackets in \eq{eq0i-Richard-II} to obtain
\begin{equation}
 \sum_{i=0}^{m} \frac{\LL^{(i)}(\la_0)}{i!} v^{[s-i]} = \sum_{i=1}^{m} \frac{\LL^{(i)}(\la_0)}{i!}  v^{[s-i]} + \LL(\la_0) v^{[s]} = 0.
\label{eq0j-Richard-II}
\end{equation} 
Since by the induction assumption $v^{[j]} = u^{[j]}$ for $j=0,1, \dots, m-1$, one can set $v^{[m]} = u^{[m]}$. 
The same argument yields the nonexistence of a solution of \eq{eq0b-Richard-II} for $s= m+1$ if \eq{eq0-Richard-II} does not have a solution for $s= m+1$  and vice-versa.  
Therefore, the maximal chains of characteristic vectors of $\LL$ and of $\MM$ at $\la = \la_0$ agree, i.e., 
there is one-to-one correspondence between them. 

Finally, note that the statement (c) is a direct consequence of the statements (a) and (b), given the definitions of algebraic and geometric multiplicity of characteristic values. \qed

\subsection{Proof of Proposition~\ref{prop:chains_derivatives}}\label{app2}
We first show that the case of $\LL$ having compact resolvent can be reduced to the Fredholm case.  Associated with $\LL$ we have the Fredholm operator $\MM$ defined as in \eqref{eq:MMdefine1} by choosing $\la'=\la_0$ and taking suitable real $\delta\neq 0$.  According to Proposition~\ref{theorem:algebraic-equivalence}, $\la_0$ is also a real characteristic value of $\MM$ of geometric multiplicity $k$, and the maximal chains of $\MM$ and $\LL$ for $\la_0$ are identical.
It therefore remains to relate the eigenvalue and eigenvector branches for the two pencils
$\LL$ and $\MM$.  

Suppose that $\mu=\mu(\la)$ is an analytic eigenvalue branch associated with
the analytic eigenvector branch $u=u(\la)$ for the selfadjoint spectral problem \eqref{Lmu}, and that $\mu(\la_0)=0$.  Then, keeping in mind the relation \eqref{eq:MMdefine1} between $\LL$ and
$\MM$ and the fact that both are selfadjoint for $\la_0-\epsilon<\la<\la_0+\epsilon$, a simple 
application of the functional calculus shows that $v(\la):=u(\la)$ is one of the eigenvector branches of the related spectral problem $\MM(\la)v(\la)=\varphi(\la)v(\la)$ with corresponding eigenvalue branch given by
\begin{equation}
\varphi(\la):=\frac{\mu(\la)}{\mu(\la)+\delta},\quad \la_0-\epsilon<\la<\la_0+\epsilon.
\end{equation}
From this result, it follows quickly that $\varphi(\la)$ and $\mu(\la)$ vanish to exactly the same order at $\la=\la_0$.  Since the eigenvector branches coincide, this implies that the flags $\{Y_s\}_{s=1}^\infty$ are exactly the same for the pencils $\LL$ and $\MM$.

We therefore consider $\LL(\la)$ to be a selfadjoint holomorphic pencil on $X$
of Fredholm form $\LL(\la)=\II-\BB(\la)$ with $\BB(\la)$ compact and injective.
The selfadjoint operator $\LL(\la):X\to X$ admits the spectral representation
$\LL(\la)=\UU(\la)\DD(\la)\VV(\la)$, where 
$\DD(\la):\ell_2\to\ell_2$ is bounded and diagonal:
\begin{equation}
\DD(\la)\{a_j\}_{j=1}^\infty:=\{a_j\mu_j(\la)\}_{j=1}^\infty.
\end{equation}
According to Definition~\ref{def:op_chains} we will need to calculate derivatives of $\LL(\la)$, and these are given by the formula
\begin{equation}
\frac{1}{n!}\LL^{(n)}(\la)=\sum_{n_1=0}^n\sum_{n_2=0}^{n-n_1}
\frac{U^{(n-n_1-n_2)}(\la)M^{(n_2)}(\la)V^{(n_1)}(\la)}{(n-n_1-n_2)!n_2!n_1!},
\quad
\la\in S\cap\RR,
\label{eq:pencilderivs}
\end{equation}
where the derivatives of $\UU(\la)$ are defined by \eqref{eq:UUderivs}, and those of $\DD(\la)$ and $\VV(\la)$ are given by
\begin{equation}
\DD^{(n)}(\la)\{a_j\}_{j=1}^\infty:=\left\{a_j\mu_j^{(n)}(\la)\right\}_{j=1}^\infty, \quad
\VV^{(n)}(\la)w:=\left\{\left(w,u_j^{(n)}(\la)\right)\right\}_{j=1}^\infty.
\label{eq:operatorderivs}
\end{equation}
For the special value of $\la=\la_0$ we simplify the notation:
$\UU_0^{(n)}:=\UU^{(n)}(\la_0)$, $\DD_0^{(n)}:=\DD^{(n)}(\la_0)$, and $\VV_0^{(n)}:=\VV^{(n)}(\la_0)$, 
and we omit superscripts of ``$(0)$'' altogether.%
In general, the definitions \eqref{eq:UUderivs} and \eqref{eq:operatorderivs} have to be regarded in a formal sense without
further conditions on the behavior of the eigenvalue and eigenvector derivative sequences 
$\{\mu_j^{(n)}(\la)\}_{j=1}^\infty$ and $\{u_j^{(n)}(\la)\}_{j=1}^\infty$. 
However, since we have reduced the problem to the Fredholm case, the derivatives $\LL^{(n)}$
are all defined and bounded on the whole Hilbert space $X$ for every $\la\in (\la_0-\epsilon,\la_0+\epsilon)$, and this implies that the combinations appearing in \eqref{eq:pencilderivs} as well as those resulting from other formal manipulations to follow shortly also make sense on the whole space.  This would be a substantially more difficult issue were $\LL$ unbounded.

We prove the proposition in the Fredholm case by induction on $m$.  The case $m=1$ is obvious, so we assume the inductive hypothesis holds for $m=m_0$ 
and prove it for $m=m_0+1$.  For there to exist a chain of length $m=m_0+1$, according to Definition~\ref{def:op_chains} we 
need to find a nonzero vector $u^{[0]}\in Y_1=\Ker(\LL(\la_0))$  and vectors $u^{[1]},\dots,u^{[m_0]}$ in $X$ such that for 
each $q=0,1,\dots,m_0$,
\begin{equation}
\sum_{n=0}^q\frac{1}{n!}\mathcal{L}^{(n)}(\lambda_0)u^{[q-n]}=0.
\end{equation}
According to \eqref{eq:pencilderivs} this can be written in the form
\begin{equation}
\sum_{n=0}^q\sum_{n_1=0}^n\sum_{n_2=0}^{n-n_1}
\frac{\UU_0^{(n-n_1-n_2)}\DD_0^{(n_2)}\VV_0^{(n_1)}u^{[q-n]}}{(n-n_1-n_2)!n_2!n_1!}=0,\quad q=0,1,2,\dots,m_0.
\label{e01}
\end{equation}
By the inductive hypothesis,  \eq{e01} holds true for $q=0,1,2,\dots,m_0-1$ if and only if
\begin{equation}
u^{[r]}=\sum_{d=0}^r\frac{1}{d!}\UU_0^{(d)}\VV_0 w^{[r-d]},\quad r=0,1,2,\dots,m_0-1,\quad\text{where $w^{[s]}\in Y_{m_0-s}$},
\end{equation}
so it remains to set $q=m_0$ and attempt to solve for $u^{[m_0]}$:
\begin{equation}
\UU_0\DD_0\VV_0 u^{[m_0]}=-\sum_{n=1}^{m_0}\sum_{n_1=0}^n\sum_{n_2=0}^{n-n_1}\sum_{d=0}^{m_0-n}
\frac{\UU_0^{(n-n_1-n_2)}\DD_0^{(n_2)}\VV_0^{(n_1)}\UU_0^{(d)}\VV_0 w^{[m_0-n-d]}}{(n-n_1-n_2)!n_2!n_1!d!}.
\label{eq:ForcingTerms}
\end{equation}
All derivatives $\VV_0^{(n_1)}$ for $n_1>0$ of $\VV(\la)$ at $\la=\la_0$ can be eliminated in favor 
of $\VV_0$ and derivatives of $\UU(\la)$ at $\la=\la_0$ by repeated  differentiation of the identities
$\UU(\la)\VV(\la)=\mathbb{I}_X$ and $\VV(\la)\UU(\la)=\mathbb{I}_{\ell_2}$.  This leads
to dramatic cancellations in the four-fold sum appearing on the right-hand side; some straightforward but lengthy calculations 
show that in fact  \eqref{eq:ForcingTerms} can be rewritten in the much simpler form
\begin{equation}
\UU_0\DD_0\VV_0 u^{[m_0]}=\sum_{s=1}^{m_0}\frac{1}{s!}\UU_0\DD_0\VV_0\UU_0^{(s)}\VV_0 w^{[m_0-s]} -\sum_{s=1}^{m_0}\sum_{b=0}^s
\frac{\UU_0^{(s-b)}\DD_0^{(b)}\VV_0 w^{[m_0-s]}}{b!(s-b)!}.
\label{eq:UMeqn}
\end{equation}
At this point we can use the maps $\UU_0$ and $\VV_0$ to shift from an equation on $X$ to an equivalent equation 
on $\ell_2$.  We therefore write
\begin{equation}
\tilde{u}^{[m_0]}:=\VV_0 u^{[m_0]}\in\ell_2\quad\text{and}\quad
\tilde{w}^{[j]}:=\VV_0 w^{[j]}\in\ell_2,\quad j=0,1,2,\dots,m_0-1,
\end{equation}
and applying the operator $\VV_0$ on the left of \eqref{eq:UMeqn} yields
\begin{equation}
\DD_0\tilde{u}^{[m_0]}=\sum_{s=1}^{m_0}\frac{1}{s!}\DD_0\VV_0\UU_0^{(s)}\tilde{w}^{[m_0-s]} -\sum_{s=1}^{m_0}\sum_{b=0}^s
\frac{\VV_0\UU_0^{(s-b)}\DD_0^{(b)}\tilde{w}^{[m_0-s]}}{b!(s-b)!}.
\label{eq:UMeqnEll2}
\end{equation}

We now invoke the inductive hypothesis that $w^{[m_0-s]}\in Y_s$ for $s=1,\dots,m_0$.
Since $\DD_0^{(b)}$ is the diagonal multiplication operator on $\ell_2$ with diagonal entries $\mu_j^{(b)}(\la_0)$,
 it is easy to see that, by definition of the flag of subspaces $Y_j$,
we have $\DD_0^{(b)}\tilde{w}^{[m_0-s]}=0$ for $b=0,1,2,\dots,s-1$.  The inductive hypothesis therefore implies that the 
right-hand side of \eqref{eq:UMeqnEll2} can be simplified yet further:
\begin{equation}
\DD_0\tilde{u}^{[m_0]}=\sum_{s=1}^{m_0}\frac{1}{s!}\DD_0\VV_0\UU_0^{(s)}\tilde{w}^{[m_0-s]}-\sum_{s=1}^{m_0}
\frac{1}{s!}\DD_0^{(s)}\tilde{w}^{[m_0-s]}
\label{eq:UMeqnFinal}
\end{equation}
where we have used the identity $\VV_0\UU_0=\mathbb{I}_{\ell_2}$.

Now consider the solvability of \eqref{eq:UMeqnFinal} for $\tilde{u}^{[m_0]}$. 
It is obvious that the first sum is in the range of $\DD_0$, so the solvability condition that we require is that
\begin{equation}
\sum_{s=1}^{m_0}\frac{1}{s!}\DD_0^{(s)}\tilde{w}^{[m_0-s]}\in\Ran(\DD_0)=\Ker(\DD_0)^\perp.
\label{eq:solvability}
\end{equation} 
On the other hand, it follows in particular from the inductive hypothesis and the nesting of the subspaces making up the flag
$\{Y_s\}_{s=1}^\infty$ that $w^{[j]}\in Y_1=\Ker(\UU_0\DD_0\VV_0)$ for $j=0,\dots,m_0-1$, a fact which implies in turn that
$\tilde{w}^{[j]}\in\Ker(\DD_0)$ for $j=0,\dots,m_0-1$.  Since $\DD_0^{(s)}$ commutes with $\DD_0$ (both being diagonal) 
we therefore see that the solvability condition takes the form
\begin{equation}
\sum_{s=1}^{m_0}\frac{1}{s!}\DD_0^{(s)}\tilde{w}^{[m_0-s]}=0.
\label{eq:solvabilityII}
\end{equation}
Now, because $w^{[0]}\in Y_{m_0}$ by the inductive hypothesis, and since the operator $\DD_0^{(m_0)}$ is diagonal, 
the vector $\DD_0^{(m_0)}\tilde{w}^{[0]}$ can only have a nonzero entry if the corresponding diagonal entry of the 
operator $\DD_0^{(s)}$ vanishes for $s=0,1,2,\dots,m_0-1$.  This shows that the solvability condition \eqref{eq:solvabilityII} 
actually splits into two independent conditions:
\begin{equation}
\DD_0^{(m_0)}\tilde{w}^{[0]}=0\quad\text{and}\quad\sum_{s=1}^{m_0-1}\frac{1}{s!}\DD_0^{(s)}\tilde{w}^{[m_0-s]}=0.
\end{equation}
By a subordinate inductive argument in which one considers the term in the sum with the largest value of $s$ and shows that it must vanish
independently of the terms in the sum with smaller values of $s$ due to the standing inductive hypothesis that $w^{[m_0-s]}\in Y_s$,
one shows that the solvability condition \eqref{eq:solvabilityII} is actually equivalent to $m_0$ independent conditions:
\begin{equation}
\DD_0^{(s)}\tilde{w}^{[m_0-s]}=0,\quad s=1,\dots,m_0.
\end{equation}
Combining these conditions with the inductive hypothesis proves that
the solvability condition is satisfied if and only if we impose on $w^{[j]}$, $j=0,\dots,m_0-1$, the additional condition that
\begin{equation}
w^{[m_0-s]}\in Y_{s+1},\quad s=1,2,\dots,m_0.
\end{equation}
(For the chain to exist it is necessary that $w^{[0]}\neq 0$ and hence the subspace $Y_{m_0+1}$ must be nontrivial.)
These additional conditions obviously reduce \eqref{eq:UMeqnFinal} to the form
\begin{multline}
\DD_0\tilde{u}^{[m_0]}=\DD_0\sum_{s=1}^{m_0}\frac{1}{s!}\VV_0\UU_0^{(s)}\tilde{w}^{[m_0-s]}\implies\\
\tilde{u}^{[m_0]}=\sum_{s=1}^{m_0}\frac{1}{s!}\VV_0\UU_0^{(s)}\tilde{w}^{[m_0-s]}\pmod{\Ker(\DD_0)}.
\end{multline}
Since $u^{[m_0]}=\UU_0\tilde{u}^{[m_0]}$ and $\UU_0$ is an isomorphism from $\Ker(\DD_0)$ onto $\Ker(\UU_0\DD_0\VV_0)=Y_1$, we obtain
\begin{equation}
u^{[m_0]}=\sum_{s=1}^{m_0}\frac{1}{s!}\UU_0^{(s)}\VV_0 w^{[m_0-s]} + w^{[m_0]}=\sum_{s=0}^{m_0}
\frac{1}{s!}\UU_0^{(s)}\VV_0 w^{[m_0-s]},\quad w^{[m_0]}\in Y_1,
\end{equation}
which completes the induction step by showing that the statement holds for $m=m_0+1$. \qed

\end{document}